\documentclass[12pt,oneside]{amsart}

\usepackage[inline]{showlabels}

\title{Toposes have an optimal noetherian form}
\author{Zurab Janelidze \and Francois van Niekerk}
\address{Department of Mathematical Sciences\\ Stellenbosch University, South Africa\\
and\\
National Institute for Theoretical and Computational Sciences (NITheCS), South Africa}
\email{zurab@sun.ac.za}

\address{Department of Mathematical Sciences\\ Stellenbosch University, South Africa\\
and\\
National Institute for Theoretical and Computational Sciences (NITheCS), South Africa}
\email{fkvn@sun.ac.za}

\makeatletter
\@namedef{subjclassname@2020}{
  \textup{2020} Mathematics Subject Classification}
\makeatother


\usepackage{amssymb, thmtools, color, enumitem, hyperref, graphicx}
\usepackage{tasks}
\usepackage[all,cmtip]{xy}
\hypersetup{pdfborder=0 0 0}

\usepackage{geometry}
\setlist{leftmargin=1cm}

\declaretheoremstyle[
  headformat = \textcolor{red}{\NUMBER. }\NAME{}\NOTE,
  headindent=0.8cm
]{actstyle}
\theoremstyle{actstyle}
  \newtheorem{theorem}{Theorem}
  \newtheorem{lemma}[theorem]{Lemma}

  \newtheorem{remark}[theorem]{Remark}
  \newtheorem{example}[theorem]{Example}
  
  \newtheorem{definition}[theorem]{Definition}
  \newtheorem{convention}[theorem]{Convention}

\newcommand{\im}{\operatorname{Im}}
\renewcommand{\ker}{\operatorname{Ker}}
\newcommand{\s}{\mathsf{s}}
\newcommand{\e}{\mathsf{e}}
\renewcommand{\c}{\mathsf{c}}
\newcommand{\n}{\mathsf{n}}
\newcommand{\op}{\mathsf{op}}
\newcommand{\ccl}[1]{\overline{#1}^\c}
\newcommand{\cccl}[1]{\underline{#1}_\c}
\newcommand{\ncl}[1]{\overline{#1}^\n}
\newcommand{\nccl}[1]{\underline{#1}_\n}
\renewcommand{\le}{\leqslant}
\renewcommand{\ge}{\geqslant}
\renewcommand{\ast}{\mathsf{w}}
\newcommand{\astop}{\rotatebox[origin=c]{180}{$\mathsf{w}$}}

\begin{document}

\maketitle

\vspace{-\baselineskip}

\begin{abstract} A noetherian form is an abstract self-dual framework suitable for establishing homomorphism theorems (such as the isomorphism theorems and homological diagram lemmas) for group-like structures. In this paper we identify and carry out an axiomatic analysis of a particular class of noetherian forms which exist for both group-like structures and for sheaves. More abstractly, such noetherian forms can be produced from all semi-abelian categories, Grandis exact categories and toposes.

\end{abstract}

\makeatletter
  \def\l@subsection{\@tocline{2}{0pt}{4pc}{5pc}{}}
\makeatother

\section*{Introduction}

Noetherian forms were introduced in \cite{GJ19} as means for systematizing a self-dual approach to homomorphism theorems for groups and other group-like structures. This work builds on earlier developments (see e.g., \cite{J14,JW14,JW16,JW16b}) and takes its roots from \cite{M50}, where S.~Mac~Lane pioneered the study of such a context in the case of abelian groups (which was a precursor of the present-day concept of an abelian category).  Noetherian forms resemble doctrines in the sense of categorical logic (see e.g., \cite{doctrines1,doctrines2,doctrines3,doctrines4,doctrines5,doctrines6}). Loosely speaking, a `form' (as defined in \cite{J14,JW14}) is a category equipped with an abstractly given data of subobjects, quotients, or some other data (poset) of `clusters' associated to every given object. Such data can be formalised as a faithful functor between two categories, whose fibres are posets --- the intended posets of clusters. Doctrines are forms in this sense too, and in fact, not only doctrines but also many other categorical structures fit into such framework: monadic functors (fibres are discrete in this case), concrete categories and topological functors (where fibres are complete lattices). With noetherian forms, the idea is to view the fibres of the form as abstractions of, specifically, posets of subgroups of groups, or more generally, posets of substructures of some group-like structures. General monadic, forgetful and topological functors do not fit into this narrower intuitive framework. Doctrines come close to it, since the fibres there are often posets of subobjects in a suitable category, such as a topos. However, these subobject doctrines/forms are not noetherian, at least not for toposes. The reason for this is duality: axioms of a noetherian form are self-dual in the sense that they hold for a functor $F\colon \mathbb{C}\to\mathbb{D}$ just as well as they hold for the opposite functor between opposite categories, $F^\mathsf{op}\colon \mathbb{C}^\mathsf{op}\to\mathbb{D}^\mathsf{op}$. Already in the case of the category of sets, the subobject doctrine/form does not have a rich supply of self-dual properties. The situation is different for the subgroup form of groups, or for substructure forms of group-like structures in general, as shown in the references mentioned above. It turns out, nevertheless, that the category of sets (and more generally, any topos), does admit a form, different from the subobject form, which is not only noetherian, but it also has additional properties that are shared by (duals of) substructure forms of not only group-like structures, but also all major examples of noetherian forms: all Grandis exact categories \cite{G84,G92,G12,G13} and all Janelidze-M\'arki-Tholen semi-abelian categories \cite{JMT02}. Analysis of these properties is the subject of the present paper. They deal with decomposing an arbitrary cluster into a `normal' and a `conormal' component: 
essentially, these properties assert that clusters can be represented by a particular type of subquotients.

The notions of a Grandis exact category \cite{G84,G92,G12,G13} and of a Janelidze-M\'arki-Tholen semi-abelian category \cite{JMT02} are two different generalizations of the notion of an abelian category \cite{B55,G59}, where essential homomorphism theorems valid in an abelian category can be recovered. Recall that a semi-abelian category is the same as a pointed category with finite sums (coproducts) that is simultaneously Barr exact \cite{Bar71} and protomodular \cite{Bou91}. The notion of a Grandis exact category removes limits, colimits and pointedness (i.e., existence of a zero object) from the notion of an abelian category, but keeps factorization of morphisms into a normal epimorphism followed by a normal monomorphism (which, without pointedness, are defined relative to an ideal of null morphisms \cite{Ehr64,Kel64,Lav65}). On the other hand, the notion of a semi-abelian category, instead, weakens the factorization axiom, but keeps limits/colimits and pointedness, imposing some exactness properties expressed using them (which would have been consequences of the factorization axiom), gentle enough to allow coverage of categories of (non-abelian) group-like structures as examples (these, in general, are not examples of Grandis exact categories). Noetherian forms provide, in some sense, a unification of Grandis exact categories and semi-abelian categories: both of these types of categories are instances of neotherian forms. Consequently, noetherian forms include: module categories, the category of groups, of graded abelian groups, the categories of Lie algebras, of cocommutative Hopf algebras, the category of Heyting semilattices, of loops, the category of modular/distributive lattices and modular connections, and many others. The results of this paper show that we can add to this list duals of toposes. In fact, we show that a topos has, up to isomorphism, a unique noetherian form satisfying the additional axioms introduced in this paper, and moreover, that this noetherian is a subform of any other noetherian form over the topos: it is `optimal', in a suitable sense. Furthermore, we establish that the standard noetherian forms of semi-abelian and Grandis exact categories also satisfy these additional axioms, that they can be characterized using them, and that they are optimal too in a similar sense.

Every noetherian form gives rise to a proper factorization system on the ground category. Our additional axioms on a noetherian form ensure that the form is uniquely determined by the factorization system. The reason for uniqueness of such form for toposes is simply that a topos has exactly one proper factorization system (which is due to the fact that it is a balanced category). At the end of the paper we describe those proper factorization systems which give rise to a noetherian form satisfying the (duals of) new additional axioms. Applying this to the standard factorization system of a regular category, we obtain a new class of categories that are very close to semi-abelian categories, include all of them, and at the same time include duals of toposes.

Sections 1-4 prepares the groundwork: it develops various nuances of the theory of forms required for the rest of the paper (and also serves as a comprehensive introduction to the subject of forms). Section 5 is where the new axioms are introduced and analyzed. Section 6 is devoted to describing the corresponding factorization systems. The concluding Section 7 contains a more technical summary of the contents of the paper, which can be read before reading the rest of the paper.

\tableofcontents

\section{Forms}

\subsection*{From cluster systems to forms} We begin the paper with a merger of two alternative approaches in defining the main concept that lies at the basis of our work. The first of these (`cluster system') is useful when considering examples, while the second (`form') will be  useful for developing the theory.  

\begin{definition}
A \emph{cluster system} over a category $\mathbb{C}$ specifies for each object $X$ of $\mathbb{C}$ a set, called the set of \emph{clusters} in $X$, and for each morphism $f\colon X\to Y$ in $\mathbb{C}$ a relation $\ge_f$ between the set of clusters in $X$ and the set of clusters in $Y$ (which we write as $\ge_X$ when $f=1_X$, or simply $\ge$, when this does not cause a confusion), such that the following conditions hold:
\begin{itemize}
  \item[(F1)] For any cluster $S$ in $X$, we have $S\ge S$.

  \item[(F2)] The implication 
    \[\exists_{B\subseteq_F Y}[C\ge_g B\ge_f A]\implies[C\ge_{gf} A]\]
  holds for any sequence
    \[\xymatrix{X\ar[r]^-{f} & Y\ar[r]^-{g} & Z}\]
  of morphisms and clusters $A$ and $C$ in $X$ and $Z$, respectively.

  \item[(F3)] For any two clusters $S$ and $T$ in $X$, if $S\ge T$ and $T\ge S$, then $S=T$.
\end{itemize}
Note that we may sometimes write $A\le_f B$ for  $B\ge_f A$ (or $A\le B$ for $B\ge A$).
\end{definition}

\begin{example}\label{ExaU}
A cluster system over a single-morphism category $\mathbf{1}$ is essentially the same as a partially ordered set. The clusters are the elements of the partially ordered set and the partial order is given by the relation $\ge_{X}$, where $X$ is the unique object of $\mathbb{C}$. Axiom (F1) states reflexivity of the relation, (F2) states transitivity and (F3) states anti-symmetry. 
\end{example}

\begin{example}\label{ExaV}
  A monotone map $f\colon X\to Y$ between posets $X$ and $Y$ can be viewed as a cluster system over a category $\mathbf{2}$ with three morphisms and exactly one non-identity morphism $f\colon 1\to 2$. The clusters in $1$ and $2$ are their elements, and $b\ge_f a$ if and only if $f(a)\le b$ in the poset $Y$. The relations $\ge_{1}$ and $\ge_{2}$ are the partial orders of the posets $X$ and $Y$, respectively. More generally, a cluster system over $\mathbf{2}$ is given by two posets $P_1$ and $P_2$ and a binary relation $R$ between them satisfying $[a\le bRc\le d]\implies [aRd]$.
\end{example}

\begin{example}
  There is a cluster system implicit in the variant of sequent calculus discussed in the appendix of \cite{LR03}. The objects $X$ of the base category of this cluster system are universes where entities about which we make statements reside, while clusters in a universe $X$ are the statements about entities in that universe (in other words, the formulas). Morphisms in the base category are change of universe arrows, while $B\ge_f A$ if and only if $A\vdash Bf$, in the notation of \cite{LR03}.
\end{example}

\begin{example}\label{ExaT}
Single-sorted set-based mathematical structures of a given type usually give rise to cluster systems over the category $\mathbf{Set}$ of sets. For a given type of structures, clusters in a set $X$ are structures on $X$ of that type, while for a function $f\colon X\to Y$, we have $B\ge_f A$ if and only if $f$ defines a homomorphism $(X,A)\to (Y,B)$. For instance, when `structure' means `topology', clusters in $X$ are topologies on $X$. Then, for a function $f\colon X\to Y$ and topologies $\tau$ on $X$ and $\sigma$ on $Y$, we have $\sigma\ge_f \tau$ if and only if $f$ is a continuous function $f\colon(X,\tau)\to (Y,\sigma)$. 
\end{example}

\begin{example}\label{ExaS}
Any amnestic concrete category $(\mathbb{B},U)$ over a base category $\mathbb{C}$ (in the sense of \cite{AHS}) gives rise to a cluster system over $\mathbb{C}$, where clusters in an object $X$ are objects $B$ of $\mathbb{B}$ such that $U(B)=X$, while $B\ge_f A$ if and only if there is a morphism $A\to B$ in $\mathbb{B}$ mapped to $f$ by $U$. This example is in fact essentially a generalization of the previous one, as it will soon become apparent.
\end{example}

\begin{definition}
A \emph{form} over a category $\mathbb{C}$ is a faithful amnestic functor $F\colon\mathbb{B}\to\mathbb{C}$. 
\end{definition}

\begin{remark} A form is essentially another name for an amnestic concrete category \cite{AHS}. There is, however, a philosophical difference between the two notions. In a concrete category, it is the domain of the functor that we want to study. In a form, it is the codomain, similar to Grothendieck's original motivation for the notion of a fibration \cite{G57}. Furthermore, we will soon be considering axioms on a form that are hardly ever fulfilled by standard concrete categories.
Note that according to Example~\ref{ExaS}, any form gives rise to a cluster system.   
\end{remark}

\begin{remark} 
Recall that a functor is amnestic when the only isomorphisms mapped to the identity morphisms by the functor are the identity morphisms. Together with faithfulness, this requirement is equivalent to the requirement that the fibres of the functor are posets, rather than merely preorders, as guaranteed by faithfulness. 
\end{remark}

\begin{definition} Consider a form $F\colon\mathbb{B}\to\mathbb{C}$.
\begin{itemize}
  \item for each object $X$ of $\mathbb{C}$, the set of all objects of $\mathbb{B}$, which by the functor $F$ are mapped to $X$, is denoted by $F^{-1}(X)$ --- these objects are called \emph{clusters} in $X$ for the form $F$ (or simply, $F$-clusters in $X$) and we write $B\subseteq_F X$ when $B$ is a cluster in $X$,  
    
  \item for each morphism $f\colon X\to Y$ in $\mathbb{C}$, we define a relation $F^{-1}(Y)\to F^{-1}(X)$, which we denote by $\ge_f^F$, and by $\ge_X^F$ when $f=1_X$ (omitting $F$ in the superscript when convenient), as follows: $B\ge_f^F A$ if and only if there a morphism $A\to B$ in $\mathbb{B}$ mapped to $f$ by $F$.
\end{itemize} 
A \emph{presentation} of a form $F\colon \mathbb{B}\to\mathbb{C}$ is a cluster system over $\mathbb{C}$ such that for each object $X$ there is a bijection $\zeta_X$ from the set of clusters in $X$ (of the cluster system) to the set of $F$-clusters, such that \[B\ge_f A\quad\iff\quad \zeta_Y(B)\ge^{F}_f \zeta_X(A).\]
When $\zeta_X$ is an identity function, we call the presentation of $F$ a \emph{canonical presentation}.
\end{definition}

\begin{lemma}\label{LemO} Every form has a (canonical) presentation. 
Moreover, every cluster system is a presentation of a form.
\end{lemma}

\begin{proof} Let $F\colon \mathbb{B}\to\mathbb{C}$ be a form. To show that $F$ has a canonical presentation, we must check that $F$-clusters and the relations $\ge^F_f$ have the properties (F1-3). The property (F1) comes out of the presence of identity morphisms in $\mathbb{B}$ and their preservation by $F$. (F2) --- from composition of morphisms in $\mathbb{B}$ and $F$ preserving composition. (F3) is a consequence of amnesticity of $F$.

Next, for a given cluster system over $\mathbb{C}$, build a form $F\colon\mathbb{B}\to\mathbb{C}$ as follows. Objects of $\mathbb{B}$ are pairs $(X,A)$, where $X$ is an object in $\mathbb{C}$ and $A$ is a cluster in $X$. A morphism $f\colon (X,A)\to (Y,B)$ is given by a morphism $f\colon X\to Y$ in $\mathbb{C}$ such that $B\ge_f A$. Composition of morphisms in $\mathbb{B}$ is defined as in $\mathbb{C}$, which is made possible by (F2). Then the composition will necessarily be associative. Thanks to (F1), $\mathbb{B}$ has identity morphisms. Mapping $f\colon (X,A)\to (Y,B)$ to $f\colon X\to Y$ defines a functor $F\colon\mathbb{B}\to\mathbb{C}$, by construction of $\mathbb{B}$. The original cluster system is a presentation of this form, where $\zeta_X(A)=(X,A)$. 
\end{proof}

\begin{example}
The cluster system described in Example~\ref{ExaT}, which was obtained from mathematical structures of a given type, is a presentation of the forgetful functor $U$ of the concrete category $(\mathbb{B},U)$ of those mathematical structures. The canonical presentation of the same concrete category $(\mathbb{B},U)$ is the one discussed in Example~\ref{ExaS}. For example, in the case of topological spaces, in the former presentation clusters in a set $X$ are topologies $\tau$ on $X$, while clusters for the canonical presentation are topological spaces $(X,\tau)$.  
\end{example}

\subsection*{Isomorphism and duality} The notions of `isomorphism of forms' and `dual of a form' are introduced here and illustrated with some applications and examples. 

\begin{definition}\label{DefC}
Two forms $F_1\colon \mathbb{B}_1\to\mathbb{C}$ and $F_2\colon \mathbb{B}_2\to\mathbb{C}$ over the same category are said to be \emph{isomorphic} when $F_1=F_2\circ \varphi$ for some isomorphism $\varphi$ of categories (and such $\varphi$ is called an \emph{isomorphism of forms}):
  \[\xymatrix{ \mathbb{B}_1\ar[rr]^-{\varphi}\ar[dr]_-{F_1} & & \mathbb{B}_2\ar[dl]^-{F_2} \\ & \mathbb{C} & }\]
\end{definition}

\begin{lemma}
Two forms over the same base category are isomorphic if and only if they have a common presentation.  
\end{lemma}

\begin{proof}
Let $F_1$ and $F_2$ be forms over a category $\mathbb{C}$. If they are isomorphic, it is easy to see that the canonical presentation of $F_1$ will be a presentation for the other form. The bijections $\zeta_X$ that describe the presentation will be given by $\zeta_X(A)=\varphi(A)$. Conversely, suppose that the two forms have a common presentation and consider the bijections $\zeta_{1,X}$ and $\zeta_{2,X}$ for the presentations of $F_1$ and $F_2$, respectively, given by the same cluster system. Then an isomorphism $\varphi$ of forms  can be created uniquely by setting $\varphi(A)=\zeta_{2,F_1(A)}(\zeta^{-1}_{1,F_1(A)}(A)).$
\end{proof}

\begin{convention}
We often describe a form up to an isomorphism, by describing one of its presentations. Furthermore, in later sections, we introduce notions dealing with forms and describe them in terms of the canonical presentation of forms. However, in each such case a canonical presentation can be traded with an arbitrary one, and in examples, we may switch between the two without a warning. 
\end{convention}

\begin{example}\label{ExaG}
  Any category $\mathbb{C}$ gives rise to a form $F$ as follows: define clusters in an object $X$ to be ordinary subobjects of $X$, i.e., equivalence classes of monomorphisms with codomain $X$, under the equivalence relation
    \[m\sim m'\quad\iff\quad \exists_{u,v}[[m'u=m]\wedge [mv=m']].\]
  Define $[t]\ge_f [s]$ to mean
    \[[t]\ge_f [s]\quad\iff\quad \exists_{u}[tu=fs].\]
  See \cite{JW14} for an elaboration of this example. We call a form obtained in this way the \emph{form of subobjects} of the category $\mathbb{C}$. In the case when $\mathbb{C}$ is the category of groups, for instance, the form of subobjects is isomorphic to the \emph{form of subgroups} of groups: a form where clusters in a group $X$ are its subgroups and $B\ge_f A$ means $f(A)\subseteq B$, or equivalently, $A\subseteq f^{-1}(B)$. More generally, the \emph{form of subalgebras} of universal algebras in a variety (defined similarly to the form of subgroups) is isomorphic to the form of subobjects for the variety viewed as a category. When the variety is that of sets, we get the \emph{form of subsets}.
\end{example}

\begin{definition}The \emph{dual} of a form $F\colon\mathbb{B}\to\mathbb{C}$ is the dual functor $F^\op\colon\mathbb{B}^\op\to\mathbb{C}^\op$.
\end{definition}

\begin{remark}
The dual of a form is again a form. In terms of a representation of a form, duality reverses all morphisms (and hence the order in which morphisms are composed) in the base category, and reverses the relations `$\ge_f$'. Note that the axioms (F1-3) of a cluster system are indeed self-dual for this notion of duality. In this paper we will consider further axioms on a form that are self-dual.
\end{remark}

\begin{example}\label{ExaH}
  A distinguished class $\mathcal{M}$ of monomorphisms in a category gives rise to a form, called the \emph{form of $\mathcal{M}$-subobjects}, similarly to how a form was constructed in Example~\ref{ExaG}: all monomorphisms in the construction of the form are now required to belong to the class $\mathcal{M}$. For various non-algebraic mathematical structures, their `forms of substructures' arise this way for a subclass $\mathcal{M}$ of the class of all monomorphisms. For instance, choosing $\mathcal{M}$ to be the class of regular monomorphisms in the category of topological spaces, we get a form isomorphic to the \emph{form of subspaces}, where clusters in a topological space are its subspaces. The dual construction gives us \emph{forms of $\mathcal{E}$-quotients}, for classes $\mathcal{E}$ of epimorphisms. These forms capture `forms of quotient structures' of mathematical structures. For instance, choosing $\mathcal{E}$ to be the class of regular epimorphisms in an algebraic category, we get the \emph{form of quotient algebras}, which is isomorphic to the \emph{form of congruences}, where clusters in an algebra are congruences of the algebra (in the case of the category of sets, this is the \emph{form of equivalence relations}). In this form, for a homomorphism $f\colon X\to Y$ of algebras, and congruences $A$ on $X$ and $B$ on $Y$, we have $B\ge_f A$ if and only if 
    \[x A y\quad\implies\quad f(x) B f(y)\]
for all $x,y\in X$.
\end{example} 

\begin{example}\label{ExaAC}
The form of $\mathcal{M}$-subobjects and the form of $\mathcal{E}$-quotients from the previous example can unified as follows. Let $\mathcal{M}$ be a class of monomorphisms and let $\mathcal{E}$ be a class of epimorphisms in the same category $\mathbb{C}$. Consider the following equivalence relation on the set of pairs $(e,m)$ where $e$ and $m$ have the same domain:
\[(e,m)\sim (e',m')\quad\iff\quad \exists_{s,t,u,v}[[m'u=m]\land [mv=m']\land [e'u=se]\land [ev=te']]\]
Let us write an equivalence class of a pair $(e,m)$ under this equivalence relation $[e,m]$. We call it an \emph{$(\mathcal{E},\mathcal{M})$-subquotient}. In the \emph{form of $(\mathcal{E},\mathcal{M})$-subquotients}, clusters in an object $X$ are given by $(\mathcal{E},\mathcal{M})$-subquotients $[e,m]$ where the codomain of $m$ is $X$. For a morphism $f\colon X\to Y$ we have:
\[[e',m']\ge_f [e,m]\quad\iff\quad \exists_{s,u}[[m'u=mf]\land [e'u=se]]\]
Now, it is not difficult to see that if $\mathcal{M}$ is the class of identity morphisms, then the form of $(\mathcal{E},\mathcal{M})$-subquotients is isomorphic to the form of $\mathcal{E}$-quotients, while if $\mathcal{E}$ is the class of identity morphisms, then the form of $(\mathcal{E},\mathcal{M})$-subquotients is isomorphic to the form of $\mathcal{M}$-subobjects.      
\end{example}

\subsection*{Forms vs monoidal categories} We end off the section by illustrating a striking parallelism between the notion of a form and the notion of a monoidal category, which gives one a framework for understanding the purpose of the theory of forms. 

A form is an example of a `structured category' --- a category endowed with additional structure. Among the most common structured categories are monoidal categories (see \cite{M98} for a background on monoidal categories). There is a certain analogy between these two types of structured categories, exhibited by Figure~\ref{figA}, which we explain below (following Figure~\ref{figA} from top to bottom):
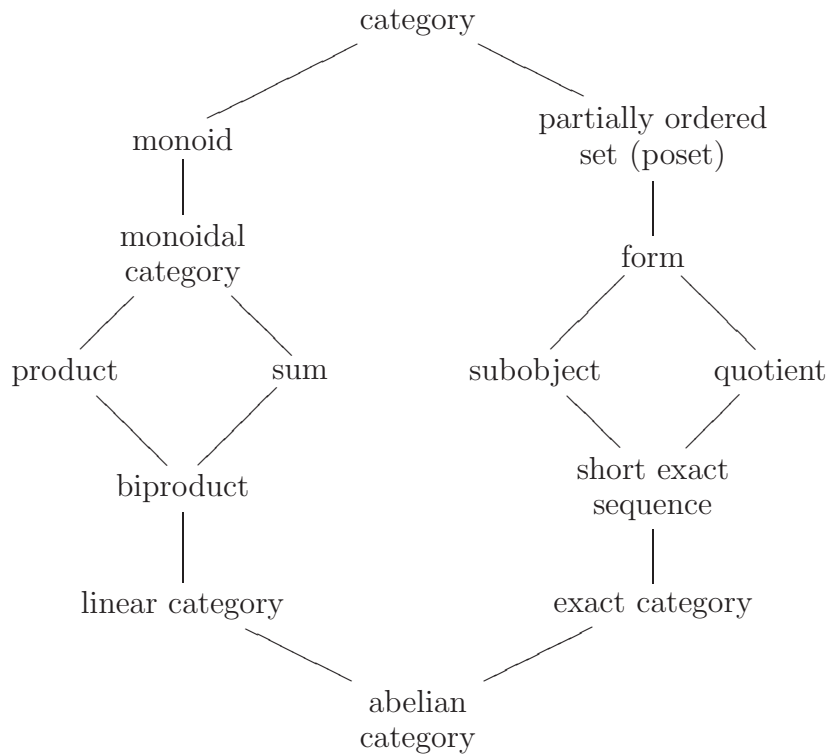
\begin{figure}
  \label{figA}
    \[\xymatrix@!=20pt{ & & & \txt{category}\ar@{-}[lld]\ar@{-}[rrd] & & & \\ & \txt{monoid}\ar@{-}[d] & & & & \txt{partially ordered\\ set (poset)}\ar@{-}[d] & \\ & \txt{monoidal\\ category}\ar@{-}[ld]\ar@{-}[rd] & & & & \txt{form}\ar@{-}[ld]\ar@{-}[rd] & \\ \txt{product}\ar@{-}[rd] & & \txt{sum}\ar@{-}[ld] & & \txt{subobject}\ar@{-}[rd] & & \txt{quotient}\ar@{-}[ld] \\ & \txt{biproduct}\ar@{-}[d] & & & & \txt{short exact\\ sequence}\ar@{-}[d] & \\ & \txt{linear category}\ar@{-}[rrd] & & & & \txt{exact category}\ar@{-}[lld] & \\ & & & \txt{abelian\\ category} & & & }\]
  \caption{Analogy between forms and monoidal category structures.}
\end{figure}
\begin{itemize}
  \item Monoids and posets can be seen as two extreme types categories: a monoid is a category with few objects --- exactly one object, to be precise, whereas a poset is a category with few morphisms --- between any two objects there is at most one morphism (in either direction). If a category is seen as a mixture of algebraic structure (composition of morphisms) and geometric structure (objects), then in some sense a monoid is a category whose geometric features are trivialized, and a poset is a category whose algebraic features are trivialized. 

  \item Now, a category can be equipped with additional structure. It can be equipped with further algebraic structure --- a monoidal structure, derived from the notion of a monoid, or further geometric structure --- a form, derived from the notion of partial order (we are now on the third line of Figure~\ref{figA}). 

  \item There are two canonical monoidal structures in a category, given by categorical product and sum (coproduct). These two structures are dual to each other. Similarly, on the form side, there are two canonical forms over a category --- that of subobjects from Example~\ref{ExaG}, and its dual, the \emph{form of quotients}. 

  \item The notion of a `biproduct' is a merger of the notions of product and sum; in fact, the product and the sum in a biproduct determine each other (via zero morphisms). Similarly, the notion of a `(short) exact sequence' is a merger of the notions of subobject and quotient: the kernel part of a short exact sequence represents a subobject and the cokernel part represents a quotient, and here too these two determine each other by the kernel-cokernel correspondence (and hence again via zero morphisms). 

  \item Moving on to the next line of Figure~\ref{figA}, `linear categories' (in the terminology of Lawvere) are pointed categories where any two objects can be combined in a biproduct. Similarly (but in some sense an opposite way), on the form side, `exact categories' (in the sense of Puppe and Mitchell \cite{P62,M65}, which are the same as pointed generalized exact categories in the sense of Grandis \cite{G13}, referred to in this paper, following \cite{JW16}, as \emph{Grandis exact categories}) can be seen as pointed categories where every morphism is a combination of two short exact sequences determined by the image and the coimage of the morphism (exactness requires that every morphism factorizes as a cokernel followed by a kernel). Moreover, similarly to how in a linear category products and sums are isomorphic, in an exact category the form of subobjects and the form of quotients are isomorphic. Furthermore, as it follows from Theorem~5.1 in \cite{JW16}, existence of such an isomorphism together with monomorphisms and epimorphisms (that represent subobjects and quotients, respectively) being part of a `(proper) factorization system' in the sense of \cite{FK72,I58}, is equivalent to exactness. This is analogous to how linearity of the category is equivalent to the matching of products and sums. 

  \item Finally, abelian categories are where the algebraic and geometric branches of Figure~\ref{figA} meet once again, as abelian categories are those categories which are simultaneously linear and exact. An example of an exact category which is not linear is the category of Dedekind groups (i.e., groups whose every subgroup is normal), while an example of a linear category which is not exact is the category of commutative monoids. It is worth noting that an exact category that has binary products is already linear and hence is an abelian category.
\end{itemize}

\section{Orean Forms}\label{secA}

\subsection*{Definition and elementary properties of orean forms} Substructures of mathematical structures admit direct images and inverse images along structure-preserving maps. Substructures also intersect and can be joined. Below we introduce axioms on a general form that conceptualize this behavior of substructures. 

\begin{definition}
  An \emph{orean form} is a form $F$ satisfying the following additional axioms:
  \begin{itemize}
    \item[(O1)] For any object $X$, the poset of clusters in $X$ (the poset structure is given by the relation $\ge_X$) is a bounded lattice (we write $\bot_X^F$ and $\top_X^F$ for the bottom and top elements of this lattice, and $\vee_X^F$ and $\wedge_X^F$ for its join and meet --- omitting subscripts and superscripts, when convenient).

    \item[(O2)] The reverse implication to the one in (F2), 
      \[\exists_{B\subseteq_F Y}[C\ge_g B\ge_f A]\impliedby [C\ge_{gf} A],\]
    holds for any sequence
      \[\xymatrix{X\ar[r]^-{f} & Y\ar[r]^-{g} & Z}\]
    of morphisms and clusters $A$ and $C$ in $X$ and $Z$, respectively.

    \item[(O3)] For any morphism $f\colon X\to Y$ and cluster $S$ in $X$, the set    \[\{T\subseteq_F Y\mid T\ge_f S\}\]
    has smallest element (with respect to the order of clusters in $Y$), called \emph{direct image} of $S$ along $f$ and written as 
      \[f\cdot S=f\cdot^F S=\mathsf{min}\{T\subseteq_F Y\mid T\ge_f S\}.\]
    Dually, for each $T\subseteq_F Y$ there exists
      \[T\cdot f=T\cdot^F f=\mathsf{max}\{S\subseteq_F X\mid T\ge_f S\},\]
    called \emph{inverse image} of $T$ along $f$.
  \end{itemize}
\end{definition}

\begin{remark}\label{RemC}
  It is easy to prove that under (O1) and (O3), axiom (O2) is equivalent to the following `associativity' law:
    \[g\cdot (f\cdot S)=(g\circ f)\cdot S.\]
  By duality, it is also equivalent to the dual of this law:
    \[(U\cdot g)\cdot f=U\cdot (g\circ f).\]
  Another equivalent expression of (O2) is the law
    \[U\cdot g\ge_f h\cdot R\quad \iff\quad  U\ge_{gfh} R.\]
  Explicit proofs of these equivalences are given in \cite{vN19}.
  Axiom (O3) gives that the relation $T\ge_f S$ constitutes a Galois connection in the sense of Ore \cite{O44}:
    \[T\ge f\cdot S\quad\iff\quad T\ge_f S \quad\iff\quad T\cdot f\ge  S.\]
  Hence, by the well-known facts about Galois connections, we have the following laws:
\begin{itemize}
    \item $(f\cdot S)\cdot f\ge S$, and dually, $T\ge f\cdot (T\cdot f)$,
    \item $f\cdot ((f\cdot S)\cdot f)=f\cdot S$, and dually, $(f\cdot (T\cdot f))\cdot f=T\cdot f$.
    \item If $S\ge S'$ then $f\cdot S\ge f\cdot S'$. Dually, if $T\ge T'$ then $T\cdot f\ge T'\cdot f$.
    \item $(f\cdot S)\cdot f\ge (f\cdot S')\cdot f$ if and only if $f\cdot S\ge f\cdot S'$. Dually, $f\cdot (T\cdot f)\ge f\cdot (T'\cdot f)$ if and only if $T\cdot f\ge T'\cdot f$. 
    \item $f\cdot \bot=\bot$, and dually, $\top\cdot f=\top$. 
    
    \item $f\cdot(S\vee S')=(f\cdot S)\vee (f\cdot S')$, and dually, $(T\wedge T')\cdot f=(T\cdot f)\wedge (T'\cdot f)$. 
    \end{itemize}
\end{remark}
\begin{remark}
  A form is orean if and only if it is a bifibration in the sense of Grothendieck \cite{G57}, whose fibres are bounded lattices. If we further require that the fibres are complete lattices, we get the notion of a \emph{topological functor} (see e.g.~\cite{Bru84}).
\end{remark}

\begin{example}
  An orean form over $\mathbf{1}$ is essentially the same as a bounded lattice, while an orean form over $\mathbf{2}$ is essentially the same as a Galois connection between bounded lattices (see Examples~\ref{ExaU}, \ref{ExaT}).
\end{example} 

\begin{example}\label{ExaF}
  The forms of substructures and forms of quotients of mathematical structures are usually orean. The form of subobjects of a category (as considered in Example~\ref{ExaG}) is an orean form if and only if monomorphisms admit pullbacks along arbitrary morphisms (this gives construction of inverse images as well as meets of subobjects), subobjects of each object admit finite joins, and any morphism $f$ factorizes as $f=me$, where $m$ is a monomorphism, so that whenever $hf=m'e'$ with $m'$ a monomorphism, we have $m'u=hm$ for some morphism $u$ (this gives construction of direct images of subobjects) --- see the discussion after Theorem~2.4 in \cite{JW14}, from which one can get a similar characterization of when is, more generally, a form of subobjects relative to a class $\mathcal{M}$ of monomorphisms (as considered in Example~\ref{ExaH}) orean, as long as $\mathcal{M}$ contains identity morphisms and is closed under composition with isomorphisms. 
\end{example}

\begin{example}
Consider a form of $(\mathcal{E},\mathcal{M})$-subquotients $F$, as described in Example~\ref{ExaAC}, where $\mathcal{M}$ is a class of monomorphisms and $\mathcal{E}$ is a class of epimorphisms. It is an easy routine to show that if the forms of $\mathcal{M}$-subobjects and $\mathcal{E}$-quotients separately are orean, then $F$ is also orean.
\end{example}

\begin{lemma}\label{LemQ}
Consider a category $\mathbb{C}$ equipped with an orean form $F$. For any isomorphism $f\colon X\to Y$ and any cluster $S$ in $X$ we have (where $f^{-1}$ denotes the inverse of $f$):
\[(f\cdot S)\cdot f=S,\quad f\cdot S=S\cdot f^{-1}, \quad f^{-1}\cdot (S\cdot f^{-1})=S.\]
Furthermore, we have: 
\[1_X\cdot S=S=S\cdot 1_X.\]
\end{lemma}

\begin{proof}
We begin by proving the last statement of the lemma:
\[1_X\cdot S=\min\{T\subseteq X\mid T\ge S\}=S.\]
We get $S\cdot 1_X=S$ by duality. Now we prove the left equality in the first part of the lemma, using what we just proved, as well as Remark~\ref{RemC}. 
\begin{align*}
    (f\cdot S)\cdot f &= 1_X\cdot ((f\cdot S)\cdot f)\\
    &=(f^{-1}\circ f)\cdot ((f\cdot S)\cdot f)\\
    &=f^{-1}\cdot ( f\cdot ((f\cdot S)\cdot f))\\
    &=f^{-1}\cdot ( f\cdot S)\\
    &=(f^{-1}\circ f)\cdot S\\
    &=1_X\cdot S\\
    &=S.
\end{align*}  
The middle equality $f\cdot S=S\cdot f^{-1}$ follows with similar computation by taking inverse image of both sides of the left equality along $f^{-1}$. The right equality is dual to the left one (with $f$ and $f^{-1}$ swapped), but it can also be obtained by taking the direct image, along $f^{-1}$, of both sides of the middle equality.  
\end{proof}

\begin{remark}
Consider an orean form $F$ over a category $\mathbb{C}$. Assigning to an object $X$ of $\mathbb{C}$ its poset of clusters (which is a bounded lattice, by (O1)), we get a functor $\tilde{F}$ from $\mathbb{C}$ to the category of posets and Galois connections. Seeing an orean form as a bifibration, this is a familiar representation of $F$ related to the so-called `Grothendieck construction'. Knowing this, Lemma~\ref{LemQ} can be obtained as a consequence of the fact that $\tilde{F}$ (just as any functor) preserves isomorphisms, by using the well-known description of isomorphisms and identity morphisms in the category of posets, with Galois connections as morphisms.
\end{remark}

\subsection*{Normality and conormality of clusters} Here we introduce and explore two fundamental dual notions pertaining to orean forms, which bring together the notions of an image of a homomorphism and a kernel of a homomorphism in abstract algebra. `Normal' clusters generalize normal subgroups of groups, while `conormal' clusters correspond to arbitrary subgroups. The main result of the paper, established towards the end of the paper, will deal with forms that decompose any cluster into a normal and a conormal one. Note that in group theory, such decomposition is trivial, since any normal cluster is a conormal cluster (a normal subgroup is a subgroup). In general, this need not be the case, even in classical algebra. For instance, in the theory of rings with identity, ideals can be seen as normal clusters and subrings of rings as the conormal ones: it is no longer true that one of these include the other. 

\begin{definition} 
For a given orean form, a cluster $S$ in an object $X$ is said to be an \emph{image} of a morphism $f\colon W\to X$ if $S=f\cdot \top_W$, and dually, it is said to be a \emph{kernel} of a morphism $g\colon X\to Y$ if $S=\bot_Y\cdot g$. A cluster that is \emph{image}/\emph{kernel} of some morphism is called a \emph{conormal}/\emph{normal} cluster. We write the image and kernel of a morphism $f$ as $\im ^Ff$ and $\ker ^Ff$, respectively (omitting $F$ in the superscript when convenient).
An orean form $F$ is said to be a \emph{conormal form} when every cluster is conormal. A \emph{normal form} is defined dually. A form is \emph{binormal} if it is both normal and conormal.
\end{definition}

\begin{remark}
The terms introduced in the first part of the definition above allude to what the corresponding notions give in the case of the form of subgroups, for instance. There the terms `image', `kernel', `normal', etc., all obtain their usual meaning. The only new term is `conormal' (used already in \cite{M50}, but with a slightly different meaning). We do not encounter the notion of a `conormal subgroup' in classical group theory because for the form of subgroups of groups, every cluster (subgroup) is conormal. Note that its dual property does not hold: not every subgroup is normal. Thus, the form of subgroups is conormal, but not normal.
\end{remark}

\begin{example}\label{ExaAD}
  Forms of subobjects relative to a class $\mathcal{M}$ of monomorphisms as in Example~\ref{ExaF} are conormal. So forms of substructures of mathematical structures are usually conormal. Similarly, forms of quotients of mathematical structures are normal.
\end{example}

\begin{example}\label{ExaN}
  Recall that an \emph{ideal $\mathcal{N}$ of null morphisms} (see \cite{Ehr64,Kel64,Lav65}) is a class of morphisms in a category which is closed under left and right composition with arbitrary morphisms. A particular case of this is the class of null morphisms in a pointed category. The notion of kernel and cokernel straightforwardly generalize from the pointed context to that of a category equipped with an arbitrary fixed ideal $\mathcal{N}$ of null morphisms. When all morphisms have kernels and cokernels relative to $\mathcal{N}$, it can be easily checked that the form of $\mathcal{M}$-subobjects (see Example~\ref{ExaH}) is a binormal orean form, where $\mathcal{M}$ is the class of kernels; it is actually isomorphic to the form of $\mathcal{E}$-quotients, where $\mathcal{E}$ is the class of cokernels. These observations take us to \cite{G13} and the earlier work of Marco Grandis cited there, where the reader may find a rich supply of examples that fit in this context. The forms of submodules of modules of a ring $R$, and more generally, the forms of subobjects in abelian categories, provide classical instances of binormal orean forms arising from an ideal of null morphisms. It is worth noting that an ideal $\mathcal{N}$ of null morphisms admitting kernels and cokernels can be actually recovered from the corresponding form as the class of those morphisms for which the direct image of the top cluster in the domain is the bottom cluster in the codomain (see \cite{JW14}).    
\end{example}

\begin{definition}
Given an orean form $F$, its subform (in the sense of Definition~\ref{DefF}) of conormal $F$-clusters is called the \emph{conormal hull} of $F$ and is denoted by $F_\c$. The \emph{normal hull} of $F$ is defined dually as the subform of normal $F$-clusters and is denoted by $F_\n$.
\end{definition}

\begin{convention}
We often denote a form by a letter with a subscript, such as in $F_\c$. In such cases we use only the subscript to reference the corresponding form notation, e.g., instead of $A\vee_X^{F_\c} B$ we write $A\vee_X^{\c} B$.
\end{convention}

\begin{definition} Let $F$ be an orean form.
For an $F$-cluster $A$, we write $\cccl{A}$ to denote the largest conormal $F$-cluster below $A$, when it exists, and call it the \emph{conormal interior} of $A$, while $\ccl{A}$ denotes the smallest conormal $F$-cluster above $A$, when the latter exists, called the \emph{conormal exterior} of $A$. Dually, the \emph{normal exterior} of $A$ is the smallest normal cluster $\ncl{A}$ above $A$, when it exists, and the \emph{normal interior} of $A$ is the largest normal cluster $\nccl{A}$ below $A$, when the latter exists.
\end{definition}

\begin{remark}
Note that the construction $\cccl{A}$ is dual to $\ncl{A}$, while $\ccl{A}$ is dual to $\nccl{A}$. 
\end{remark}

\begin{example}
Let $F$ be the form of equivalence relations on sets. For an equivalence relation $E$ on a set $X$, we have:
\begin{itemize}
    \item $\cccl{E}$ exists if and only if $E$ is conormal, i.e., $E$ has at most one non-singleton equivalence class (in which case, $\cccl{E}=E$). 
    
    \item $\ccl{E}$ is obtained from $E$ by joining all of its non-singleton equivalence classes into one equivalence class.
    
    \item $\ncl{E}=\nccl{E}=E$ since $E$ is always normal. 
\end{itemize}
\end{example}

\begin{lemma}\label{LemM}
  Given an orean form $F$, the form $F_\c$ is orean if and only if $F$ admits the following constructions: 
  \begin{itemize}
    \item $\ccl{\bot^F_X}$, for any object $X$,
    \item $\ccl{A\lor B}$ and $\cccl{A\land B}$ for any two conormal $F$-clusters $A$ and $B$ in an object $X$,
    \item $\cccl{A\cdot f}$ for any morphism $f\colon X\to Y$ and conormal $F$-cluster $A$ in $Y$. 
  \end{itemize}
  Moreover, when $F_\c$ is orean, $F_\c$ is conormal and we have:
  \begin{align*}
    \top_X^\c&=\top_X^F=\ccl{\top_X^F}=\cccl{\top_X^F}, &
    \bot_X^\c &=\ccl{\bot_X^F},\\
    A\land^\c B&=\cccl{A\land^F B}, &
    A\lor^\c B&=\ccl{A\lor^F B},\\
    A\cdot^\c f&= \cccl{A\cdot^F f}, &
    f\cdot^\c A &= f\cdot^F A=\ccl{f\cdot^F A}=\cccl{f\cdot^F A}.
  \end{align*}
\end{lemma}
\begin{proof}
  It is an easy routine to get the first part of the lemma. The second part follows from it together with the fact that the largest cluster is always conormal and the direct image of a conormal cluster is conormal.  
\end{proof}

\begin{definition} An orean form $F$ is said to be \emph{strongly orean} if $F_\c$ and $F_\n$ are both orean.
\end{definition}

\begin{example}\label{ExaAB}
The form $F$ of additive subgroups of rings is an example of a strongly orean form over the category $\mathbf{Rng}_1$ of rings with identity. It can be obtained as a pullback of the form of subgroups over the category $\mathbf{Ab}$ of abelian groups, along the forgetful functor $\mathbf{Rng}_1\to \mathbf{Ab}$, which assigns to a ring its underlying additive group. In this case, $F_\c$ is the form of subrings (which is isomorphic to the form of subobjects of $\mathbf{Rng}_1$), while $F_\n$ is the form of ideals (which is isomorphic to the form of quotient rings).
\end{example}

\subsection*{Noetherian Forms} Here we introduce a concept that lies at the basis of this paper. A `noetherian form' is an orean form with sufficient properties to give a unified self-dual context for establishing homomorphism theorems (in particular, isomorphism theorems and diagram lemmas of homological algebra). 

Consider the form of subgroups of groups. In this form, a subgroup $S$ of a group $G$ is a cluster and not an object in the category of groups. However, it defines at the same time an object too, the subgroup $S$ viewed as a group, along with a morphism $S\to G$ --- the subgroup inclusion map. On the other hand, every normal subgroup $N$ of $G$ can also be realized as the quotient group $G/N$, which also comes with a morphism, the quotient map $G\to G/N$. These two constructions of realizing a cluster in the ground category can in fact be described by dual universal properties detailed in the following definition. 

\begin{definition} For a given orean form, an \emph{embedding} of a conormal cluster $S$ is a morphism $f\colon W\to X$ whose image is $S$, and which has the property that any morphism $f'\colon W'\to X$ whose image is smaller than $S$ factors through $f$ uniquely, i.e., $fg=f'$ for a unique morphism $g$.
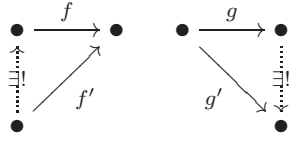
  \[\xymatrix{\bullet\ar[r]^-{f} & \bullet \\ \bullet\ar@{..>}[u]|-{\exists!}\ar[ur]_-{f'} & }\quad \xymatrix{\bullet\ar[r]^-{g}\ar[rd]_-{g'} & \bullet\ar@{..>}[d]|-{\exists!} \\ & \bullet }\]
Dually, a \emph{quotient} of a normal cluster $S$ is a morphism $g\colon X\to Y$ whose kernel is $S$ and is such that any morphism $g'\colon X\to Y'$ whose kernel is larger than $S$ factors through $g$ uniquely. A morphism that is a quotient of some kernel is also called a \emph{projection}. 
\end{definition}

\begin{remark}\label{RemJ}
The following observations will be useful:
\begin{itemize}
\item It is easy to see that an embedding is always a monomorphism and a quotient is always an epimorphism.

\item If $f$ is an embedding of a conormal cluster $S$, then $f'$ is also an embedding of $S$ if and only if $f'=fi$ for an isomorphism $i$. Dually, if $g$ is a quotient of a normal cluster $S$, then $g'$ is also a quotient of $S$ if and only if $g'=ig$ for an isomorphism $i$. 

\item In the definition of an embedding, we could have required $S\ge \im f$ instead of $S=\im f$, since by applying the universal property to a morphism $f'$ such that $S=\im f'$, we would get $\im f\ge f\im g=\im (fg)=\im f'=S$.

\item For an orean form $F$, if every conormal $F$-cluster has an embedding, then $F_\c$ is isomorphic to the form of $\mathcal{M}$-subobjects of Example~\ref{ExaH}, where $\mathcal{M}$ is the class of $F$-embeddings. The isomorphism assigns to each conormal cluster the $\mathcal{M}$-subobject represented by its embeddings.
\end{itemize}
\end{remark}

\begin{convention}
We write $\iota_A$ to denote an embedding of a conormal cluster $A$. In light of Remark~\ref{RemJ}, $\iota_A$ is not uniquely determined by $A$. We assume that at the time of first usage of the notation, one embedding has been arbitrarily chosen. If this first usage involves an additional condition, we assume that it is possible to choose $\iota_A$ in such a way as to satisfy that condition. The dual notion, for a quotient, is $\pi_S$. When we need to refer to the form $F$ with respect to which we are considering an embedding and a projection, we write $\iota^F_A$ and $\pi^F_S$, respectively.  
\end{convention}

\begin{lemma}\label{LemV}
Consider an orean form $F$ over a category $\mathbb{C}$. For a morphism $f\colon X\to Y$ in $\mathbb{C}$, the following conditions are equivalent: 
\begin{itemize}
    \item $f$ is an isomorphism.
    
    \item $f$ is an embedding of the top $F$-cluster in $Y$, i.e., $f=\iota_{\top_Y}$.
    
    \item $f$ is a quotient of the bottom $F$-cluster in $X$, i.e., $f=\pi_{\bot_X}$.
\end{itemize}
\end{lemma}

\begin{proof}
Applying the second bullet point of Remark~\ref{RemJ}, as well as duality, it is sufficient to show that an identity morphism $1_X\colon X\to X$ is an embedding of $\top_X$. This follows easily from the fact that $1_X\cdot \top_X=\top_X$, which we have by Lemma~\ref{LemQ}. 
\end{proof}

The following lemma was established in \cite{vN19}. We include the proof for the sake of completeness. The following is a consequence of Proposition 3.2.2 in \cite{vN19}, but we include a proof nevertheless.

\begin{lemma}\label{LemP}
  Given an orean form, a conormal cluster $B$ in an object $Y$ whose embedding is a morphism $\iota_B=m\colon M\to Y$, and given also a morphism $f\colon X\to Y$ such that $B\cdot f$ is conormal, we have: a morphism $m'\colon M'\to X$ is an embedding of $B\cdot f$, i.e., $m'=\iota_{B\cdot f}$, if and only if $m'$ is a pullback of $m$ along $f$.     
\end{lemma}
\begin{proof}
  Suppose first $m'=\iota_{B\cdot f}$. Then $\im m'=B\cdot f$. This implies $B\ge f\cdot \im m'=\im (fm')$. We obtain a morphism $f':M'\to M$ making the following diagram commute:
    $$\xymatrix{M'\ar[r]^-{f'}\ar[d]_-{m'} & M\ar[d]^-{m}\\ X\ar[r]_-{f} & Y}$$
  To prove that this diagram is a pullback, consider its commutative extension:
    $$\xymatrix{W\ar@/_15pt/[ddr]_{r}\ar@/^15pt/[rrd]^-{s} & & \\ & M'\ar[r]^-{f'}\ar[d]_-{m'} & M\ar[d]^-{m}\\ & X\ar[r]_-{f} & Y}$$
  Since $B=\im m\ge \im (ms)=\im (fr)=f\cdot \im r$, we have $B\cdot f\ge \im r$. This produces a morphism $w\colon W\to M'$ such that $m'w=r$. Every embedding is a monomorphism (Remark~\ref{RemJ}), so in particular, both $m$ and $m'$ are monomorphisms. As $m$ is a monomorphism, we get $f'w=s$. Since $m'$ is a monomorphism, $w$ is a unique morphism such that $m'w=r$ and $f'w=s$, proving the pullback property:
    $$
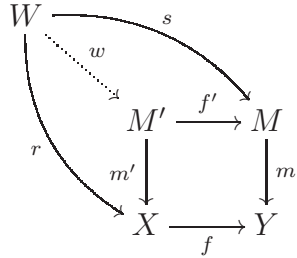
\xymatrix{W\ar@{..>}[dr]^-{w}\ar@/_15pt/[ddr]_{r}\ar@/^15pt/[rrd]^-{s} & & \\ & M'\ar[r]^-{f'}\ar[d]_-{m'} & M\ar[d]^-{m}\\ & X\ar[r]_-{f} & Y}$$
  Suppose now we have a pullback given by the square two diagrams above the last diagram. We want to prove $m'=\iota_{B\cdot f}$. Commutativity of the square gives $B= \im m \ge \im (mf')=\im (fm')= f\im m'$, and so $B\cdot f\ge \im m'$. Let $r$ be a morphism $r\colon W\to X$ such that $B\cdot f\ge \im r$. Then $B\ge f\im r=\im (fr)$. Since $m=\iota_B$, there is a morphism $s\colon W\to M$ such that $ms=fr$. The pullback property gives a factorization $m'w=r$. Such $w$ is unique, since being a pullback of a monomorphism $m$, the morphism $m'$ is itself a monomorphism.  
\end{proof}

\begin{example}\label{ExaW}
  For a class $\mathcal{M}$ of monomorphisms as in Example~\ref{ExaF}, a morphism $m$ is an $F$-embedding of a cluster $M$, where $F$ is the form of $\mathcal{M}$-subobjects, if and only if $m$ is one of the monomorphisms from the class $\mathcal{M}$ that represents the subobject $M$. 
\end{example}

\begin{remark}\label{RemE}
  As established in Corollary 3.2 in \cite{JW14}, binormal orean forms arising from an ideal of null morphisms admitting kernels and cokernels, as described in Example~\ref{ExaN}, are precisely those binormal orean forms (up to isomorphism of forms) where every conormal cluster admits an embedding and every normal cluster admits a quotient (see also \cite{JW16}). 
\end{remark}

\begin{definition}\label{DefA}
  A \emph{prenoetherian form} is an orean form where every conormal cluster admits an embedding and every normal cluster admits a quotient. A \emph{noetherian form} is an orean form satisfying the following additional axioms (notice that (N2) subsumes the defining requirement of a prenoetherian form):
  \begin{itemize}
    \item[(N1)] For any $S\subseteq_F X$ and $f\colon X\to Y$, we have
      \[(f\cdot S)\cdot f=S\vee_X \ker f,\]
  and dually, for any $T\subseteq_F Y$,
      \[f\cdot (T\cdot f)=T\wedge_Y\im f.\]

    \item[(N2)] Any morphism $f\colon X\to Y$ factorizes as $f=\iota_{\im f}\circ \pi_{\ker f}$.

    \item[(N3)] Join of two normal clusters is normal, and dually, meet of two conormal clusters is conormal.
  \end{itemize}
\end{definition}

\begin{example}\label{ExaAA} The form of subsets and the form of equivalence relations over the category of sets are both orean forms. Moreover, they both satisfy (N3). However, (N2) fails for both forms. The form of subsets admits embeddings --- these are injective functions. The form of equivalence relations admits quotients, which are the surjective functions. Since in the form of subsets the empty subsets are the only normal clusters, quotients for this form are bijections. Any embedding in an orean form is a monomorphism, so if (N2) were true for the form of subsets, then any function would have been an injection. For the form of equivalence relations, all maps $f$ satisfy the requirements of (N2) except the empty maps and the constant maps. The images of such maps are the discrete equivalence relations and these do not admit embeddings. (N1) does hold for the form of equivalence relations, but the join formula in (N1) fails for the form of subsets (the meet formula, on the other hand, holds). Overall, neither the form of subsets nor the form of equivalence relations are noetherian forms.
\end{example}

\begin{example}\label{Ex0}
  Let $\mathbf{Grp}$ denote the category of groups, and let $\mathbf{Grp}_2$ denote the category of `pairs of groups', in the sense of algebraic topology. Its objects are pairs $(G,H)$ where $G$ is a group and $H$ is a subgroup of $G$. A morphism between pairs, $f\colon (G,H)\to (G',H')$ is a group homomorphism $f\colon G\to G'$ such that $f(x)\in H'$ for any $x\in H$. The functor $\mathbf{Grp}_2\to\mathbf{Grp}$ which maps each $f\colon (G,H)\to (G',H')$ to $f\colon G\to G'$ is a from that is isomorphic to the form of subgroups, which itself is isomorphic to the form of subobjects in the category of groups. This form is a noetherian form (see \cite{GJ19}). In particular, (N2) holds because of the first Noether isomorphism theorem. Groups in this example can be replaced with algebras in any variety having a unique constant $0$, as long as the variety contains the Ursini terms \cite{U73}, i.e., binary terms $d_1,\dots,d_n$ and an $(n+1)$-ary term $p$ satisfying
    \[\left\{\begin{array}{l} d_1(x,x)=0,\dots,d_n(x,x)=0,\\ p(d_1(x,y),\dots,d_n(x,y),y)=x.\end{array}\right.\] 
  These are precisely those varieties whose categories of algebras are semi-abelian in the sense of \cite{JMT02} (see \cite{BouJan03}). As it follows from Theorem 9.1 in \cite{J14}, for any pointed category $\mathbb{C}$ having finite limits and colimits, the form of subobjects is noetherian if and only if $\mathbb{C}$ is a semi-abelian category (see also \cite{vN19}). Thus, the forms of subobjects of categories of the following algebraic structures are noetherian: modules, groups, rings without identity, Heyting meet semi-lattices, Lie algebras, and essentially all other pointed group-like structures. Counterexamples include monoids and meet semi-lattices. As shown in \cite{GSV19} (see also \cite{GKV16}), the category of cocommutative Hopf algebras over any field is a semi-abelian category; hence, its form of subobjects is noetherian. 
\end{example}

\begin{remark}\label{RemB}
  The join formula in (N1) is equivalent to the requirement that it holds when $S\ge_X \ker  f$, and dually, the meet formula is equivalent to the requirement that it holds when $ \im  f\ge_Y T$ (see \cite{J14}). So the formulas in (N1) are each equivalent to the following implications: 
    \[S\ge_X \ker  f\;\implies\;(f\cdot S)\cdot f=S,\quad \im  f\ge_Y T \;\implies\; f\cdot (T\cdot f)=T.\]
  These implications state that direct and inverse image maps constitute an isomorphism of lattices between the upset of the kernel of $f$ and the downset of the image of $f$. This is known in group theory as the `lattice isomorphism theorem' (see e.g.~\cite{MB99}).
\end{remark}

\begin{remark}\label{RemD}
  Under (N1) and (N2), axiom (N3) has the following equivalent reformulation (see \cite{GJ19}): direct image of a normal cluster under a quotient is normal.
\end{remark}

\begin{remark}\label{RemF}
As it follows from the results contained in \cite{JW16,G13}, a binormal orean form is noetherian if and only if (N2) holds (notice that (N3) in particular is a trivial consequence of binormality), and moreover, isomorphism classes of binormal noetherian forms over a category $\mathbb{C}$ are given by ideals $\mathcal{N}$ of null morphisms such that the pair $(\mathbb{C}, \mathcal{N})$ is a Grandis exact category \cite{G84,G92,G12,G13}. In other words, conceptually, a binormal noetherian form is the same as a Grandis exact category, just as a binormal orean form where every conormal cluster admits an embedding and every normal cluster admits a quotient is the same as a category equipped with an ideal of null morphisms admitting kernels and cokernels (see Remark~\ref{RemE}). Recall from \cite{G12,G13} that when $\mathbb{C}$ is a pointed category, there is exactly one $\mathcal{N}$ such that $(\mathbb{C}, \mathcal{N})$ is a Grandis exact category, and it is given by null morphisms of the pointed category. We recall also from \cite{G12} that pointed categories $\mathbb{C}$ such that the class $\mathcal{N}$ of null morphisms is such are precisely the Puppe-Mitchell categories, which, in the case of existence of finite products and sums, are the same as abelian categories. 
\end{remark}

\section{Formal Operators}\label{SecI}

\subsection*{Definition and examples} Morphisms of forms are called `operators'. Here we define them and provide some illustrative examples.

\begin{definition}\label{DefF}
Given two forms $F\colon\mathbb{A}\to\mathbb{C}$ and $G\colon\mathbb{B}\to\mathbb{C}$ over the same category $\mathbb{C}$, we define:
\begin{itemize}
\item an \emph{operator} $\tau\colon F\to G$ is a functor $\tau\colon\mathbb{A}\to\mathbb{B}$ such that $G\circ\tau=F$:
  \[\xymatrix{ \mathbb{A}\ar[rr]^-{\tau}\ar[dr]_-{F} & & \mathbb{B}\ar[dl]^-{G} \\ & \mathbb{C} & }\]
It amounts to mapping each $F$-cluster $A$ in an object $X$ to a $G$-cluster $\tau A$ in $X$ so that the implication
  \[A_2\ge_f^F A_1\quad\implies\quad \tau A_2\ge_f^G \tau A_1\]
holds. We call this implication \emph{monotonicity} of $\tau$. 

\item When the reverse implication also holds, $\tau$ is said to be a \emph{full operator} (it is equivalent to $\tau$ being full as a functor $\mathbb{A}\to\mathbb{B}$). 

\item When $\tau$ is injective on objects, we say that $\tau$ is an \emph{injective operator}. 

\item When in addition to being a full operator, we have $\tau A=A$ for all $A\in\mathbb{A}$, we say that $F$ is a \emph{subform} of $G$ and $\tau$ is then called \emph{subform inclusion}. 

\item When $F=G$ and $\tau\circ \tau=\tau$, we say that $\tau$ is \emph{idempotent}. 

\item Suppose $F=G$. We say that a cluster $A$ is $\tau$-\emph{closed} (or, closed relative to $\tau$) when $\tau A=A$. We write $F_\tau$ for the subform of $F$ consisting of $\tau$-closed clusters. 

\item When $\tau$ is an identity functor we say that it is an \emph{identity operator}.
\end{itemize}
\end{definition}

\begin{remark} We note that:
\begin{itemize}
\item Let $F,G$ be forms over the same category. Thanks to faithfulness of $G$, as a functor, an operator $\tau\colon F\to G$ is fully described by its action on $F$-clusters: two operators $\tau,\tau'\colon F\to G$ are equal if and only if we have $\tau A=\tau' A$ for any $F$-cluster $A$. 

\item Any full operator is injective. 

\item When one orean form $G$ is a subform of another orean form $F$, this does not mean that $G$ is closed in $F$ under orean constructions (top/bottom clusters, meets/joins of clusters, direct/inverse images of clusters along morphisms).  
\end{itemize}
\end{remark}

\begin{remark}
Note that forms over a given category along with operators form a category, under the obvious notion of composition of operators. Isomorphisms in this category are the same as isomorphisms of forms in the sense of Definition~\ref{DefC}. 
\end{remark}

\begin{lemma}\label{LemY}
An operator $\tau\colon F\to G$ from a form $F$ to a form $G$ is full if and only if the mapping $A\mapsto \tau A$ defines an isomorphism between $F$ and the subform of $G$ consisting of clusters of the type $\tau A$.  
\end{lemma}

\begin{proof}[The proof of this lemma is rather straightforward.]  
\end{proof}

\begin{remark}
Mapping a form to its dual form determines an isomorphism of the category of forms over a given category $\mathbb{C}$ and the category of forms over the dual category $\mathbb{C}^\mathsf{op}$, which extends duality of forms to operators. The dual of an operator $\tau:F\to G$ is given by $\tau^\mathsf{op}:F^\mathsf{op}\to G^\mathsf{op}$, where $\tau^\mathsf{op}$ is the dual functor $\tau^\mathsf{op}\colon \mathbb{A}^\mathsf{op}\to\mathbb{B}^\mathsf{op}$. In terms of their effect on clusters, there is of course no difference between $\tau$ and $\tau^\mathsf{op}$ (in general, a functor and the dual functor act the same way both on objects and on morphisms).  
\end{remark}

\begin{example}
Any selection of clusters for a form $F$, per object $X$ of the ground category, defines a (unique) subform of $F$. For example, the form of subgroups has a subform given by subgroups of a fixed order $k$.
\end{example}

\begin{example}\label{ExaX}
For a finite set $X$, let $\mathsf{P}X$ denote the power-set of $X$. Define a \emph{palette} in a finite set $X$ to be a non-empty subset $P$ of $\mathsf{P}X$ such that
\[\forall_{A_1,A_2\in P}[[A_1\subseteq A_2]\implies [A_1=A_2]].\]
Note that any non-empty subset $P$ of $\mathsf{P}X$ can be turned into a palette $P^*$ by selecting the maximal elements of $P$, relative to the subset inclusion order.
Given a function $f\colon X\to Y$, define $Q\ge_f P$ to mean
\[\forall_{A\in P}\exists_{B\in Q}[\{f(x)\mid x\in A\}\subseteq B].\]
We get an orean form $F$, which we will call the \emph{form of palettes of finite sets}, where:
\begin{itemize}
    \item The bottom cluster in a set $X$ is given by $\bot=\{\varnothing\}$.
    
    \item The top cluster in $X$ is given by $\top=\{X\}$.
    
    \item The join of two clusters is given by 
    $P_1\vee P_2=(P_1\cup P_2)^*$.

    \item The meet of two clusters is given by 
    $P_1\wedge P_2=\{A\cap B\mid [A\in P_1] \& [B\in P_2]\}^*$.

    \item $f\cdot P=\{\{f(x)\mid x\in A\}\mid A\in P\}^*$.
    
    \item $P\cdot f =\{\{x\in X\mid f(x)\in B\}\mid B\in P\}^*$.
\end{itemize}
We consider some operators associated to the form of palettes:
\begin{itemize}
\item Assigning to each cluster $P$ the union $\bigcup \sigma$ of its elements, 
\[\gamma\colon P\mapsto \bigcup P=\{x\in X\mid \exists_{A\in P}[x\in A]\},\]
we get an operator $\gamma\colon F\to G$, where $G$ denotes the form of subsets of finite sets. This operator is not injective, and hence it is not full either. 

\item Similarly, we have an operator $\lambda\colon F\to G$, assigning to each cluster $\sigma$, the intersection $\bigcap \sigma$ of its elements,
\[\lambda\colon P\mapsto \bigcap P=\{x\in X\mid \forall_{A\in P}[x\in A]\},\]
which, again, is not injective. 

\item We also have an operator $\tau\colon G\to F$ defined by $A\mapsto \{A\}$. This operator is full and hence injective. The clusters in the subform of $F$ that $G$ is isomorphic to by Lemma~\ref{LemY} are in fact the conormal palettes. Note that both $\gamma$ and $\lambda$ are left inverses of $\tau$, and hence the composites $\tau\gamma$ and $\tau\lambda$ are both idempotent operators. In both cases, closed clusters are the conormal palettes. 

\item We define an operator $\omega\colon F\to F$ as follows. Given a palette $P$, consider the equivalence relation on $P$ generated by the relation $A\cap B\neq\varnothing$. This will partition $P$. Now, take the union of each equivalence class. The set of these unions will result in a palette $\omega P$ that is \emph{separated}, i.e., any two distinct elements of $\omega P$ are disjoint. It is not difficult to see that this construction is monotone, and so $\omega$ is an operator $F\to F$. Moreover, it is not difficult to see either that $\omega$ is idempotent. $\omega$-closed clusters are precisely the separated palettes. $\omega$ is not injective on objects, and hence neither it is full.

\item Something that is not an operator: if instead of unions of equivalence classes in definition of $\omega$, we took their intersection, then monotonicity would fail and so we would not get an operator $F\to F$.
\end{itemize}
\end{example}

\subsection*{Product of forms} Next, we introduce a notion of cartesian product of forms.

\begin{definition}
Consider a pair $(F_1 ,F_2 )$ of forms. Their \emph{product} $F_1 \times F_2 $ is defined as the diagonal functor in the pullback 
  \[\xymatrix@=18pt{ & \mathbb{B}_1\!\times_\mathbb{C}\mathbb{B}_2\ar[dl]_-{\varrho_1}\ar[dr]^-{\varrho_2}\ar[dd]|-{F_1 \times F_2 } & \\ \mathbb{B}_1\ar[dr]_-{F_1 } & & \mathbb{B}_2\ar[dl]^-{F_2 } \\ & \mathbb{C} & }\]
The operators $\varrho_1$ and $\varrho_2$ are called, respectively, the first and the second \emph{product projections}.
Given two operators $\tau_1\colon G\to F_1$ and $\tau_2\colon G\to F_2$, we write $(\tau_1,\tau_2)$ for the unique operator $\tau\colon G\to F_1\times F_2$ such that $\varrho_1\tau=\tau_1$ and $\varrho_2\tau=\tau_2$.
\end{definition}

\begin{remark}
A product of two forms over a category $\mathbb{C}$ in the sense of the definition above is of course the usual product of these forms as objects in the category of forms over $\mathbb{C}$.
\end{remark}

\begin{convention}
The cluster system that we usually work with when considering a product $F_1\times F_2$ of forms is given by:
\begin{itemize}
    \item Clusters in an object $X$ are pairs $(A_1,A_2)$, where $A_1$ is an $F_1$-cluster in $X$ and $A_2$ is an $F_2$-cluster in $X$.
    
    \item For a morphism $f\colon X\to Y$ we have $(B_1,B_2)\ge_f (A_1,A_2)$ if and only if $B_1\ge_f A_1$ and $B_2\ge_f A_2$.
\end{itemize}
That this cluster system is a presentation of $F_1\times F_2$ follows easily from the well-known description of pullbacks of functors.
\end{convention}

\begin{example}\label{ExaY}
Let $F,G,\gamma,\lambda$ be as in Example~\ref{ExaX}. The operator $(\gamma,\lambda)\colon F\to G\times G$ is defined by $(\gamma,\lambda)P=(\gamma P,\lambda P)$. We have an operator $\delta\colon G\times G\to F$ defined by $\delta(A,B)=\{A,B\}^*$, where $\{A,B\}^*$ equals $\{A,B\}$ when $A$ and $B$ are not subsets of each other, and it equals $\{A\cup B\}$ otherwise. The composite $\varepsilon=(\gamma,\lambda)\delta\colon G\times G\to G\times G$ is not idempotent. However, we do have $\varepsilon\varepsilon\varepsilon=\varepsilon\varepsilon$. Closed clusters relative to $\varepsilon$ are clusters of the form $(A,A)$. The other composite $\delta(\gamma,\lambda)\colon F\to F$ is, on the other hand, equal to the composite $\tau\lambda$, where $\tau$ is as in Example~\ref{ExaX}, and so $\delta(\gamma,\lambda)$ is idempotent (which explains why we have $\varepsilon\varepsilon\varepsilon=\varepsilon\varepsilon$).      
\end{example}

\begin{remark}\label{RemL}
When $F_1 $ and $F_2 $ are orean, so is their product $F_1 \times F_2 $, in which all orean constructions (meet, direct images, etc.)~are done component-wise. In fact, it is not difficult to show that the converse is also true: if the product of two forms is orean then both forms are also orean. 
\end{remark}

\begin{remark}
Consider two orean forms $F$ and $G$. Since the orean structure of the product $F\times G$ is component-wise, we have:
  \[(F\times G)_\c=F_\c\times G_\c\textrm{, and dually, }(F\times G)_\n=F_\n\times G_\n,\]
From this and the fact that the product of two forms is orean if and only if each of the two forms are orean (Remark~\ref{RemL}), follows that the product of two forms is strongly orean if and only if each of the two forms are strongly orean.
\end{remark}

\begin{example}\label{ExaO} 
There are a number of operators associated with the product of forms: 
  \begin{itemize}
    \item $\varrho_1\colon F_1\times F_2 \to F_1$ and $\varrho_2\colon F_1\times F_2 \to F_2$ defined by $\varrho_1(A_1,A_2)=A_1$ and $\varrho_2(A_1,A_2)=A_2$.
    
    \item For an orean form $F$, the mappings $(A_1,A_2)\mapsto A_1\lor A_2$ and $(A_1,A_2)\mapsto A_1\land A_2$ define operators $F\times F\to F$.

    \item For orean forms $F,G$, we have operators $F\to F\times G$ and $G\to F\times G$ defined by $A\mapsto (\bot,A)$ and $B\mapsto (\top,B)$, for instance.

    \item There is an isomorphism $F_1\times F_2\to F_2\times F_1$ given by the mapping $(A_1,A_2)\mapsto (A_2,A_1)$.
  \end{itemize}
\end{example}

\subsection*{Closure operators} Many concepts pertaining to posets naturally generalize to forms. One of them is the concept of a closure operator. It will be a useful instrument in this paper, so we take some time to explore it. 

\begin{remark}
The so-called `categorical closure operators' are defined for subforms of the form of subobjects (see \cite{DG87,DT95,T11}). As shown in \cite{DAGJ17}, the following more general notion has similar features and new uses. We note also that the notion of a dual closure operator considered in \cite{DT15} is a special case of our notion of a co-closure operator introduced below, which is formal dual of the notion of a closure operator.
\end{remark}

\begin{definition}
A \emph{closure operator} on a form $F$ over a category $\mathbb{C}$ is an operator $\kappa\colon F\to F$ such that $\kappa S\ge_X S$ for each cluster $S$ in each object $X$. The dual of the notion of a closure operator on a form is the notion of a \emph{co-closure operator}. A co-closure operator is thus given by an operator $\delta\colon F\to F$ such that $S\ge_X \delta S$ for each cluster $S$ in each object $X$.
\end{definition}

\begin{example}
  The form $F$ of subspaces of topological spaces has an idempotent closure operator $\kappa$, given by topological closure of the underlying set of the subspace inside the ambient space. $F_\kappa$ is then the form of closed subspaces of topological spaces. 
\end{example}

\begin{remark}\label{RemK}
Let $F$ be a form. The following can be established thanks to well-known facts about idempotent closure operators on posets applied to fibres of $F$:
\begin{itemize}
    \item[(i)] For any idempotent closure operator $\kappa$ on $F$ and any cluster $S$, the cluster $\kappa S$ is the smallest closed cluster above $S$.
    
    \item[(ii)] Let $G$ be a subform of $F$ such that the subform inclusion $G\to F$ has a left inverse $\nu$. Assume furthermore that $\nu S\ge S$ for each cluster $S$. Then $\nu$ is a unique operator having these properties. Furthermore, the mapping $S\mapsto \nu S$ defines an idempotent closure operator $\kappa$ on $F$ with $F_{\kappa}=G$.
    
    \item[(iii)] More generally, given an operator $\tau\colon G\to F$ between forms, it has at most one left inverse $\nu\colon F\to G$ having the property that $\tau \nu S\ge S$. If such operator $\nu$ exists, then the composite $\kappa=\tau \nu$ is an idempotent closure operator on $F$ and $\tau$ restricts to an isomorphism $G\to F_{\tau \nu}$ of forms. Furthermore, when $\nu$ exists, $\tau$ preserves all joins that exist in the fibres of $G$.   
    
    \item[(iv)] For any idempotent closure operator $\kappa$ on $F$, the subform inclusion $F_\kappa\to F$ has a left inverse $\tau$ given by $\tau S=\kappa S$.
    
    \item[(v)] In the case when the fibres of $F$ have binary joins, for any idempotent closure operator $\kappa$ on $F$, the following conditions are equivalent: (a) $F_\kappa$ is closed in $F$ under binary joins; (b) $F_\kappa$ has binary joins and the subform inclusion $F_\kappa\to F$ preserves them; (c) $\kappa$ preserves binary joins.
\end{itemize}
\end{remark}

\begin{example} Consider from Example~\ref{ExaX} the form $F$ of palettes of finite sets, the form $G$ of subsets of finite sets, and the operators $\gamma, \lambda, \tau$ defined there.
For a palette $P$ we have: \[\tau\lambda P=\left\{\bigcap P\right\}\subseteq \sigma\subseteq\left\{\bigcup P\right\}=\tau\gamma P.\]
By Remark~\ref{RemK}, $\tau\gamma$ is then an idempotent closure operator on $F$ and by the dual of Remark~\ref{RemK}, $\tau\lambda$ is an idempotent co-closure operator on $F$. We note also that $\omega$ from Example~\ref{ExaX} is an idempotent closure operator. The operator $\varepsilon$ from Example~\ref{ExaY} is neither a closure operator nor a co-closure operator. However, $\varepsilon\varepsilon$ is a closure operator, as it is given by $\varepsilon\varepsilon\colon (A,B)\mapsto (A\cup B,A\cup B)$. 
\end{example}

\begin{theorem}\label{ThmN}
  Any orean form has at least two closure operators and two co-closure operators. The form of subsets of sets has just one additional closure operator, which fixes the empty subsets and maps each nonempty subset to the full set. The same form has exactly two co-closure operators. The form of equivalence relations of sets has exactly two closure operators and exactly two co-closure operators.
\end{theorem}
\begin{proof}
  The two (co-)closure operators of any orean form are given by the identity operator and the one that maps every cluster to the top cluster in case of a closure operator and to the bottom cluster in case of a co-closure operator. 

  It is easy to that the mapping described in the theorem is indeed a closure operator on the form of subsets of sets, different from the ones described above. In particular, its monotonicity is a consequence of the fact that the image of a nonempty subset is nonempty. Let $\kappa$ be a closure operator on the form of subsets of sets, which is not the identity operator. Then there is a set $X$ and a subset $S\subseteq X$ such that we have strict subset inclusion $S\subset \kappa S$. Pick any $x\in\kappa S\setminus S$. Consider a set $Y$ having at least two elements. For each pair of distinct elements $y_1\neq y_2$ in $Y$, consider the function $f\colon X\to Y$ mapping all elements of $X$ to $y_1$, except $x$, which $f$ maps to $y_2$. Then $\{y_1\}\ge_f S$ and so $\kappa\{y_1\}\ge_f \kappa S$, which implies that $y_2\in \kappa\{y_1\}$. We then get that the closure of any subset of $Y$ containing some element $y_1$ must contain any other element $y_2$. So, closure of any nonempty subset is the entire set. We must now show that either all empty subsets get fixed by $\kappa$ or they map to the entire sets as well. Suppose there is a set $X'$ whose empty subset does not get fixed. Then $X'$ has at least one element $x'$ in the closure of the empty subset of $X'$. For any nonempty set $Y'$, consider the function $f'\colon X'\to Y'$ that maps every element of $X'$ to some fixed element $y\in Y'$. Since we must have $\kappa\bot_{Y'}\ge_f \kappa\bot_{X'}$, we get $y\in \kappa\bot_{Y'}$. So the closure of the empty subset of any set $Y'$ contains every single element of $Y'$ --- it is thus the entire set $Y'$.

  Consider now a co-closure operator $\delta$ on the form of subsets of sets. Suppose there is some set $Y$ which has a subset $T$ such that there is a strict inclusion $\delta T\subset T$. Pick an element $y\in T\setminus \delta T$. For any set $X$, consider the function $f\colon X\to Y$ that maps all elements of $X$ to $y$. Then by monotonicity of $\delta$, elements of $\delta X$ would have to map somewhere else than $y$ and so $\delta X$ must be empty. This implies that $\delta S$ too is empty, for every subset $S$ of $X$. We have thus shown that if a co-closure operator of the form of subsets of sets is not the identity operator, then it maps every subset to the empty one.

  Similar arguments give the desired result for the form of equivalence relations. We only sketch them here. First, we consider the case of closure operators. If there is an equivalence relation that is different from its closure then the closure must merge at least two of its equivalence classes. Consider a function to any other set having at least two elements, which is constant on each equivalence class and maps elements from the two equivalence classes to any two distinct elements. Then the closure of the smallest equivalence relation on that set must merge those two elements in the same equivalence class. This gives that the closure of any equivalence relation on that set must be the top equivalence relation. Next, we consider the case of co-closure operators. If there is an equivalence relation whose co-closure is different from it, then there must be at least one equivalence class which has two elements that have been separated by the smaller equivalence relation. Consider any set and any two distinct elements in it. We can find a suitable function from this set to the first one to show that the co-closure of the top equivalence relation on the second set separates the two given elements. So it separates all pairs of elements and hence must be the bottom equivalence relation. Hence the co-closure of every equivalence relation is the bottom one.
\end{proof}
 
\begin{theorem}\label{ThmA}
  For an orean form $F$ and any idempotent closure operator $\kappa$ on $F$, the form $F_\kappa$ is also orean. In addition, $F_\kappa$ is closed in $F$ under inverse images and finite meets of clusters in the same object (and so they are computed in $F_\kappa$ the same way as in $F$); furthermore, finite joins and direct images of clusters for the form $F_\kappa$ are given by the closure of their finite joins and direct images, respectively, for the form $F$.  
\end{theorem}
\begin{proof}
  A top $F$-cluster is always closed and so it is a top $F_\kappa$-cluster. A bottom $F_\kappa$-cluster is given by the closure of a bottom $F$-cluster. For the form $F$, meet of closed clusters is closed and the join of closed clusters is given by
    \[S\vee^{\kappa}_X T=\kappa(S\vee^F_X T).\]
  So (O1) holds. For (O3), we can set:
    \[f\cdot^{\kappa} S=\kappa(f\cdot^{F} S)\textrm{ and }T\cdot^{\kappa} f=T\cdot^{F} f.\]
  The second of these equalities is possible since, if $T$ is closed, then so is $T\cdot^{F} f$ (this is an easy consequence of monotonicity of $\kappa$). To confirm (O2), it is sufficient to verify the second law in Remark~\ref{RemC} for $F_\kappa$, which reduces to the same law for $F$.
\end{proof} 

\subsection*{Conormal and normal operators} These two types of operators are central to the theme of the present paper. They are operators that preserve image and kernel of a homomorphism. Intuitively, they allow passage from one form to another without significant loss of information.   

\begin{definition}
An operator $\tau\colon F\to G$ between orean forms is said to be \emph{conormal} when the formula $\tau(\im ^F f)=\im ^G f$ holds. It is said to be \emph{normal} when, dually, the formula $\tau(\ker ^F g)=\ker ^G g$ holds and  \emph{binormal} when it is both conormal and normal. 
\end{definition}

\begin{remark}\label{RemN}
Note that a conormal operator preserves top clusters, since in any orean form, a top cluster is image of an identity morphism. Dually, a normal operator preserves bottom clusters.
\end{remark}

\begin{lemma}\label{LemAA}
Given a subform inclusion $F\to G$ of orean forms, the following conditions are equivalent:
\begin{itemize}
\item[(i)] All conormal $G$-clusters are $F$-clusters.

\item[(ii)] Images are computed in $F$ the same way as in $G$.

\item[(iii)] The inclusion $F\to G$ is conormal.

\item[(iv)] All conormal $G$-clusters are conormal $F$-clusters.
\end{itemize}
\end{lemma}

\begin{proof}
Suppose all conormal $G$-clusters are $F$-clusters. Then the top $G$-clusters must be the top $F$-clusters. Furthermore, direct images of those relative to $G$ must be $F$-clusters, which will force them to match with direct images relative to $F$. So (i)$\implies$(ii). The equivalence (ii)$\iff$(iii) is immediate from the definition of a conormal operator. The implication (iii)$\implies$(iv) is easy and (iv)$\implies$(i) is trivial.
\end{proof}

\begin{remark}\label{RemO}
We note that:
\begin{itemize}
    \item The classes of conormal/normal/binormal operators are each closed under composition.
    
    \item An isomorphism of orean forms preserves all structural components of an orean form: top/bottom clusters, meets/joins, and direct/inverse images of clusters. 
    
    \item Thus, an isomorphism of orean forms is a binormal operator.
    
    \item Also, the above gives that if two forms are isomorphic and one of them is orean, then so is the other one.
\end{itemize}
\end{remark}

\begin{lemma}\label{LemT}
Both product projections of a product of two orean forms are binormal.   
\end{lemma}

\begin{proof}
This follows from Remark~\ref{RemL}.  
\end{proof}

\begin{lemma}\label{LemN}
  Given two orean forms $F$ and $G$, if $F$ is conormal then there is at most one conormal operator $\alpha\colon F\to G$, and when $\alpha$ exists, it preserves direct images, i.e.: 
  $$\alpha(f\cdot^F A)=f\cdot^F \alpha(A).$$
  Dually, if $F$ is normal, then there is at most one normal operator $\beta\colon F\to G$, and when $\beta$ exists, it preserves inverse images.
  This implies that there is at most one isomorphism between two conormal orean forms, as well as between two normal orean forms. 
\end{lemma}
\begin{proof}
  When $F$ is conormal, any cluster is of the form $\im ^F g$ and so the formula $\alpha(\im ^F g)=\im ^G g$ can define at most one conormal operator $F\to G$. Furthermore, we have:
  \begin{align*}
    \alpha(f\cdot^F \im ^F g) &=\alpha(\im ^F fg)\\
    &=\im ^G fg\\
    &=f\cdot^G \im ^G g\\
    &=f\cdot^G \alpha(\im ^F g).
  \end{align*}
  The last statement of the lemma follows from the first two statements and Remark~\ref{RemO}.
\end{proof}

\begin{example}\label{ExaL}
  Let $F$ be the form of subsets of sets and let $G$ be the form of equivalence relations on sets. Each subset $S$ of a set $X$ determines an equivalence relation with singleton equivalence classes except the equivalence classes of elements of $S$ which are all equal to $S$. This correspondence determines a non-normal conormal operator $F\to G$. The operator $G\to F$ that maps an equivalence relation $E$ on a set $X$ to the empty subset $\varnothing$ of $X$ is normal, but not conormal.
\end{example}

\begin{lemma}\label{LemZ}
For an orean form $F$ and an operator $\kappa\colon F\to F$, the following conditions are equivalent: 
\begin{itemize}
    \item[(i)] $\kappa$ is conormal.
    
    \item[(ii)] Conormal clusters are $\kappa$-closed.
\end{itemize}
When, in addition, $F_\kappa$ is orean, the conditions above are further equivalent to:
\begin{itemize}
    \item[(iii)] The subform inclusion $F_\kappa\to F$ is  conormal.   
\end{itemize}
\end{lemma}

\begin{proof}
$\kappa$ is conormal if and only if for any morphism $f\colon X\to Y$, we have $\kappa(\mathsf{Im}f)=\mathsf{Im}f$. This is exactly the same as to say that conormal clusters are closed, so (i)$\iff$(ii). 
When $F_\kappa$ is orean, the equivalence (ii)$\iff$(iii) is given by Lemma~\ref{LemAA}.
\end{proof}

\begin{example}
The operators $\gamma,\lambda,\omega$ from Example~\ref{ExaX} are binormal. The operators $\delta$, $(\gamma,\lambda)$ and $\varepsilon$ from Example~\ref{ExaY} are binormal too, thanks to Remark~\ref{RemL}.
\end{example}

\begin{lemma}\label{LemU}
Consider two orean forms $F$ and $G$ over the same category $\mathbb{C}$, such that $G$ is a subform of $F$. The following conditions are equivalent:
\begin{itemize} 
\item[(i)] $G$ is a conormal form and the subform inclusion $G\to F$ is conormal.

\item[(ii)] Every $G$-cluster is a conormal $F$-cluster and the subform inclusion $G\to F$ is conormal.

\item[(iii)] $G$ is a subform of $F$ consisting of all conormal $F$-clusters.
\end{itemize}
When these conditions hold, if $\tau$ is a left inverse of the subform inclusion $G\to F$, then $\tau$ is conormal.
\end{lemma}

\begin{proof}
(i)$\implies$(ii): Suppose (i) holds. Since $G$ is conormal, every $G$-cluster is of the form $\im^G f$ and, since the subform inclusion $G\to F$ is conormal, we have $\im^G f=\im^F f$. Hence, every $G$-cluster is a conormal $F$-cluster.

(ii)$\implies$(iii): Suppose now (ii) holds. Every conormal $F$-cluster is of the form $\mathsf{Im}^F f$ for some morphism $f$. By conormality of the subform inclusion $G\to F$, we get $\im^G f=\im^F f$. This shows that every conormal $F$-cluster is a $G$-cluster. We then have (iii).

(iii)$\implies$(i): Suppose (iii) holds. Then all top $F$-clusters are $G$-clusters and so they must be top $G$-clusters. Then, also, direct images of those clusters relative to the form $F$ are $G$-clusters, and so, they must be direct images relative to $G$. This gives that the subform inclusion $G\to F$ must be conormal. Since $G$ has no other $F$-clusters except conormal ones, it also gives that every $G$-cluster is conormal relative to $G$. We then get (i).  

This last part of the lemma follows from (ii).
\end{proof}

\begin{lemma}\label{LemJ}
  Let $F$ and $G$ be orean forms. If there is an injective conormal operator $F\to G$, then for the $G$-image of a morphism $f\colon X \to Y$, its $G$-embedding  $\iota^G_{\im^G f}$ (when the latter exists) is at the same time an $F$-embedding for the $F$-image of $f$, i.e., $\iota^G_{\im^G f}=\iota^F_{\im^F f}$. The converse is also true provided the operator is full: for the $F$-image of a morphism $f\colon X \to Y$, its $F$-embedding $\iota^F_{\im^F f}$ (the latter exists) is at the same time a $G$-embedding for the $G$-image of $f$, i.e., $\iota^F_{\im^F f}=\iota^G_{\im^G f}$.  
\end{lemma}
\begin{proof} 
  This can be proved by a direct use of the definition of an embedding. We give the argument only for the first part of the lemma. Suppose there is a conormal operator $\alpha\colon F\to G$. Let $m$ be a $G$-embedding of $\im^G f$. Since $\alpha(\im^F m)=\im^G m=\im^G f=\alpha(\im^F f)$ and the operator $\alpha$ is injective, we have $\im^F m=\im^F f$. Suppose now $\im^F m\ge^F_Y\im^F m'$. Then $\im^G m=\alpha(\im^F m)\ge^G_Y\alpha(\im^F m')=\im^G m'$ and so $m'=mu$ for a unique morphism $u$. 
\end{proof}

\begin{theorem}\label{ThmB}
  Let $F$ be an orean form and let $\kappa$ be a binormal idempotent closure operator on $F$. Then:
  \begin{itemize}
    \item[(i)] If $F$ satisfies (N3) then $F_\kappa$ satisfies (N3).

    \item[(ii)] $F$ satisfies (N2) if and only if $F_\kappa$ satisfies (N2).

    \item[(iii)] $F_\kappa$ satisfies (N1) if and only if the following implications hold for $F$:
    \begin{align*} 
      S=\kappa S\ge \ker  f\quad&\implies\quad\kappa(f\cdot S)\cdot f =S,\\
      \im  f\ge T=\kappa T\quad&\implies\quad\kappa(f\cdot(T\cdot f))=T.
    \end{align*} 

    \item[(iv)] If $F$ satisfies (N1), for $F_\kappa$ to satisfy (N1) it is necessary and sufficient for the form $F$, direct image under each morphism $f\colon X\to Y$ of a closed $S\ge_X \ker f$ to be closed.
  \end{itemize}
\end{theorem}
\begin{proof}
  Let $F$ be an orean form and let $\kappa$ be a binormal idempotent closure operator on $F$. By Theorem~\ref{ThmA}, the form $F_\kappa$ is an orean form. By Lemmas~\ref{LemZ} and \ref{LemAA} and their duals, we have:
  \begin{itemize}
  \item Conormal and normal clusters are closed.
  \item Images and kernels in the form $F_\kappa$ are computed the same way as in $F$.
  \item The subform inclusion $F_\kappa\to F$ is binormal.
  
  \item Conormal and normal $F_\kappa$-clusters are the same as conormal and normal $F$-clusters, respectively.
  \end{itemize}
  With Lemma~\ref{LemJ} we then immediately get (ii). We also get (i) easily. For the first part of (i), note that if join of two normal clusters is normal, then their $F_\kappa$-join will be forced to match with their $F$-join. We can get the second part dually, or even more easily by the fact that meet in $F_\kappa$ is computed the same way as in $F$ by Theorem~\ref{ThmA}.
  
  The implications in (iii) are the implications of Remark~\ref{RemB} applied to the form $F_\kappa$ and written out for the form $F$ (based on the bullet points above and Theorem~\ref{ThmA}). It remains to prove (iv). Suppose both $F$ and $F_\kappa$ satisfy (N1). Then by (iii), for a closed $S\ge \ker f$, we have:
    \[\kappa(f\cdot S)\cdot f =S.\]
  Since $\im f\ge f\cdot S$ and $\im f$ is closed (by the first bullet-point above), also $\im f\ge \kappa(f\cdot S)$. Then, thanks to Remark~\ref{RemB} applied to $F$, 
    \[\kappa(f\cdot S)=f\cdot (\kappa(f\cdot S)\cdot f)=f\cdot S.\]
  Conversely, suppose $F$ satisfies (N1) and for the form $F$, direct image under each morphism $f\colon X\to Y$ of a closed $S\ge_X \ker f$ is closed. Then the second implication under (iii) holds (thanks to Remark~\ref{RemB} and Theorem~\ref{ThmA}). If $S\ge \ker f$ where $S$ is closed, then, thanks to Remark~\ref{RemB},
    \[\kappa(f\cdot S)\cdot f=(f\cdot S)\cdot f=S,\]
  which proves the first implication under (iii). Therefore, by (iii), $F_\kappa$ satisfied (N1). 
\end{proof}

\begin{example}\label{Ex1}
  Additive subgroups of rings constitute a noetherian form $F$ over the category $\mathbf{Rng}$ of rings without identity. Conormal clusters for this form are subrings, while normal clusters are ideals. Generating a subring from an additive subgroup of a ring is an idempotent closure operator $\kappa$ on this form. Since ideals are subrings, we can apply Theorem~\ref{ThmB} to conclude that $F_\c$, the form of subrings, is noetherian --- it is in fact isomorphic to the form of subobjects of the semi-abelian category  $\mathbf{Rng}$ (see Example~\ref{Ex0}).  
\end{example}

\begin{theorem}
If a form $F$ is isomorphic to a noetherian form, then $F$ itself is noetherian.
\end{theorem}

\begin{proof}
This follows from Remark~\ref{RemO} and Lemma~\ref{LemJ}.  
\end{proof}

\begin{lemma}\label{LemA}
  Let $F$ be an orean form and let $\kappa$ be an idempotent closure operator on $F$. The operator 
    \[F\to F_\kappa,\quad S\mapsto \kappa S\]
  is conormal; it is binormal if and only if in the form $F$, taking inverse image of a bottom cluster permutes with closure: $\kappa(\bot^F\cdot^F f)=\kappa(\bot^F)\cdot^F f$. In particular, this is the case when $\kappa$ is normal.
\end{lemma}
\begin{proof} We use here Theorem~\ref{ThmA}.
  The first part of the lemma is established by the following computation: 
    \[\kappa(\im ^F f)
    =\kappa(f\cdot^F\top^F)
    =\kappa(f\cdot^F\kappa\top^F)=f\cdot^\kappa\kappa\top^F
    =f\cdot^\kappa\top^\kappa=\im ^\kappa f.\]
  The second part is obtained  from the following equalities:
    \[\kappa(\ker ^F f)=\kappa(\bot^F\cdot^F f),\quad\ker ^\kappa f=\bot^\kappa\cdot^\kappa f=\kappa(\bot^F)\cdot^\kappa f=\kappa(\bot^F)\cdot^F f.\]
    The third part of the lemma follows from the second part after noting that when $\kappa$ is normal, we have: $\kappa(\bot^F\cdot^F f)=\kappa(\mathsf{Ker}^F f)=\mathsf{Ker}^F f=\bot^F\cdot^F f=\kappa(\bot^F)\cdot^F f$. 
\end{proof}

\begin{remark}
An idempotent closure operator $\kappa$ on an orean form $F$, has two other operators associated to it: the subform inclusion $\tau\colon F_\kappa\to F$ and the operator $\sigma\colon F\to F_\kappa$ described in Lemma~\ref{LemA}. By Theorem~\ref{ThmA},  Lemma~\ref{LemZ} (together with its dual) and Lemma~\ref{LemA}, we have:
\begin{itemize}
    \item $\sigma$ is always conormal. It is normal if and only if the following holds: $$\kappa(\mathsf{Ker}^F (fg))=\kappa(\mathsf{Ker}^F f)\cdot^F g.$$
    
    \item $\kappa$ is conormal if and only if so is $\tau$, and if and only if conormal $F$-clusters are $\kappa$-closed.
    
    \item $\kappa$ is normal if and only if so is $\tau$, and if and only if normal $F$-clusters are $\kappa$-closed.
    
    \item If $\kappa$ is normal, then so is $\sigma$.
    
    \item All three operators $\kappa,\tau,\sigma$ are binormal if and only if $\kappa$ is binormal and if and only if conormal $F$-clusters as well as normal $F$-clusters are all $\kappa$-closed. 
\end{itemize}
\end{remark}

\subsection*{Antinormality and Isoforms} Here we spend some time on useful trivialities. We omit all proofs since they are rather elementary/straightforward.

\begin{definition}\label{DefE}
For an orean form $F$, we write $F^\bot$ to denote its subform of bottom clusters and we write $F^\top$ to denote its subform of top clusters. $F$ is said to be 
\begin{itemize}
    \item \emph{antinormal} when $F^\bot$ is its subform of normal clusters,
    
    \item \emph{anticonormal} when $F^\top$ is its subform of conormal clusters,
    
    \item \emph{antibinormal} when it is both antinormal and anticonormal,
    
    \item an \emph{isoform} when, as a functor, it is an isomorphism.
\end{itemize}
\end{definition}

\begin{remark}
Any isomorphism of categories is an orean form and hence an isoform.
\end{remark}

\begin{example}\label{ExaZ}
The form of subsets of sets is an example of a conormal antinormal form. The form of palettes of finite sets from Example~\ref{ExaX} is antinormal, but not conormal. An identity functor is an isoform. Every isoform is antibinormal. A form over a category $\mathbb{C}$ that is not an isoform, but is antibinormal, can be obtained by setting clusters in an object $X$ to be elements of some fixed bounded lattice $P$, with $y \ge_f x$ whenever $x\le y$ in $P$. This is an antibinormal orean form, where $f\cdot x=x$ and $y\cdot f=y$. It is an isoform if and only if $P$ is a lattice (i.e., it has only one element).  
\end{example}

\begin{lemma}\label{LemAB}
For an orean form $F$ the following conditions are equivalent:
\begin{itemize}
    \item $F$ is antinormal.
    
    \item For any orean form $G$ there exists a normal operator $G\to F$.
    
    \item There exist an isoform $G$ and a normal operator $G\to F$.
    
    \item There exist an antinormal form $G$ and a normal operator $G\to F$.
\end{itemize}
\end{lemma}

\begin{lemma}\label{LemX} For a form $F$ over a category $\mathbb{C}$, the following conditions are equivalent:
\begin{itemize}
    \item $F$ is an isoform.
    
    \item $F$ is an orean form whose every top $F$-cluster is a bottom $F$-cluster (i.e., $F^\top=F_\bot$).
    
    \item $F$ is an orean form such that in every object there is a unique cluster.
    
    \item $F$ is an orean form that is both conormal and anticonormal.
    
    \item $F$ is an orean form that is both normal and antinormal.
    
    \item $F$ is an orean form that is both binormal and antibinormal.
    
    \item $F$ is the terminal object in the category of forms (over $\mathbb{C}$): for any orean form $G$ over $\mathbb{C}$ there is a unique operator $G\to F$.
    
    \item For any orean form $G$ over $\mathbb{C}$ there is a binormal operator $G\to F$.
\end{itemize}
\end{lemma}

\begin{lemma}\label{LemR} There is a one-to-one correspondence between subforms of an orean form $F\colon \mathbb{B}\to \mathbb{C}$ that are isoforms and right inverses $G$ of the functor $F$. The subform of $F$ corresponding to $G$ is given by the values of $G$. Both $F^\bot$ and $F^\top$ are isoforms, and so they are both binormal and antibinormal. The corresponding right inverses of $F$ are a right adjoint and a left adjoint of $F$, respectively.
\end{lemma}

\begin{lemma}\label{LemS} For orean forms $F$ and $G$, if $G$ is an isoform, then the first projection of the form $F\times G$ and the second projection of the form $G\times F$ are both isomorphisms. Moreover, in each case, the converse is also true: if either the first projection of $F\times G$ or the second projection of $G\times F$ is an isomorphism, then $G$ is an isoform. 
\end{lemma}

\section{Orean Factorizations}\label{secH}

\subsection*{Definition and basic facts} Every noetherian form gives rise to a proper factorization system of the underlying category. We develop here a form-theoretic language for working with proper factorization systems. The resulting notion is that of an `orean factorization'. It corresponds to a proper factorization system which allows for various natural constructions not included in the definition of the latter. 

As it follows from Theorem 2.4 in \cite{JW14} (where an `embedding' is called a `left universalizer'), conormal orean forms admitting embeddings are uniquely determined (up to an isomorphism) by the classes $\mathcal{M}$ of their embeddings. A conormal orean form can be recovered (up to an isomorphism) from its class $\mathcal{M}$ of embeddings as the form of $\mathcal{M}$-subobjects. As attested by Corollary~2.5 in \cite{JW14} and the discussion preceding it, for such form the following conditions are equivalent (with $\mathcal{M}$ denoting the class of its embeddings):
\begin{itemize}
  \item $\mathcal{M}$ is closed under composition.
  \item $\mathcal{M}$ is part of a factorization system $(\mathcal{E},\mathcal{M})$ in the sense of \cite{FK72}.
  \item for any $m\in\mathcal{M}$, the formula $(m\cdot A)\cdot m=A$ holds.
\end{itemize}
Furthermore, in the $(\mathcal{E},\mathcal{M})$-factorization $f=me$ of a morphism, $m=\iota_{\im f}$. 

Dually, normal orean forms admitting quotients are uniquely determined (up to an isomorphism) by their classes $\mathcal{E}$ of quotients (as forms of $\mathcal{E}$-quotients) and we have equivalence of the following conditions:
\begin{itemize}
  \item $\mathcal{E}$ is closed under composition;
  \item $\mathcal{E}$ is part of a proper factorization system $(\mathcal{E},\mathcal{M})$;
  \item for any $e\in\mathcal{E}$, the formula $e\cdot (E\cdot e)=E$ holds;
\end{itemize}
and in the $(\mathcal{E},\mathcal{M})$-factorization $f=me$ of a morphism, $e=\pi_{\mathsf{Ker} f}$. It then follows that given a conormal orean form $F_\s $ admitting embeddings and a normal orean form $F_\e $ admitting quotients, requiring that the class $\mathcal{M}$ of $F_\s$-embeddings and the class $\mathcal{E}$ of $F_\e$-quotients are part of the same factorization system $(\mathcal{E},\mathcal{M})$ implies that any morphism $f$ decomposes as $f=\iota^\s_{\im^\s f}\circ \pi^\e_{\ker^\e f}$. This brings us to the following notion:

\begin{definition}\label{DefD}
  An \emph{orean factorization} of a category $\mathbb{C}$ is an ordered pair $(F_\s,F_\e)$ of forms over $\mathbb{C}$, where $F_\s$ is a conormal orean form and $F_\e$ is a normal orean form, such that every morphism $f$ decomposes as $f=\iota^\s_{\im^\s f}\circ \pi^\e_{\ker^\e f}$. Two orean factorizations $(F_\s,F_\e)$ and $(F'_\s,F'_\e)$ are said to be \emph{component-wise isomorphic} when there are isomorphisms $F_\s\approx F'_\s$ and $F_\e\approx F'_\e$.
\end{definition}

\begin{theorem}\label{ThmS}
  If $(F_\s,F_\e)$ is an orean factorization, then any pair $(G,H)$ of forms such that there are isomorphisms $F_\s\approx G$ and $F_\e\approx H$, is an orean factorization.
\end{theorem}
\begin{proof}
  If there are such isomorphisms, then both $G$ and $H$ are necessarily orean forms. Moreover, it follows easily from Rem~\ref{RemO} that $G$ is a conormal orean form and $H$ is a normal orean form. The rest follows from Lemma~\ref{LemJ}.   
\end{proof}

\begin{remark}
Note that we have duality for orean factorizations: $(F_\s,F_\e)$ is an orean factorization of a category $\mathbb{C}$ if and only if $(F_\e^\mathsf{op},F_\s^\mathsf{op})$ is an orean factorization of $\mathbb{C}^\mathsf{op}$.
\end{remark}

\begin{lemma}\label{LemK}
  Let $(F_\s,F_\e)$ be an orean factorization. A morphism is an $F_\s$-embedding if and only if its $F_\e$-kernel is a bottom cluster, and dually, a morphism is an $F_\e$-quotient if and only if its $F_\s$-image is a top cluster. Moreover, composite $\iota^\s_A\circ \iota^\s_B$ of two $F_\s$-embeddings is an $F_\s$-embedding and $\iota^\s_A\circ \iota^\s_B=\iota^\s_{\iota^\s_A\cdot^\s B}$, and dually, composite $\pi^\e_S\circ \pi^\e_R$ of two $F_\e$-quotients is an $F_\e$-quotient and $\pi^\e_S\circ \pi^\e_R=\pi^\e_{S\cdot^\e \pi^\e_R}$. 
  Finally, the $F_\s$-kernel of an $F_\s$-embedding $f$ is a bottom cluster and  (N1) holds for $f$, relative to the form $F_\s$, and dually, the $F_\e$-image of any $F_\e$-quotient $f$ is a top cluster and (N1) holds for any $F_\e$-quotient $f$, relative to the form $F_\e$.
\end{lemma}
\begin{proof}
  Suppose the $F_\e$-kernel of a morphism $f\colon X\to Y$ is a bottom cluster. Factorise $f$ as $f=me$, where $e$ is an $F_\e$-quotient of $\ker^\e f$ and $m$ is an $F_\s$-embedding of $\im^\s f$. If $\ker^\e f$ is a bottom cluster, then $e$ is an isomorphism and hence $f$ is an $F_\s$-embedding. If $f$ is an $F_\s$-embedding, then it is an $F_\s$-embedding of $\im^\s f$ and so the morphism $e$ in a similar factorization as before is an isomorphism. It then follows that $\ker^\e f$ is a bottom cluster.
  
  The second part of the lemma is an easy consequence of the first part.
  
  We prove the third part of the lemma. The join formula of (N1) for an $F_\s$-embedding $f$ is easily implied by the formula $(f\cdot^\s S)\cdot^\s f=S$. We prove this latter formula. Let $m$ denote an embedding of $S$. Then $f\cdot^\s S$ is the image of the composite $fm$. It follows from the second part of the lemma that $fm$ is an embedding. So a pullback of $fm$ along $f$ gives an embedding for $(f\cdot^\s S)\cdot^\s f$, by Lemma~\ref{LemP}. But since $f$ is a monomorphism (Remark~\ref{RemJ}), $m$ is such pullback. Hence $(f\cdot^\s S)\cdot^\s f=\im  ^\s m=S$. To prove the meet formula, consider a conormal cluster $T$ in the codomain $Y$ of $f$. Let $m$ denote an embedding of $T$ and let $m'$ be a pullback of $m$ along $f$. Thanks to Lemma~\ref{LemP} again, $m'$ is an embedding of $T\cdot^\s f$. The composite $fm'$ is an embedding of $\im^\s(fm')= f\cdot^\s (T\cdot^\s f)$, by the second part of the lemma. At the same time, it can be verified by an easy routine that $\im^\s(fm')=\im^\s f\wedge^\s T$. So $f\cdot^\s (T\cdot^\s f)=\mathsf{Im}^\s f\wedge^\s T$, as desired.
\end{proof}

\begin{remark}
Recall that a proper factorization system in the sense of \cite{FK72,I58} is a factorization system $(\mathcal{E},\mathcal{M})$ where $\mathcal{E}$ is a class of epimorphisms and $\mathcal{M}$ is a class of monomorphisms. The first part of the lemma above will be used to establish the following theorem, which essentially states that orean factorizations are conceptually the same as proper factorization systems. The second part of the lemma above could be concluded as a corollary of this theorem, and the third bullet point at the start of this section.
\end{remark}

\begin{theorem}\label{ThmV}
For a category $\mathbb{C}$, there is a bijection between classes of component-wise isomorphic orean factorizations of $\mathbb{C}$ and proper factorization systems $(\mathcal{E},\mathcal{M})$ in $\mathbb{C}$ such that pushouts of morphisms along those in  $\mathcal{E}$ and pullbacks of morphisms along those in $\mathcal{M}$, as well as finite meets and joins of $\mathcal{E}$-quotients and $\mathcal{M}$-subobjects exist. For a factorization system $(\mathcal{E},\mathcal{M})$ and an orean factorization $(F_\s,F_\e)$ from the corresponding class of component-wise isomorphic orean factorizations, we have:
\begin{itemize} 
\item[(i)] $F_\s$ is isomorphic to the form of $\mathcal{M}$-subobjects and  $\mathcal{M}$ is the class of $F_\s$-embeddings.  

\item[(ii)] $F_\e$ is isomorphic to the form of $\mathcal{E}$-quotients and  $\mathcal{E}$ is the class of $F_\e$-quotients.  
\end{itemize}
\end{theorem}

\begin{proof}
The following parts of the theorem are straightforward consequences of the results contained in \cite{JW14}:
\begin{itemize}
    \item Any proper factorization system $(\mathcal{E},\mathcal{M})$ in the sense of \cite{FK72,I58}, such that pushouts of morphisms along those in  $\mathcal{E}$ and pullbacks of morphisms along those in $\mathcal{M}$, as well as finite meets and joins of $\mathcal{E}$-quotients and $\mathcal{M}$-subobjects exist, determines an orean factorization $(F_\s,F_\e)$ of the underlying category for which (i) and (ii) hold.
    
    \item Given an orean factorization $(F_\s,F_\e)$, if we set $\mathcal{E}$ to be the class of $F_\e$-quotients and $\mathcal{M}$ the class of $F_\s$-embeddings, then pushouts of morphisms along those in $\mathcal{E}$ and pullbacks of morphisms along those in $\mathcal{M}$, as well as finite meets and joins of $\mathcal{E}$-quotients and $\mathcal{M}$-subobjects will exist.
\end{itemize}
By Remark~\ref{RemO} and Lemma~\ref{LemJ}, isomorphic orean forms will have the same embeddings and quotients. This, together with the first bullet point above, gives us that there is a one-to-one mapping assigning to each proper factorization system $(\mathcal{E},\mathcal{M})$ the class of component-wise isomorphic orean factorizations satisfying (i) and (ii). 

To establish the bijection in the theorem, it remains to show that the $\mathcal{E}$ and the $\mathcal{M}$ in the second bullet point above constitute a proper factorization system $(\mathcal{E},\mathcal{M})$. We know that every morphism in $\mathcal{E}$ must be an epimorphism and that every morphism in $\mathcal{M}$ must be a monomorphism (Remark~\ref{RemJ}). It follows from Lemma~\ref{LemK} that both $\mathcal{E}$ and $\mathcal{M}$ are closed under composition. Therefore, $\mathcal{M}$ is part of some factorization system $(\mathcal{E}',\mathcal{M})$ (see the first paragraph of this section). We want to establish that $\mathcal{E}'=\mathcal{E}$. Let $e'\in \mathcal{E}'$. Then $e'$ decomposes as $e'=me$, where $e\in \mathcal{E}$ and $m\in \mathcal{M}$, by the definition of an orean factorization. This implies that $m\in \mathcal{E}'$ and hence $m$ is an isomorphism, which in turn implies $e'\in\mathcal{E}$. Conversely, let $e\in \mathcal{E}$. Consider a decomposition $e=me'$ where $m\in \mathcal{M}$ and $e'\in\mathcal{E}'$. From the first paragraph of this section we know that $m$ is an $F_\s$-embedding of the $F_\s$-image of $e$. By Lemma~\ref{LemK}, the $F_\s$-image of $e$ is a top cluster. This implies that $m$ is an isomorphism. Hence $e\in \mathcal{E}'$.  
\end{proof}

\begin{example}\label{ExaJ}
Given a proper factorization system $(\mathcal{E},\mathcal{M})$ on a category, the requirements on the classes $\mathcal{E}$ and $\mathcal{M}$ from Theorem~\ref{ThmV} are automatically fulfilled once all finite limits and finite colimits exist in the category. Thus, orean factorizations are quite common for categories of mathematical structures. The category of topological spaces, for instance, has two natural orean factorizations: in the first one, $F_\s$ is the form of subspaces --- with $\mathcal{E}$ the class of epimorphisms, while in the second one, $F_\e$ is the \emph{form of quotient spaces} --- with $\mathcal{M}$ the class of monomorphisms. The form $F_\s$ of subspaces witnesses that the implication in the second part of Lemma~\ref{LemK} cannot be reversed: morphisms $f$ such that the $F_\s$-kernel of $f$ is a bottom cluster and the join formula of (N1) holds for $f$ are the injective continuous functions (i.e., the monomorphisms in the category of topological spaces), which is a wider class than the class of $F_\s$-embeddings. An injective continuous function $f$ is an $F_\s$-embedding if and only if every open set in its domain is an inverse image of some open set in its codomain.
\end{example}

\begin{remark}\label{RemH} 
  As it follows from basic facts about factorization systems, in the description of orean factorizations in terms of pairs $(\mathcal{E},\mathcal{M})$ given in Theorem~\ref{ThmV}, the requirements on the pair $(\mathcal{E},\mathcal{M})$ can be equivalently changed to: $\mathcal{M}$ is a class of monomorphisms that is closed under composition and contains all isomorphisms, while $\mathcal{E}$ is a class of epimorphisms having the same two properties, such that the forms of $\mathcal{M}$-subobjects and $\mathcal{E}$-quotients are both orean and any morphism $f$ in $\mathbb{C}$ factorizes uniquely (up to an isomorphism) as $f=me$ with $e\in\mathcal{E}$ and $m\in\mathcal{M}$.
\end{remark}

\begin{theorem}\label{ThmAE}
Let $(F_\s,F_\e)$ be an orean factorization of a category $\mathbb{C}$. With $F=F_\s$ in (N1), we have: (N1) holds for all morphisms $f$ as soon as it holds for all quotients. 
\end{theorem}

\begin{proof}
Assume (N1) holds for all $F_\e$-quotients. Consider a morphism $f\colon X\to Y$ and decompose it as $f=me$, where $e$ is a quotient of its $F_\e$-kernel and $m$ is an embedding of its $F_\s$-image. Thanks to Lemma~\ref{LemK} and that (N1) holds for $e$ by our assumption, given an $F_\s$-cluster $B$ in $Y$, we have: \[f\cdot^\s (B\cdot^\s f)=m\cdot^\s (e\cdot^\s ((B\cdot^\s m)\cdot^\s e))=m\cdot^\s (B\cdot^\s m)=\im^\s m\land^\s B=\im^\s f\land^\s B.\]
Similarly, given an $F_\s$-cluster $A$ in $X$, we have:
\[(f\cdot^\s A)\cdot^\s f=((m\cdot^\s (e\cdot^\s A))\cdot^\s m)\cdot^\s e=(e\cdot^\s A)\cdot^\s e=\ker^\s e\lor^\s A=\ker^\s f\lor^\s A.\qedhere\]
\end{proof}

\begin{theorem}\label{ThmT}
  For a strongly orean form $F$, the subform inclusion $F_\c\to F$ is conormal and the subform inclusion $F_\n\to F$ is normal. Consequently, a morphism is an $F$-embedding/quotient of a conormal/normal $F$-cluster if and only if it is an $F_\c$-embedding/$F_\n$-quotient. This furthermore gives that (N2) holds for $F$ if and only if $(F_\c,F_\n)$ is an orean factorization.
\end{theorem}
\begin{proof}
  The first part of the lemma follows immediately from Lemma~ \ref{LemM} and its dual. To get the rest, apply Lemma~\ref{LemJ}.
\end{proof}

\subsection*{Semiexact pairs} Next, we explore a generalization of kernel-cokernel correspondence, in a very general form-theoretic setting of which every orean factorization will turn out to be an example. 

\begin{definition}
  A \emph{semiexact pair} is a pair $(F_\s ,F_\e )$, where $F_\s$ is a conormal orean form and $F_\e$ is a normal orean form, such that there exists a conormal operator $\alpha\colon F_\s \to F_\e $ and a normal operator $\beta\colon F_\e \to F_\s $. A \emph{left/right exact pair} is a semiexact pair such that $\alpha$ is a left/right inverse of $\beta$. A \emph{biexact pair} is a semiexact pair that is both a left and a right exact pair. 
\end{definition}

\begin{lemma}\label{LemAH}
For a left exact pair $(F,F_\e)$, the normal operator $\beta\colon F_\e\to F$ restricts to an isomorphism $F_\e\to F_\n$. Consequently, the conormal operator $\alpha\colon F\to F_\e$ is normal and the pair $(F,F_\n)$ is left exact as well.
\end{lemma}

\begin{proof}
Suppose $(F,F_\e)$ is a left exact pair. Then $F$ is conormal and $F_\e$ is normal. Moreover, the conormal operator $\alpha\colon F\to F_\e$ is a left inverse of the normal operator $\beta\colon F_\e\to F$. Then $\alpha$ is normal and moreover, $\beta$ restricts to an isomorphism between $F_\e$ and $F_\n$. That $(F,F_\n)$ is left exact as well now follows from Remark~\ref{RemO}.
\end{proof}

\begin{lemma}
  \label{LemW}
  Let $G$ be an isoform. For a conormal form $F$, the following conditions are equivalent:
\begin{enumerate}[label=(\roman*)]
  \item $F$ is antinormal.    
  \item $(F,G)$ is a semiexact pair.    
  \item There exists a normal operator $G\to F$.
\end{enumerate}
\end{lemma}
\begin{proof}
(i)$\implies$(ii) follows from Lemma~\ref{LemAB} and its dual. (ii)$\implies$(iii) is trivial. (iii)$\implies$(i) follows again from Lemma~\ref{LemAB}.
\end{proof}

\begin{theorem}\label{ThmF}
  Any orean factorization $(F_\s,F_\e)$ of a category $\mathbb{C}$ is a semiexact pair, with the conormal operator $\alpha\colon F_\s\to F_\e$ and the normal operator $\beta\colon F_\e\to F_\s$ defined by $\alpha(A)=\im^\e \iota^\s_A$ and $\beta(R)=\ker^\s \pi^\e_R$, respectively.
\end{theorem}
\begin{proof}
  We first prove that any orean factorization $(F_\s,F_\e)$ of a category $\mathbb{C}$ is a semiexact pair. For $A\subseteq_\s X $, define $\alpha(A)$ to be the $F_\e $-image of an $F_\s $-embedding of $A$. This definition does not depend on the choice of an $F_\s $-embedding, since:
  \begin{itemize}
    \item by the universal property of an $F_\s $-embedding, it is unique up to an isomorphism,

    \item top clusters are preserved under isomorphisms (this follows from the fact that in an orean form, direct image along an isomorphism is the same as inverse image along its inverse). 
  \end{itemize}
   $1_X$ is an $F_\s $-embedding of the top cluster in $X$, and since direct image along $1_X$ does not change a cluster, we can conclude that $\alpha$ preserves top clusters. 
\end{proof}
\begin{remark}
Note that in the theorem above, the formula $\alpha(A)=\im^\e\iota^\s_A$ is forced by conormality of $\alpha$: $\alpha(A)=\alpha(\im^\s\iota^\s_A)=\im^\e\iota^\s_A$. When using this formula, we may interchangeably refer either directly to \ref{ThmF} or re-derive it from conormality of $\alpha$. Similar remarks apply to the dual formula for $\beta$. 
\end{remark}

\begin{convention}
Thanks to the theorem above, for each orean factorization $(F_\s,F_\e)$ there is a conormal operator $F_\s\to F_\e$ and a normal operator $F_\e\to F_\s$. We will write $\alpha$ and $\beta$, respectively, to denote these operators. When necessary, we may also write $\alpha^F$ and $\beta^F$ instead of $\alpha$ and $\beta$. Note that by Lemma~\ref{LemN}, $\alpha$ and $\beta$ are uniquely determined by an orean factorization $(F_\s,F_\e)$.  
\end{convention}

\begin{remark}\label{RemV} In the theorem above, when $\mathbb{C}$ has a zero object $Z$ (i.e.,~when $\mathbb{C}$ is a pointed category), for the factorization system $(\mathcal{E},\mathcal{M})$ corresponding to the given orean factorization, the maps $\mathcal{M}\to \mathcal{E}$ and $\mathcal{E}\to \mathcal{M}$ corresponding to the conormal operator $\alpha\colon F_\s \to F_\e $ and the normal operator $\beta\colon F_\e \to F_\s $ are given by $Z$-cokernel and $Z$-kernel (that is, the usual categorical cokernel and kernel in a pointed category), respectively. To prove this, let $A\subseteq_\s Y$. Consider the $F_\s $-embedding $m\colon X\to Y$ of $A$ and let $Z$ be a zero object. The unique morphism $X\to Z$ is a split epimorphism, and hence it must be an $F_\e $-quotient (this is thanks to the fact that $(F_\s,F_\e)$ is an orean factorization). Moreover, $Z$ has only one cluster for the form $F_\e $ (since $Z$ cannot have nontrivial $F_\e $-quotients, as its any $F_\e $-quotient will be a slit monomorphism and hence an isomorphism), which implies that $X\to Z$ is the $F_\e $-quotient of the top cluster $\top^\e $ in $X$. We have $\alpha(A)=m\cdot^\e  \top^\e $. By Lemma~\ref{LemP}, the $F_\e $-quotient of $m\cdot^\e  \top^\e $ is obtained by a pushout of $X\to Z$ along $m$. Such pushout is exactly the cokernel $c$ of $m$. This proves that the map $\mathcal{M}\to\mathcal{E}$ corresponding to $\alpha$ is given by $Z$-cokernels.  Dually, we have that the map $\mathcal{E}\to\mathcal{M}$ corresponding to $\beta$ sends each morphism to its $Z$-kernel.
\end{remark}

\begin{example}\label{ExaK} 
The pair of operators described in Example~\ref{ExaL} is an example of one coming out of Theorem~\ref{ThmF} in the case when a zero object is absent. The orean factorization in this case is the one corresponding to the usual factorization system $(\mathcal{E},\mathcal{M})$ of the category of sets, where $\mathcal{E}$ is the class of epimorphisms and $\mathcal{M}$ is the class of monomorphisms.
\end{example}

\begin{example}
Over the category of topological spaces, let $F_\s$ be the form of subspaces of topological spaces and let $F_\e$ be the form of quotient spaces of topological spaces. Then $(F_\s,F_\e)$ is a semiexact pair, but not an orean factorization. The easiest way to break the factorization property of Definition~\ref{DefD}, is to consider an identity map $f\colon X\to Y$ between two topological spaces $X$ and $Y$ having identical underlying sets, such that $X$ has more open sets than $Y$. The fact that $(F_\s,F_\e)$ is a semiexact pair is an easy consequence of a similar fact for the category of sets. 
\end{example}

\begin{remark}
  The following natural question arises: when are the operators of Theorem~\ref{ThmF} binormal? It is not difficult to verify that in the context of a pointed category, the conormal operator $F_\s\to F_\e$ is normal if and only if every morphism in the class $\mathcal{E}$ is a cokernel, and dually, $F_\e\to F_\s$ is conormal if and only if every morphism in the class $\mathcal{M}$ is a kernel. From this it follows that for a pointed category these operators are binormal if and only if the category is Puppe-Mitchell exact \cite{P62,M65}; in this case $\mathcal{E}$ coincides with the class of epimorphisms and $\mathcal{M}$ coincides with the class of monomorphisms. In the general case (i.e., when the category is not necessarily pointed), if both $\alpha$ and $\beta$ are binormal, then so will be their composites. By Lemma~\ref{LemN}, these composites must be identity operators, and so we get that $F_\s$ is isomorphic to $F_\e$. The converse is true too; if $F_\s$ is isomorphic to $F_\e$, then the isomorphism will necessary be binormal (with binormal inverse) by Remark~\ref{RemO}, and furthermore, the isomorphism and its inverse then must be given by $\alpha$ and $\beta$, by Lemma~\ref{LemN}. The fact that the form of normal subobjects in a Grandis exact category \cite{G84,G92,G12,G13} is orean (moreover, it is noetherian, by Remark~\ref{RemF}), along with Theorem~5.1 of \cite{JW16}, shows that an orean factorization for which $\alpha$ and $\beta$ are binormal is the same as an orean factorization corresponding to a factorization system $(\mathcal{E},\mathcal{M})$ of a Grandis exact category $(\mathbb{C},\mathcal{N})$, where $\mathcal{E}$ is the class of cokernels relative to $\mathcal{N}$ and $\mathcal{M}$ is the class of kernels relative to $\mathcal{N}$. Note that, as recalled in the Introduction, in the pointed case, a Grandis exact category becomes a Puppe-Mitchell exact category, so we recover fully the earlier result.
\end{remark}

\begin{remark}\label{RemM} The following statements are easy to verify:
\begin{itemize}
\item In an orean factorization $(F_\s,F_\e)$ of a category $\mathbb{C}$, the form $F_\e$ is an isoform if and only if in the corresponding factorization system $(\mathcal{E},\mathcal{M})$ (see Theorem~\ref{ThmV}), the class $\mathcal{E}$ is the class of all isomorphisms. 

\item In this case, $\mathcal{M}$ is the class of all morphisms and every morphism in the category is a monomorphism. 

\item Such orean factorization $(F_\s,F_\e)$ exists if and only if every morphism is a monomorphism and in addition, finite joins and meets of subobjects exist in the category. 

\item Finally, a category $\mathbb{C}$ is a groupoid if and only if it has an orean factorization $(F_\s,F_\e)$ with both $F_\e$ and $F_\s$ being isoforms.  
\end{itemize}
\end{remark}

\begin{lemma}\label{LemAD}
For an orean factorization $(F_\s,F_\e)$, the conormal operator $\alpha$ preserves bottom clusters if and only if for every object $X$, we have: $\bot^\s_X=\top^\s_X$ implies $\bot^\e_X=\top^\e_X$. 
\end{lemma}

\begin{proof}
Suppose $\alpha$ preserves bottom clusters. Let $X$ be an object such that $\bot^\s_X=\top^\s_X$. Then
\begin{align*}
    \top^\e_X &= 1_X\cdot^\e \top^\e_X & \textrm{(Lemma~\ref{LemQ})}\\
    &= \iota^\s_{\top^\s_X}\cdot^\e \top^\e_X & \textrm{(Lemma~\ref{LemV})}\\
    &= \iota^\s_{\bot^\s_X}\cdot^\e \top^\e_X\\
    &=\im^\e \iota^\s_{\bot^\s_X}\\
    &=\alpha(\bot^\s_X) & \textrm{(Theorem~\ref{ThmF})}\\
    &=\bot^\e_X
\end{align*}
Conversely, suppose $\bot^\s_X=\top^\s_X$ always implies $\bot^\e_X=\top^\e_X$. Consider an object $Y$ and the bottom $F_\s$-cluster $\bot_Y^\s$ in $Y$. Let $m\colon X\to Y$ be an embedding of $\bot_Y^\s$. By applying Lemma~\ref{LemK}, it is not difficult to show that $\bot^\s_X=\top^\s_X$. Then $\bot^\e_X=\top^\e_X$ and so $\alpha(\bot_Y^\s)=m\cdot^\e \top^\e_X=m\cdot^\e \bot^\e_X=\bot^\e_Y$.
\end{proof}

\subsection*{Wyler joins} The construction of `Wyler join' describes the effect of a normal cluster on a conormal one, even when these two clusters do not necessarily belong to the same form. It plays a key role for constructing the main noetherian form in this paper. After defining this notion, we prove a variety of highly technical facts about Wyler joins, which are needed for arriving to the main results of the paper.

\begin{definition}
Consider an orean factorization $(F_\s ,F_\e )$ of a category $\mathbb{C}$. When $A\subseteq_\s X$ and for $R\subseteq_\e X$, we write $A\ast R$ for
  \[A\ast R=(\pi_{R}^\e  \cdot^\s  A)\cdot^\s  \pi_{R}^\e \] 
  and call it a \emph{Wyler join} of $A$ and $R$.
  We write $A\astop R$ for the dual construction, and call it \emph{Wyler meet}: $A\astop R$ is the same as $R\ast A$ relative to $(F_\e^\mathsf{op},F_\s^\mathsf{op})$, and hence,
\[A\astop R=  \iota_{A}^\s\cdot^\e(R \cdot^\e  \iota_{R}^\s).\]
\end{definition}

\begin{remark} $A\ast R$ is independent of the choice of a particular $F_\e$-quotient $\pi^\e_R$, which is an easy consequence of Lemma~\ref{LemQ} and Remark~\ref{RemJ}. The $A\ast R$ construction is an adaptation of the construction from \cite{W71,JW16b}, to the context of an orean factorization. 
\end{remark}

\begin{lemma}\label{LemL}
  For an orean factorization $(F_\s ,F_\e )$, the values of the normal operator $\beta\colon F_\e\to F_\s$ are given by $\beta(R)=\bot^\s\ast R$. Furthermore, the following formula holds:
  \[(\beta(R)\lor^\s \beta(S))\ast (R\lor^\e  S)=\beta(R\lor^\e  S).\]
\end{lemma}
\begin{proof}
  $\bot^\s\ast R=(\pi_R^\e\cdot^\s\bot^\s)\cdot^\s \pi_R^\e=\bot^\s\cdot^\s\pi_R^\e=\ker^\s \pi_R^\e=\beta(\ker^\e \pi_R^\e)=\beta(R)$. This proves the first part of the lemma. It also shows that the proving the formula stated in the second part is equivalent to proving that
  \[(\beta(R)\lor^\s \beta(S))\ast (R\lor^\e  S)=\bot^\s\ast(R\lor^\e  S).\]
  By the general properties of a Galois connection (Remark~\ref{RemC}), this in turn is equivalent to
  \[\pi^\e_{R\lor^\e S}\cdot^\s(\beta(R)\lor^\s \beta(S))=\bot^\s.\]
 Since direct images preserve joins (Remark~\ref{RemC}), the equality above is further equivalent to
 $\pi^\e_{R\lor^\e S}\cdot^\s \beta(R)=\bot^\s$ (and a similar formula for $\beta(S)$). We can get this by the factorization $\pi^\e_{R\lor^\e S}=u\circ \pi^\e_{R}$ arising from the universal property of $\pi^\e_R$:
 \begin{align*}
     \pi^\e_{R\lor^\e S}\cdot^\s \beta(R) &=\pi^\e_{R\lor^\e S}\cdot^\s (\bot^\s \ast R)\\
     &=\pi^\e_{R\lor^\e S}\cdot^\s ((\pi_R^\e\cdot^\s \bot^\s)\cdot^\s \pi_R^\e)\\
     &=(u\circ \pi^\e_{R})\cdot^\s ((\pi_R^\e\cdot^\s \bot^\s)\cdot^\s \pi_R^\e)\\
     &=u\cdot^\s  (\pi^\e_{R}\cdot^\s ((\pi_R^\e\cdot^\s \bot^\s)\cdot^\s \pi_R^\e))\\
     &=u\cdot^\s  (\pi^\e_{R}\cdot^\s  \bot^\s)\\
     &=\pi^\e_{R\lor^\e S}\cdot^\s \bot^\s\\
     &=\bot^\s.
     \qedhere 
 \end{align*}
 \end{proof}
 
\begin{remark}\label{RemP}
Applying the standard properties of Galois connections (see Remark~\ref{RemC}), we have the following laws: $A\ast R\ge A$, $(A\ast R)\ast R=A\ast R$. We also have the law $A\ast\bot=A$ thanks to Lemmas~\ref{LemV} and \ref{LemQ}. Thanks to Lemma~\ref{LemL}, we also have the law $\beta(R)=\beta(R)\ast R$.
\end{remark}
  
\begin{theorem}\label{ThmR}
  For an orean factorization $(F_\s ,F_\e )$, the mapping $(A,R)\mapsto A\ast R$ defines an operator $F_\s\times F_\e\to F_\s$.    
\end{theorem}
\begin{proof}
  Note that if $S\ge^\e_f R$ then we have a commutative diagram
    \[\xymatrix{X\ar[r]^-{f}\ar[d]_-{\pi^\e _R} & Y\ar[d]_-{\pi^\e_S } \\ X/R\ar@{..>}[r]_-{u} & X/S}\]
  by the universal property of the quotient $\pi^\e _R$. If $(B,S)\ge_f (A,R)$, then 
  \begin{align*}
    B\ast S &= (\pi^\e _S \cdot^\s  B)\cdot^\s \pi^\e _S  \\ 
    &\ge (\pi^\e _S \cdot^\s  (f\cdot^\s  A))\cdot^\s \pi^\e _S   \\ 
    &=(u\cdot^\s  (\pi^\e _R\cdot^\s  A))\cdot^\s \pi^\e _S  \\ 
    &=(u\cdot^\s  (\pi^\e _R\cdot^\s ((\pi^\e _R\cdot^\s  A)\cdot^\s \pi^\e _R)))\cdot^\s \pi^\e _S  \\ 
    &=(u\cdot^\s  (\pi^\e _R\cdot^\s (A\ast R)))\cdot^\s \pi^\e _S  \\
    &=(\pi^\e _S \cdot^\s  (f\cdot^\s (A\ast R)))\cdot^\s \pi^\e _S  \\ 
    &\ge f\cdot^\s (A\ast R)\\
    &\ge_f A\ast R. \qedhere
  \end{align*}  
\end{proof}

\begin{lemma}\label{alphaastlemma}
  For any orean factorization $(F_\s , F_\e )$, the following holds:
    \[R\ge^\e \alpha(A) \quad\iff\quad R\ge^\e \alpha(A\ast R).\]
Further, if the join formula from (N1) holds for $F_\e$, then we have the law
$$\alpha(A)\lor^\e R=\alpha(A\ast R)\lor^\e R.$$
Also, then the law $B=B\ast\alpha(B)$ is equivalent to the following law:
$$A\ast R=A\ast (\alpha(A)\vee^\e R).$$
Finally, if (N1) holds for $F_\e$, then the following law holds: \[R\ge^\e \ker^\e f\quad\implies\quad ((f\cdot^\s  A)\ast (f\cdot^\e  R))\cdot^\s f=A\ast R.\]
\end{lemma}
\begin{proof}
  The backward implication in the equivalence follows from the fact that $A\ast R\ge ^\e A$. For the forward implication, first note that by Lemma~\ref{LemN}, we have:
  \begin{align*}
  (\pi_R^\e\cdot^\e \alpha(A))\cdot^\e \pi_R^\e &=\alpha(\pi_R^\e\cdot^\s A)\cdot^\e \pi_R^\e\\
  &\ge^\e \alpha((\pi_R^\e\cdot^\s A)\cdot^\s \pi_R^\e)\\
  &=\alpha(A\ast R)
  \end{align*} 
  So, if $R\ge^\e \alpha(A)$, then $R=(\pi_R^\e\cdot^\e R)\cdot^\e \pi_R^\e\ge^\e (\pi_R^\e\cdot^\e \alpha(A))\cdot^\e \pi_R^\e\ge^\e \alpha(A\ast R)$.
  Suppose now the join formula from (N1) holds for $F_\e$. By the computation above, we have:
  \[\alpha(A)\lor^\e R=(\pi^\e_R\cdot^\e\alpha(A))\cdot^\e\pi^\e_R\ge^\e \alpha(A\ast R).\]
  Consequently, $\alpha(A)\lor^\e  R=\alpha(A\ast R)\lor^\e R$. 
  The law $A\ast R=A\ast (\alpha(A)\vee^\e R)$ easily implies the law $B=B\ast\alpha(B)$ (just set $A=B$ and $R=\bot^\e$). For the converse, we first note that by Lemmas~\ref{LemN} and \ref{LemK}, we have:
    \[\pi^\e_{\alpha(A)\lor^\e R}=\pi^\e_{(\pi^\e_R\cdot^\e\alpha(A))\cdot^\e \pi^\e_R}=\pi^\e_{\alpha(\pi^\e_R\cdot^\s A)\cdot^\e \pi^\e_R}=\pi^\e_{\alpha(\pi^\e_R\cdot^\s A)}\circ \pi^\e_R.\]   
  We now carry out the following computation, where in the fourth equality we use $B\ast\alpha(B)=B$ for $B=\pi_R^\e\cdot^\s A$:
  \begin{align*}
  A\ast(\alpha(A)\vee^\e R) &= (\pi^\e_{\alpha(A)\lor^\e R}\cdot^\s A)\cdot^\s \pi^\e_{\alpha(A)\lor^\e R}\\
  &=((\pi^\e_{\alpha(\pi^\e_R\cdot^\s A)}\cdot^\s(\pi^\e_R \cdot^\s A))\cdot^\s \pi^\e_{\alpha(\pi^\e_R\cdot^\s A)})\cdot^\s \pi^\e_R\\
  &=((\pi^\e_R \cdot^\s A)\ast\alpha(\pi^\e_R \cdot^\s A))\cdot^\s \pi^\e_R\\
  &=(\pi^\e_R \cdot^\s A)\cdot^\s \pi^\e_R\\
  &=A\ast R.
  \end{align*}
  Suppose now both the meet and the join formula from (N1) hold for $F_\e$. Note that given $f$ and $R$, we have $\pi^\e_{f\cdot^\e R}\circ f=g\circ \pi^\e_R$ for a (unique) $g$. If $R\ge^\e \ker^\e f$, then we have (thanks to Lemma~\ref{LemK}):
  \begin{align*}
      \mathsf{Ker}^\e g &= \top^\e\wedge^\e \mathsf{Ker}^\e g\\
      &=\im^\e \pi^\e_R\wedge^\e \mathsf{Ker}^\e g\\
      &=  \pi^\e_R\cdot^\e(\mathsf{Ker}^\e g\cdot^\e \pi^\e_R)\\
      &=\pi^\e_R\cdot^\e((f\cdot^\e R)\cdot^\e f)\\
     &=\pi^\e_R\cdot^\e(\ker^\e f \vee^\e R)\\
      &=\pi^\e_R\cdot^\e R\\
      &=\bot^\e 
  \end{align*}
  So, by Lemma~\ref{LemK}, $g$ is then an $F_\s$-embedding. Furthermore, we have:
  \begin{align*}
      ((f\cdot^\s A)\ast(f\cdot^\e R))\cdot^\s f&=((\pi^\e_{f\cdot^\e R}\cdot^\s (f\cdot^\s A))\cdot^\s \pi^\e_{f\cdot^\e R})\cdot^\s f\\
      &=(g\cdot^\s (\pi_R^\e\cdot^\s A))\cdot^\s g)\cdot^\s \pi_R^\e\\
      &=(\pi_R^\e\cdot^\s A)\cdot^\s \pi_R^\e\\
      &=A\ast R.\qedhere
  \end{align*}
\end{proof}

\begin{example}\label{ExaR}
  The pair $(F_\s,F_\e)$, where $F_\s$ is the form of subsets of sets and $F_\e$ is the form of equivalence relations of sets, is an orean factorization of the category of sets arising from its epi-mono factorization system (see Example~\ref{ExaK}). In this case, $A\ast R$ is the union of those equivalence classes for $R$ that intersect $A$,
  \[A\ast R=\bigcup\{[x]_R\mid x\in A\}.\]
  As already remarked in Example~\ref{ExaAA}, (N1) holds for the form $F_\e$. 
  The two sides of the equality $\alpha(A)\lor R=\alpha(A\ast R)\lor R$ in Lemma~\ref{alphaastlemma}  evaluate to an equivalence relation $E$ obtained from $R$ by merging all equivalence classes of elements of $A$ into a single equivalence class. In this example, the law $B=B\ast\alpha(B)$ from Lemma~\ref{alphaastlemma} does hold. Essentially, it says that when we turn a non-mepty set $B$ into an equivalence relation $E$ by generating one from the relation $B\times B$, the set $B$ will be one of the equivalence classes of $E$. This would not be necessarily true for the orean factorization $(F_\s,F_\e)$ of, say, the category of groups, where $F_\s$ is the form of subgroups and $F_\e$ is the form of normal subgroups. In this case, for a subgroup $B$ of a group $G$, we would have that $B\ast\alpha(B)$ is the normal subgroup generated by $B$, and hence $B=B\ast\alpha(B)$ if and only if $B$ is normal.  
\end{example}

\begin{lemma}\label{LemAC}
  Consider an orean factorization $(F_\s,F_\e)$. Let $A$ be an $F_\s$-cluster and let $R$ be an $F_\e$-cluster in an object $X$. We have $A\ast R=A$ and $R\ge^\e\alpha(A)$ if and only if there is a pullback
    \[\xymatrix{ \bullet\ar[d]_-{\iota^s_A}\ar[r]^-{\pi^\e_{\top^\e}} & \bullet\ar[d]^-{\iota^\s_{\pi^\e_R\cdot^\s A}}\\ \bullet\ar[r]_-{\pi^\e_R} & \bullet }\]
  Furthermore, when $A\ast R=A$ and $R\ge^\e\alpha(A)$, for an $F_\s$-cluster $B$ such that $A\ge^\s B$, we have $B\ast R=A$ if and only if the composite $\pi^\e_{\top^\e}\circ \iota^\s_{B\cdot^\s\iota^\s_A}$ is an $F_\e$-quotient. 
\end{lemma}
\begin{proof}
  First we observe that thanks to Theorem~\ref{ThmF}, $R\ge\alpha(A)$ is equivalent to $R\ge^\e \iota^\s_A\cdot^\e\top^\e$, which is further equivalent to $R\cdot^\e \iota^\s_A=\top^\e$.

  Suppose first the named pullback exists. Then, thanks to Lemma~\ref{LemP}, \[A=(\pi^\e_R\cdot^\s A)\cdot^\s\pi^\e_R=A\ast R.\]  Commutativity of the pullback square and Lemma~\ref{LemK} gives $R\cdot^\e \iota^\s_A=\top^\e$.   

  Suppose now we have $A\ast R=A$ and $R\cdot^\e \iota^\s_A=\top^\e$. The equality $A\ast R=A$ along with Lemma~\ref{LemP} and the definition of orean factorization, gives rise to the following pullback.
    \[\xymatrix{ \bullet\ar[d]_-{\iota^s_A}\ar[r]^-{\pi^\e_{R\cdot^\e\iota^\s_A}} & \bullet\ar[d]^-{\iota^\s_{\pi^\e_R\cdot^\s A}}\\ \bullet\ar[r]_-{\pi^\e_R} & \bullet }\]
  Thanks to the equality $R\cdot^\e \iota^\s_A=\top^\e$, the pullback above is exactly the desired one. 

  Consider now an $F_\s$-cluster $B$ such that $A\ge_\s B$, along with the given pullback. We can then extend the pullback diagram to the following commutative diagram, thanks to Lemma~\ref{LemK}.
  \[\xymatrix{
    \bullet
      \ar[r]^-{\iota^\s_{B\cdot^\s\iota^\s_A}}
      \ar[dr]_-{\iota^\s_B} &
    \bullet
      \ar[d]^-{\iota^s_A}
      \ar[r]^-{\pi^\e_{\top^\e}} &
    \bullet
      \ar[d]^-{\iota^\s_{\pi^\e_R\cdot^\s A}}\\ &
    \bullet
      \ar[r]_-{\pi^\e_R} & 
    \bullet }\]
  Thanks to Lemma~\ref{LemK}, we have the following sequence of equivalences (for readability, unnecessary superscripts $\s$ and $\e$ were left out):  
  \begin{align*}
    && A\ast R=A&=B\ast R\\
    \iff && (\pi_R\cdot A)\cdot \pi_R&=(\pi_R\cdot B)\cdot \pi_R\\
    \iff && \pi_R\cdot A&=\pi_R\cdot B\\
    \iff && \iota_{\pi_R\cdot A}\cdot \top^\s&=\pi_R\cdot\iota_B\cdot \top^\s=\pi_R\cdot\iota_A\cdot \iota_{B\cdot\iota_A}\cdot\top^\s=\iota_{\pi_R\cdot A}\cdot\pi_{\top^\e}\cdot\iota_{B\cdot\iota_A}\cdot \top^\s\\
    \iff && \top^\s&=\pi_{\top^\e}\cdot\iota_{B\cdot\iota_A}\cdot\top^\s=\im(\pi_{\top^\e}\circ\iota_{B\cdot\iota_A})\\
    \iff  && 
    \multispan{2}{\quad$\pi_{\top^\e}\circ\iota_{B\cdot\iota_A}$ is an $F_\e$-quotient.\hfill\qedhere}
  \end{align*}
\end{proof}

\begin{remark}\label{RemQ}
The proof of the next technical lemma requires the so-called `restricted modular law' (Lemma~5.1 in \cite{GJ19}), which holds in any noetherian form (and more generally, in any orean form satisfying (N1), as it is evident from the proof given in \cite{GJ19}). It states that for any clusters $X$, $Y$ and $Z$ in the same object, if $X\le Z$ then we have
  \[(X\lor Y)\land Z=X\lor (Y\land Z),\]
provided that $X$ is normal and $Y$ conormal, or $Y$ is normal and $Z$ conormal.
\end{remark}

\begin{lemma}
  \label{inverseimageinteriorlem}
  Consider a form $F$ that is both strongly orean and noetherian. For a conormal cluster $A$ and a normal cluster $R$, we have
  \[A\ast R=\cccl{A\vee^F R},\]
  where $A\ast R$ is relative to the orean factorization $(F_\c,F_\n)$. Furthermore, for any two pairs of conormal-normal $F$-clusters $(A,R)$ and $(B,S)$, if $A=A\ast R$, $R\ge\alpha(A)$, $B=B\ast S$ and $S\ge\alpha(B)$, where $\alpha$ denotes the conormal operator $F_\c\to F_\n$, then
    \[A\lor^F R\le B\lor^F S\quad\implies\quad A\le B\textrm{ and }R\le S.\]
\end{lemma}
\begin{proof}
  Firstly, note that by Theorem~\ref{ThmT}, the pair $(F_\c,F_\n)$ is indeed an orean factorization. By Theorem~\ref{ThmT}, $\pi_R^F=\pi_R^\n$ and so by Lemma~\ref{LemM} and (N1),
    \begin{align*}
    A\ast R &=(\pi_R^\n\cdot^\c A)\cdot^\c \pi_R^\n \\ 
    &=(\pi_R^F\cdot^\c A)\cdot^\c \pi_R^F\\ &=\cccl{(\pi_R^F\cdot^F A)\cdot^F \pi_R^F}\\ &=\cccl{A\vee^F R}.
    \end{align*}
  This proves the first part of the lemma.

  Consider now pairs of conormal-normal $F$-clusters $(A,R)$ and $(B,S)$, such that $A=A\ast R$, $R\ge\alpha(A)$, $B=B\ast S$,  $S\ge\alpha(B)$, and $A\lor^F R\le B\lor^F S$. By the first part of the lemma, we have:
  \[A=A\ast R=\cccl{A\vee^F R}\le \cccl{B\vee^F S}=B\ast S= B.\] 
  We now want to prove $R\le S$. First, note that
    \[B\lor^F R\lor^F S=B\lor^F A\lor^F R\lor^F S=B\lor^F S.\]
  We write $\iota_B$ for $\iota^\c_B$, which by Theorem~\ref{ThmT}, equals $\iota^F_B$. From $S\ge \alpha(B)=\iota_B\cdot^\n \top^\n_B$ we get $S\cdot^\n\iota_B= \top^\n_B$, and consequently, $(R\vee^F S)\cdot^\n\iota_B= \top^\n_B$ (note that $R\vee^F S$ is normal by (N3)).  We then have:
  \begin{align*}
    B\land^F S &=\iota_B\cdot^F (S\cdot^F \iota_B) & \textrm{(N1)}\\ &=\iota_B\cdot^F (S\cdot^\n \iota_B) & \textrm{(dual of Lemma~\ref{LemM})} \\ &=\iota_B\cdot^F \top^\n_B & \\ &\le\ncl{\iota_B\cdot^F \top^\n_B}  \\&= \iota_B\cdot^\n \top^\n_B & \textrm{(dual of Lemma~\ref{LemM})} \\ &=\alpha(B).
  \end{align*}    
Similarly, we get  $B\land^F (R\lor^F S) =\alpha(B)$ (just replace $S$ in the computation above with $R\lor^F S$).
Therefore, $B\land^F S=B\land^F(R\lor^F S)$. Using the restricted modular law (Remark~\ref{RemQ}), we have the following (where all joins and meets are performed in the form $F$):
  \begin{align*}
    S&=S\lor(B\land S)\\
     &=S\lor(B\land(R\lor S))\\
     &=(S\lor B)\land (R\lor S)\\
     &=(B\lor (R\lor S))\land(R\lor S)\\
     &=R\lor S.
  \end{align*}
  Thus $R\le S$.
\end{proof}

\section{Decompositions of orean forms}\label{SecG}

\subsection*{General decompositions} 
We can view an orean factorization $(F_\s,F_\e)$ of a category as a decomposition of the category in two subcategories, the subcategory of $F_\s$-embeddings and the subcategory of $F_\e$-quotients, as any morphism in the category is a composite of a morphism from one subcategory with a morphism from the other subcategory. We will now introduce a concept of decomposing an orean form into a pair $(F_\s,F_\e)$.

Let $\mathfrak{t}$ be a binary term in the algebraic theory of bounded lattices. We have the following cases:
\begin{tasks}[label=\empty, after-item-skip=4pt](4)
   \task $\mathfrak{t}_\top(x,y)=\top$
   \task $\mathfrak{t}_1(x,y)=x$
   \task* $\mathfrak{t}_\lor(x,y)=x\lor y$
   \task $\mathfrak{t}_\bot(x,y)=\bot$
   \task $\mathfrak{t}_2(x,y)=y$
   \task $\mathfrak{t}_\land(x,y)=x\land y$
\end{tasks}
For an orean form $F$, we write $\mathfrak{t}^F_X$ to denote the interpretation of $\mathfrak{t}$ in the fibre $F^{-1}(X)$.  

\begin{definition}\label{defB}
A \emph{decomposition} of an orean form $F$ is given by a pair $(\kappa_\s ,\kappa_\e )$ of idempotent operators on $F$ such that 
\begin{itemize}
\item[(D1)] $F_\s=F_{\kappa_\s}$ is a conormal orean form and $F_\e=F_{\kappa_\e}$ is a normal orean form,

\item[(D2)] there is a term $\mathfrak{t}$, called a \emph{decomposing term}, such that
$S=\mathfrak{t}^F_X(\kappa_\s(S),\kappa_\e(S))$ for any object $X$ and cluster $S\subseteq_F X$.
\end{itemize}
In a decomposition, we call the restrictions $F\to F_\s$ and $F\to F_\e$ of $\kappa_\s$ and $\kappa_\e$, respectively, \emph{decomposition projections}, and denote them by $\tau_\s$ and $\tau_e$, respectively.
When $\tau_\s $ and $\tau_\e $ are binormal and the pair $(F_\s,F_\e)$ is a semiexact pair, we say that the decomposition is \emph{semiexact}. If furthermore the subform inclusion $F_\s \to F$ is conormal and the subform inclusion $F_\e \to F$ is normal, then we say that the decomposition is \emph{exact}.
\end{definition}

\begin{convention}
We will use the following short-hand names for a decomposition, depending on its decomposing term.
\begin{itemize}
    \item When the decomposing term is $\mathfrak{t}_\bot$ or $\mathfrak{t}_\top$: \emph{nullary decomposition}.

    \item When the decomposing term is $\mathfrak{t}_1$ or $\mathfrak{t}_2$: \emph{unary decomposition}.
    
    \item When the decomposing term is $\mathfrak{t}_1$: \emph{left decomposition}.
    
    \item When the decomposing term is $\mathfrak{t}_2$: \emph{right decomposition}.
    
    \item When the decomposing term is $\mathfrak{t}_\land$ or $\mathfrak{t}_\lor$: \emph{binary decomposition}.
    
    \item When the decomposing term is $\mathfrak{t}_\land$: \emph{meet decomposition}.
    
    \item When the decomposing term is $\mathfrak{t}_\lor$: \emph{join decomposition}.
\end{itemize}
\end{convention}

\begin{remark}
The notion of (semiexact/exact) decomposition of an orean form $F$ is self-dual: $(\kappa_\s,\kappa_\e)$ is a (semiexact/exact) decomposition if and only if $(\kappa^\mathsf{op}_\e,\kappa^\mathsf{op}_\s)$ is a (semiexact/exact) decomposition of $F^\mathsf{op}$. So duality swaps the roles of $\kappa_\s$ and $\kappa_\e$. Accordingly, it swaps the roles of $F_\s$ and $F_\e$, and the roles of $\tau_\s$ and $\tau_\e$. Note also that duality swaps the decomposing terms in the following pairs: $\mathfrak{t}_\top$ vs $\mathfrak{t}_\bot$; $\mathfrak{t}_1$ vs $\mathfrak{t}_2$; and $\mathfrak{t}_\wedge$ vs $\mathfrak{t}_\vee$. So the notion of
\begin{itemize}
    \item nullary/unary/binary decomposition is self-dual,
    
    \item left decomposition is dual to right decomposition,
    
    \item meet decomposition is dual to join decomposition.
\end{itemize}
\end{remark}

\begin{remark}\label{RemT}
In a decomposition, the requirements of $\tau_\s $ and $\tau_\e $ being binormal with the subform inclusion $F_\s \to F$ being conormal and the subform inclusion $F_\e \to F$ being normal together imply that the pair $(F_\s , F_\e )$ is a semiexact pair, by virtue of Remark~\ref{RemO} (the conormal operator $\alpha\colon F_\c\to F_\n$ is then given by $\alpha(A)=\tau_\e(A)$ and the normal operator $\beta\colon F_\n\to F_\c$ is given by $\beta(R)=\tau_\s(R)$). So a decomposition is exact if and only if it satisfies the former requirements. We note also that when an orean form $F$ has exact decomposition $(\kappa_\s,\kappa_\e)$, necessarily $F_\s=F_\c$ and $F_\e=F_\n$ and so $F$ is strongly orean (thanks to Lemma~\ref{LemU}). We will later show that an orean form $F$ can have at most one exact decomposition.
\end{remark}

\begin{remark}\label{RemW}
A semiexact decomposition of $F$ gives rise to the diagram
\[\xymatrix{ & F\ar@/_10pt/[ld]_-{\tau_\s}\ar@/^10pt/[rd]^-{\tau_\e} & \\ F_\s\ar@<+3pt>[rr]^-{\alpha}\ar[ur] & & F_\e\ar@<+3pt>[ll]^-{\beta}\ar[ul] }\]
where: 
\begin{itemize}
\item $\tau_\s$ is a left inverse of the subform inclusion $F_\s\to F$ and $\tau_\e$ is a left inverse of the subform inclusion $F_\e\to F$.

\item $\alpha$ is conormal and $\beta$ is normal.

\item Both $\tau_\s$ and $\tau_\e$ are binormal.
\end{itemize}
By Remark~\ref{RemO} and Lemma~\ref{LemN}, in the case of exact decomposition, we furthermore have that the composite of $\tau_\e$ with the subform inclusion $F_\s\to F$ is $\alpha$ and the composite of $\tau_\s$ with the subform inclusion $F_\e\to F$ is $\beta$.
\end{remark}

\subsection*{Unary and nullary decompositions} Here we establish characterizations of unary and nullary decompositions.

\begin{theorem}\label{ThmAR}\label{ThmAI}
A pair
$(\kappa_\s,\kappa_\e)$ of operators is a left decomposition of an orean form $F$ if and only if $F$ is conormal, $\kappa_\s$ is an identity operator of $F$ and $\kappa_\e$ is an idempotent operator $F\to F$ such that $F_\e=F_{\kappa_\e}$ is a normal orean form. When $(\kappa_\s,\kappa_\e)$ is a left decomposition of $F$, the following conditions are equivalent:
\begin{itemize}
    \item[(i)] $(\kappa_\s,\kappa_\e)$ is a semiexact decomposition of $F$.
    
    \item[(ii)] $(\kappa_\s,\kappa_\e)$ is an exact decomposition of $F$.
    
    \item[(iii)] $\tau_\e$ is binormal and $F_\e=F_\n$.
    
    \item[(iv)] $\tau_\e$ is conormal and $F_\e=F_\n$.
\end{itemize}
An orean form $F\colon \mathbb{B}\to\mathbb{C}$ has a left exact decomposition if and only if $(F,F_\n)$ is a left exact pair. Moreover, when $F$ has a left exact decomposition, it is necessarily unique: it is the pair $(1_\mathbb{B},\beta\alpha)$, where $\beta$ is the subform inclusion $F_\n\to F$ and $\alpha$ is the unique conormal operator $F\to F_\n$. Furthermore, in this case, both $\alpha$ and $\beta$ are normal. 
\end{theorem}

\begin{proof}
The first statement of the theorem is immediate from the definition of left decomposition. We prove equivalence of conditions in the second statement.

(i)$\implies$(ii): Since $F=F_\c$ and by semiexactness there is a normal operator $F_\e\to F_\c$, thanks to the dual of Lemma~\ref{LemN} we get that the subform inclusion $F_\e\to F$ is normal. Then, the decomposition must be exact. 

(ii)$\implies$(iii) and (iv)$\implies$(i) are thanks to the dual of Lemma~\ref{LemU}. 

(iii)$\implies$(iv) is trivial. 

To prove the third statement of the theorem, suppose an orean form $F$ has a left exact decomposition. Then the left diagonal operators in the diagram of Remark~\ref{RemW} are identity operators. Combining Remarks~\ref{RemT} and \ref{RemW}, we then get that $(F,F_\n)$ is a left exact pair. From the same remarks, the fourth and the fifth statements of the theorem follow too. Suppose now $(F,F_\n)$ is a left exact pair, where $F$ is an orean form $F\colon\mathbb{B}\to\mathbb{C}$. Then $F_\n$ is orean, and so, by the dual of Lemma~\ref{LemAA}, the subform inclusion $F_\n\to F$ is normal. Moreover, by the dual of Lemma~\ref{LemM}, $F_\n$ is normal, and so, by the dual of Lemma~\ref{LemN}, the subform inclusion $F_\n\to F$ is the unique normal operator $\beta\colon F_\n\to F$. So the conormal operator $\alpha\colon F\to F_\n$ is its left inverse. It then follows easily from the dual of Lemma~\ref{LemM} that $\alpha$ is binormal. So a left exact decomposition of $F$ is given by $(1_\mathbb{B},\beta\alpha)$.
\end{proof}

\begin{remark}
Theorem~\ref{ThmAI} together with Lemma~\ref{LemAH} give that orean forms admitting left exact decomposition are precisely those which are part of a left exact pair $(F,F_\e)$.
\end{remark}

\begin{theorem}\label{ThmAJ}
For an orean form $F\colon\mathbb{B}\to\mathbb{C}$, the following conditions are equivalent:
\begin{itemize}
    \item[(i)] $F$ has a decomposition with $\mathfrak{t}_\bot$ as a decomposing term.
    
    \item[(ii)] $F$ has a nullary decomposition.
    
    \item[(iii)] $F$ is an isoform.
    
    \item[(iv)] $(1_{\mathbb{B}},1_{\mathbb{B}})$ is a decomposition of $F$ which is nullary, unary and binary.
\end{itemize}
Every isoform has a unique decomposition, which is an exact decomposition. 
\end{theorem}

\begin{proof}
(i)$\implies$(ii) is trivial.

(ii)$\implies$(iii): if $F$ has a nullary decomposition with, say, $\mathfrak{t}_\bot$ as a decomposing term (the other case is dual to this), then every cluster is a bottom cluster and hence $F$ is an isoform (Lemma~\ref{LemX}).

(iii)$\implies$(iv) is obvious after Lemma~\ref{LemX}.

(iv)$\implies$(i) is trivial.

If $F$ is an isoform, then the only idempotent operator on $F$ is the identity operator. It is obvious that $(1_{\mathbb{B}},1_{\mathbb{B}})$ is an exact decomposition. 
\end{proof}

\subsection*{Binary decompositions} Here we establish characterization theorems dealing with binary decompositions of orean forms. We also explore some examples of such decompositions.

\begin{theorem}\label{LemE}\label{ThmZ}
Let $F$ be an orean form.
\begin{itemize}
  \item[(i)] In a join decomposition $(\kappa_\s,\kappa_\e)$ of $F$, each $\kappa_\s S$ is the largest $F_\s $-cluster below $S$, and each $\kappa_\e S$ is the largest $F_\e $-cluster below $S$. In other words, $\kappa_\s$ and $\kappa_\e$ are idempotent co-closure operators on $F$. This implies, in particular, that every join decomposition of an orean form $F$ is uniquely determined by the pair $(F_\s,F_\e)$ of its subforms. 
  
  \item[(ii)] A pair $(F_\s,F_\e)$ of subforms of $F$ gives rise to a join decomposition of $F$ if and only if $F_\s$ is conormal, $F_\e$ is normal, the subform inclusion $F_\s\to F$ has a left inverse $\tau_\s$, the subform inclusion $F_\e\to F$ has a left inverse $\tau_\e$ and for each $S\subseteq_F X$ we have $S =\tau_\s S\lor_X\tau_\e S$. The operators $\tau_\s$ and $\tau_\e$ are then the decomposition projections. 
  
  \item[(iii)] For a join decomposition $(\kappa_\s,\kappa_\e)$ of $F$, the operator $\tau=(\tau_\s,\tau_\e)\colon F\to F_\s \times F_\e $ has a left inverse $\nu\colon F_\s \times F_\e \to F$, 
  \[\nu\colon (A,R)\mapsto A\vee R.\] Moreover, the composite $\tau\nu$ is an idempotent closure operator on $F_\s \times F_\e $ and the mapping $S\mapsto \tau S$ gives an isomorphism $F\approx(F_\s \times F_\e )_{\tau\nu}$. 
\end{itemize}
\end{theorem}

\begin{proof}
In (i), since $S=\kappa_\s S\vee \kappa_\e S=\tau_\s S\vee \tau_\e S$, it is certainly true that $\kappa_\s S$ is an $F_\s$-cluster below $S$ and $\kappa_\e S$ is an $F_\e$-cluster below $S$. The rest following by duality from Remark~\ref{RemK}. For similar reasons, we have (ii). 

In (iii), that $\nu$ is an operator follows from the fact that it is the composite of the subform inclusion $F_\s\times F_\e\to F\times F$ and the join operator $F\times F\to F$ from Example~\ref{ExaO}. The formula $S=\tau_\s S\vee \tau_\e S$ of join decomposition is equivalent to $\nu$ being a left inverse of $\tau$. To show that $\tau\nu$ is a closure operator on $F_\s\times F_\e$, we need to establish the inequalities $\tau_\s(A\vee R)\ge A$ and $\tau_\e(A\vee R)\ge R$. They follow from (i). The last claim of (iii) then follows from the third bullet point of Remark~\ref{RemK}.  
\end{proof}

\begin{definition}
We call a closure operator on a product $F_\s \times F_\e $ described in Theorem~\ref{ThmZ}(iii) a \emph{join decomposition closure operator}.
\end{definition}

\begin{convention}
By virtue of Theorem~\ref{LemE}(i), to specify a join decomposition $(\kappa_\s,\kappa_\e)$ of an orean form $F$, it is sufficient to specify the subforms $F_\s$ and $F_\e$ of $F$. We then write $F=F_\s\lor F_\e $ to name a join decomposition of $F$. Dually, a meet decomposition will be named by writing $F=F_\s\land F_\e $.
\end{convention}

\begin{theorem}\label{ThmY}
  For any pair $(F_\s ,F_\e )$ of orean forms, where $F_\s$ is conormal and $F_\e$ is normal, there is a join decomposition:
  \[F_\s\times F_\e=(F_\s\times F_\e^\bot)\vee (F_\s^\bot\times F_\e).\]
  The corresponding decomposition projections $\tau_\s$ and $\tau_\e$ are binormal and so this decomposition is semiexact if and only if the pair $(F_\s ,F_\e )$ is a semiexact pair. It is an exact decomposition if and only if the form $F_\e$ is an isoform and $F_\s$ is antinormal.
\end{theorem}

\begin{proof} To prove the lemma, it suffices to show that:
\begin{itemize}
    \item[(i)] $F_\s\times F_\e^\bot$ and $F_\s^\bot\times F_\e$ are subforms of $F_\s\times F_\e$, and the corresponding subform inclusions have left inverses given by $\tau_\s(A,R)=(A,\bot^\e)$ and $\tau_\e(A,R)=(\bot^\s,R)$, respectively. 
    \item[(ii)] $F_\s\times F_\e^\bot$ is a conormal orean form and $F_\s^\bot\times F_\e$ is a normal orean form.
    \item[(iii)] For any $F_\s$-cluster $A$ and $F_\e$-cluster $R$ in an object $X$, we have: $$(A,R)=(A,\bot_X^{\e})\vee^{F_\s\times F_\e} (\bot_X^{\s},R).$$
    \item[(iv)] The operators $\tau_\s$ and $\tau_\e$ are binormal.
    \item[(v)] The subform inclusion $F_\s\times F_\e^\bot\to F_\s\times F_\e$ is conormal if and only if $F_\e$ is an isoform. 
    \item[(vi)] The subform inclusion $F_\s^\bot\times F_\e\to F_\s\times F_\e$ is normal if and only if $F_\s$ is antinormal.
\end{itemize}
(i) is easy to see. (ii) follows from the fact that $F_\s$ is a conormal orean form and $F_\e$ is a normal orean form (apply Lemmas~\ref{LemR}, \ref{LemS} and Remark~\ref{RemO}). We have (iii) by Remark~\ref{RemL}. The same Remark~\ref{RemL} can be used to easily verify (iv-vi). It is also possible to deduce (iv) from Lemma~\ref{LemT}, according to which all product projections are binormal. Note that the operator $\tau_\s$ composed with the first product projection $F_\s\times F_\e^\bot\to F_\s$ is the first product projection $F_\s\times F_\e\to F_\s$. Since the former of these product projections is an isomorphism (thanks to Lemmas~\ref{LemR} and~\ref{LemS}), we get that $\tau_\s$ is itself a product projection. Similarly, $\tau_\e$ is a product projection.
\end{proof}

\begin{definition}
For a pair $(F_\s ,F_\e )$ of orean forms, where $F_\s$ is conormal and $F_\e$ is normal, the join decomposition of the product $F_\s \times F_\e $ given by Theorem~\ref{ThmY} will be referred to as the \emph{canonical} join decomposition of $F_\s \times F_\e$.  
\end{definition}

\begin{theorem}\label{ThmM}
An orean form $F$ has an exact join decomposition if and only if it has the following three properties:
  \begin{itemize}
    \item[(i)] Every cluster $S$ decomposes as a join $S=C\vee N$, where $C$ is the largest conormal cluster below $S$ while $N$ is the largest normal cluster below $S$.

    \item[(ii)] Inverse image of the largest normal cluster below a cluster $S$, along a morphism $f$, is the largest normal cluster below the inverse image of $S$ along $f$.

    \item[(iii)] Direct image of the largest normal cluster in an object along a morphism $f$ is the largest normal cluster below the image of $f$. 
  \end{itemize}
  Moreover, when an orean form $F$ has an exact join decomposition, it is necessarily unique, is given by $F=F_\c\vee F_\n$ and we have:
  \begin{itemize}
      \item direct images of normal clusters are normal, i.e., $F_\n$ is closed in $F$ under direct images, 
    
       \item joins of conormal clusters are conormal, and joins of normal clusters are normal, i.e., both $F_\c$ and $F_\n$ are closed in $F$ under joins.
  \end{itemize}
\end{theorem}
\begin{proof}
  We first prove the `only if part' of the first claim of the theorem. Let $F=F_\s \vee F_\e $ be an exact join decomposition of $F$. By Remark~\ref{RemT}, $F_\s=F_\c$ and $F_\e=F_\n$ (and so we get that the exact join decomposition is given by $F=F_\s \vee F_\e $, as claimed in the second part of the theorem). Then, by Theorem~\ref{LemE}, $\tau_\s S$ is the largest conormal cluster below $S$ and $\tau_\e S$ is the largest normal cluster below $S$. Then, the formula $S=\tau_\s S\vee \tau_\e S$ from exact join decomposition gives (i). To show (ii), consider a cluster $S$ in $Y$ and a morphism $f\colon X\to Y$. We want to show that $(\tau_\e S)\cdot^F f=\tau_\e(S\cdot^F f)$. We have $(\tau_\e S)\cdot^F f\ge^F_X \tau_\e(S\cdot^F f)$ by monotonicity of $\tau_\e$. Since inverse image of a normal cluster is normal, and since $\tau_\e(S\cdot^F f)$ is the largest normal cluster below $S\cdot^F f$, we also have $\tau_\e(S\cdot^F f)\ge^F_X (\tau_\e S)\cdot^F f$. This proves (ii). To prove the third property (iii), we first remark that by Theorem~\ref{LemE}, we have an idempotent co-closure operator $\kappa$ on $F$ given by the mapping $S\mapsto \tau_\e S$, for which $F_\kappa=F_\e$. We can then get (iii) by applying the dual of Lemma~\ref{LemA} with the operator $\kappa$ playing the role of the first operator in Lemma~\ref{LemA} (also written as $\kappa$ in Lemma~\ref{LemA}), and the operator $\tau_\e$ playing the role of the second (note that by definition of exact join decomposition, the operator $\tau_\e$ is binormal).  

  We now prove the `if' part of the theorem. Suppose $F$ is an orean form having the properties (i-iii). Let $F_\s$ denote its subform of conormal clusters and let $F_\e$ denote its subform of normal clusters. Assigning to a cluster the largest conormal cluster below it (whose existence is guaranteed by (i)) defines an operator $F\to F_\s$ (monotonicity is guaranteed by the fact that direct images of conormal clusters are conormal). Assigning to a cluster the largest normal cluster below it defines an operator $F\to F_\e$, whose monotonicity is thanks to (ii). It is easy to see that these operators are left inverses of the subform inclusions $F_\s\to F$ and $F_\e\to F$, respectively. Therefore, by the dual of the second bullet point in Remark~\ref{RemK}, they give rise to idempotent co-closure operators on $F$, for which closed clusters are given by the forms $F_\s$ and $F_\e$, respectively. We can then apply (the dual of) Theorem~\ref{ThmA} to get that the forms $F_\s$ and $F_\e $ are orean. It follows by Lemma~\ref{LemU} and its dual that $F_\s$ is a conormal form and $F_\e$ is a normal form. All of this, together with  (i), guarantee that we have a join decomposition $F=F_\s \vee F_\e $. Note that by Lemma~\ref{LemU} and its dual, we furthermore have that the subform inclusion $F_\s \to F$ is conormal and the subform inclusion $F_\e \to F$ is normal. By the dual of Lemma~\ref{LemA}, both $\tau_\s $ and $\tau_\e $ are normal. Moreover, $\tau_\s $ is obviously conormal. By the dual of Lemma~\ref{LemA} and (iii), $\tau_\e $ is conormal. So both $\tau_\s$ and $\tau_\e$ are binormal and we finally get that $F=F_\s \vee F_\e $ is an exact join decomposition of $F$.

  We note that the first claim of the theorem, thanks to Theorem~\ref{LemE}, implies that when $F$ has an exact decomposition, an $F$-cluster $S$ is conormal if and only if $S=\tau_\s S$ and $S$ is normal if and only if $S=\tau_\e S$. We will use this for the proof of the remaining, last part of the theorem. If $F$ has an exact join decomposition, then normal clusters in $F$ are the closed clusters for an idempotent co-closure operator on $F$ (as we have already remarked in the proof of the `if part'), and so direct images as well as joins of normal clusters are normal by the dual of Theorem~\ref{ThmA}. Consider a join $A\lor^F B$ of two conormal clusters. Then 
    \[A\lor^F B\ge^F \tau_\s (A\lor^F B)\ge^F\tau_\s(A)\lor^F\tau_\s( B)=A\lor^F B\]
  and so $A\lor^F B=\tau_\s(A\lor^F B)$, proving that $A\lor^F B$ is conormal.
\end{proof}

\begin{example}
  The form $F$ of subsets of sets has a join decomposition $F=F\vee F^\bot$, where $F^\bot$ is its subform of empty subsets. This decomposition is a particular instance of the one described in Theorem~\ref{ThmW}. By that theorem, it is an exact decomposition. In fact, combining Theorem~\ref{ThmN} and Theorem~\ref{LemE}, we easily get that this decomposition is the unique join decomposition of $F$. We can similarly get that $F$ has unique meet decomposition --- the one given by the dual of Theorem~\ref{ThmX}, which, by the same theorem, is a semiexact decomposition, but not an exact decomposition. 
\end{example}

\begin{example}
  It follows from Theorem~\ref{ThmN} and Theorem~\ref{LemE} that the form of equivalence relations has only one join decomposition, the one given by Theorem~\ref{ThmX}. By that theorem, this decomposition is not semiexact. Similarly, the form of equivalence relations has only one meet decomposition, the one given by the dual of Theorem~\ref{ThmW}, from which it follows that this decomposition is not semiexact.
\end{example}

\begin{example}\label{ExP}
  Consider the form $F$ of additive subgroups of rings from Example~\ref{Ex1}. The conormal hull $F_\c$ of $F$ is the form of subrings of rings, while the normal hull $F_\n$ of $F$ is the form of ideals of rings. Both of these forms are orean. Therefore, $F$ is strongly orean. At the same time, by Theorem~\ref{ThmM}, if $F$ had exact join/meet decomposition, then it would be possible to express every additive subgroup of every ring as a join/meet of a subring of that ring and an ideal in the ring. This is clearly not the case, as, for instance, the ring $\mathbb{Q}$ of rational numbers has many additive subgroups that are not subrings of $\mathbb{Q}$, while it only has two ideals --- $\{0\}$ and $\mathbb{Q}$.
\end{example}

\subsection*{Mixed decompositions of orean forms} Next, we study decompositions that are simultaneously of several different kinds.

\begin{lemma}\label{LemAG}
For a pair $(\kappa_\s,\kappa_\e)$ of operators and an orean form $F$ the following conditions are equivalent:
\begin{itemize}
\item $(\kappa_\s,\kappa_\e)$ is a left join decomposition of $F$.

\item $(\kappa_\s,\kappa_\e)$ is a join decomposition of $F$ and $F=F_\s$ (consequently, $F$ is conormal).

\item $F$ is conormal, $\kappa_\s$ is an identity operator of $F$ and $\kappa_\e$ is an idempotent co-closure operator on $F$ such that $F_\e$ is a normal orean form.   
\end{itemize}
\end{lemma}

\begin{proof}
This follows easily from Theorem~\ref{ThmAI} and Lemma \ref{LemE}.
\end{proof}

\begin{lemma}\label{LemAF}
For a pair $(\kappa_\s,\kappa_\e)$ of operators and an orean form $F$ the following conditions are equivalent:
\begin{itemize}
\item $(\kappa_\s,\kappa_\e)$ is a right join decomposition of $F$.

\item $(\kappa_\s,\kappa_\e)$ is a join decomposition of $F$ and $F=F_\e$ (consequently, $F$ is normal).

\item $F$ is normal, $\kappa_\e$ is an identity operator of $F$ and $\kappa_\s$ is an idempotent co-closure operator on $F$ such that $F_\s$ is a conormal orean form.   
\end{itemize}
\end{lemma}

\begin{proof}
This follows easily from Lemmas~\ref{LemE} and the dual of Theorem~\ref{ThmAI}.  
\end{proof}

\begin{theorem}\label{ThmQ}
  For an orean form $F$, the following conditions are equivalent:
  \begin{itemize}
    \item [(i)] $F$ has a left exact meet decomposition.
    
    \item[(ii)] $F$ is conormal and has an idempotent closure operator for which the closed clusters are the normal clusters.
    
    \item[(iii)] $F$ is conormal and each cluster has normal exterior.
    
    \item[(iv)] $F$ is conormal and has exact meet decomposition.
  \end{itemize}
\end{theorem}
\begin{proof}
  (i)$\implies$(ii): Let $F$ have an exact left meet decomposition $F= F\land F_\e$. Then $F$ is conormal and $F_\e $ is the subform of $F$ consisting of normal clusters by Remark~\ref{RemT} and the dual of Lemma~\ref{LemAF}. The rest of (ii) follows from Theorem~\ref{LemE}.

  (ii)$\implies$(iii) is trivial.

  (iii)$\implies$(iv) by the dual of Theorem~\ref{ThmM}.
  
  (iv)$\implies$(i) by the duals of Lemma~\ref{LemAF} and Theorem~\ref{ThmM}.
\end{proof}

\begin{remark}\label{RemG}
All examples of forms mentioned in Example~\ref{Ex0} are forms admitting exact left meet decomposition. This follows from an easy corollary of Theorem~\ref{ThmQ} stating that in general, an orean form of subobjects in a pointed category having cokernels always has an exact left meet decomposition. Normal clusters in this case are given by the usual normal subobjects. In fact, a pointed category with finite limits and colimits has a noetherian form admitting (a not necessarily left) exact meet decomposition if and only if the category is semi-abelian. Indeed, by Theorem~\ref{ThmM}, for such noetherian form, inverse images of conormal clusters are conormal. Furthermore, pointedness easily gives that the bottom clusters are conormal (this follows from the fact that the top cluster of the zero object is forced to be normal by exact meet decomposition, and since the zero object only has one quotient, it must be the bottom cluster). Then, by Theorem~2.17 in \cite{vanniekerk19}, the category is semi-abelian. Conversely, if a category is semi-abelian then, as recalled already in Example~\ref{Ex0}, its orean form $F$ of subobjects is noetherian. A semi-abelian category is pointed and has cokernels, so by our earlier remark, $F$ admits exact meet decomposition.   
\end{remark}

\begin{theorem}\label{ThmW}
  A conormal orean form $F$ always has a left join decomposition $F=F\vee F^\bot$. This decomposition is semiexact if and only if it is exact and if and only if $F$ is antinormal.
\end{theorem}

\begin{proof} 
If $F$ is conormal, then the join decomposition of Theorem~\ref{ThmY} applied to the pair $(F,F^\bot)$ becomes: 
\[F\times F^\bot=(F\times F^\bot)\vee (F^\bot\times F^\bot).\]
Via the first product projection $F\times F^\bot\to F$, which is an isomorphism (Lemma~\ref{LemS}), the join decomposition above produces the join decomposition $F=F\vee F^\bot$. It is a left decomposition by Lemma~\ref{LemAG}. The properties of the former decomposition attested by Theorem~\ref{ThmY} carry over to the latter decomposition to give:
\begin{itemize}
    \item The decomposition $F=F\vee F^\bot$ is semiexact if and only if $(F,F^\bot)$ is a semiexact pair.
    
    \item The decomposition $F=F\vee F^\bot$ is exact if and only if $F$ is antinormal (since $F^\bot$ is already an isoform --- Lemma~\ref{LemR}).
\end{itemize}
This leaves us to apply Lemma \ref{LemW} (along with Lemma~\ref{LemR}) to conclude that $(F,F^\bot)$ is a semiexact pair if and only if $F$ is antinormal.  
\end{proof}

\begin{theorem}\label{ThmX}
  A normal orean form $F$ always has a right join decomposition $F=F^\bot\vee F$. This decomposition is semiexact if and only if $F$ is anticonormal. It is exact if and only if $F$ is an isoform.
\end{theorem}

\begin{proof}
If $F$ is normal, then the join decomposition of Theorem~\ref{ThmY} applied to the pair $(F^\bot,F)$ becomes: 
\[F^\bot\times F=(F^\bot\times F^\bot)\vee (F^\bot\times F).\]
Via the second product projection $F^\bot\times F\to F$, which is an isomorphism (Lemma~\ref{LemS}), the join decomposition above produces the join decomposition $F=F^\bot\vee F$. It is a right decomposition by Lemma~\ref{LemAF}. The properties of the former decomposition attested by Theorem~\ref{ThmY} carry over to the latter decomposition to give:
\begin{itemize}
    \item The decomposition $F=F^\bot\vee F$ is semiexact if and only if $(F^\bot,F)$ is a semiexact pair.
    
    \item The decomposition $F=F^\bot\vee F$ is exact if and only if $F$ is an isoform (since $F^\bot$ is already antinormal --- Lemmas~\ref{LemR}).
\end{itemize}
This leaves us to apply the dual of Lemma~\ref{LemW} (along with Lemmas~\ref{LemR}) to conclude that $(F^\bot,F)$ is a semiexact pair if and only if $F$ is anticonormal.
\end{proof}

\begin{theorem}\label{ThmP}
  For an orean form $F\colon\mathbb{B}\to\mathbb{C}$, the following conditions are equivalent:
  \begin{itemize}
    \item[(i)] $F$ is binormal.
    
    \item[(ii)] $(1_\mathbb{B},1_\mathbb{B})$ is an exact, left, right, meet and join decomposition of $F$.
    
    \item[(iii)] $F$ has a left decomposition as well as a right decomposition.
    
    \item[(iv)] $F$ has an exact join decomposition and an exact meet decomposition.
    
    \item[(v)] $(1_\mathbb{B},1_\mathbb{B})$ is a decomposition of $F$.
  \end{itemize}
  Moreover, when $F\colon\mathbb{B}\to\mathbb{C}$ is binormal,
  $(1_\mathbb{B},1_\mathbb{B})$ is:
  \begin{itemize}
  \item[(a)] the only decomposition of $F$ that is both a left and a right decomposition;
  
  \item[(b)] the only exact decomposition of $F$.
  \end{itemize}
\end{theorem}
\begin{proof} The implications (i)$\implies$(ii) is easy, while (ii)$\implies$(iii),  (ii)$\implies$(iv),  (ii)$\implies$(v) are trivial. The implications (v)$\implies$(i) and (iii)$\implies$(i) are also easy. We prove (iv)$\implies$(i). If $F$ has an exact join decomposition, then direct images of normal clusters are normal by Theorem~\ref{ThmM}. If $F$ also has an exact meet decomposition, then top clusters are all normal.  Consequently, all conormal clusters are normal.  Dually, all normal clusters are also conormal. Then in each presentation $K=\tau_\s K\lor\tau_\e K$, we have $\tau_\s K=\tau_\e K$ and so $K=\tau_\s K=\tau_\e K$. This means that every cluster is both normal and conormal, and so $F$ is binormal.

As for the last claim in the theorem, (a) is obvious and (b) follows from Remark~\ref{RemT}. 
\end{proof}

\begin{theorem}\label{ThmAK}
An orean form $F$ has at most one exact decomposition. It is a left decomposition if and only if $(F_\c,F_\n)$ is a left exact pair.
\end{theorem}

\begin{proof}
Let $(\kappa_\s,\kappa_\e)$ be an exact decomposition of $F$. Then $F_\s=F_\c$ and $F_\e=F_\n$ by Remark~\ref{RemT}. By Theorem~\ref{ThmAJ}, if $(\kappa_\s,\kappa_\e)$ is a nullary decomposition, then $F$ is an isoform and it has no other decomposition. Moreover, by the same theorem, the exact decomposition is a left decomposition and since $F=F_\s=F_\c=F_\e=F_\n$ thanks to Lemma~\ref{LemX}, the pair $(F_\c,F_\n)$ is a left exact pair (it is, in fact, a biexact pair). So we can dismiss the nullary case entirely. 

Suppose $(\kappa_\s,\kappa_\e)$ is a unary decomposition. First we prove uniqueness. $(\kappa_\s,\kappa_\e)$ is either a left or a right decomposition. Suppose it is a left decomposition. Thanks to Theorem~\ref{ThmAI}, if $F$ has a possibly different exact decomposition, it can either be a right decomposition or a binary decomposition. If it is a right decomposition, then $F$ is binormal by Theorem~\ref{ThmP}. After this, applying Theorem~\ref{ThmAI} and its dual, as well as Lemma~\ref{LemN}, we can easily get that the two decompositions coincide. If $F$ has a different exact binary decomposition, then we can apply Theorem~\ref{ThmM} and its dual to conclude that this decomposition too matches with $(\kappa_\s,\kappa_\e)$ (once again, relying on Lemma~\ref{LemN}). By duality, if $(\kappa_\s,\kappa_\e)$ is a right decomposition, then it is again unique.

Now, consider again the case when $(\kappa_\s,\kappa_\e)$ is a left decomposition of $F$. Then it follows from Theorem~\ref{ThmAI} that $F=F_\c$ and the pair $(F,F_\n)$ is a left exact pair. The second claim of the theorem then follows trivially, in the case under consideration. Next, consider the case when $(\kappa_\s,\kappa_\e)$ is a right decomposition of $F$. Then by the dual of what we have just established, $(F_\c,F_\n)$ is a right exact pair. So, by the dual of Theorem~\ref{ThmAI}, $F=F_\n$ and the conormal operator $F_\c\to F$ is the subform inclusion. This means that the pair $(F_\c,F_\n)$ is a left exact pair if and only if $F_\c=F_\n=F$. But thanks to Theorem~\ref{ThmP}, the right exact decomposition $(\kappa_\s,\kappa_\e)$ is left exact if and only if $F_\c=F_\n=F$ as well. So the last claim of the theorem holds in this case too.  

Consider now the case when $(\kappa_\s,\kappa_\e)$ is a join decomposition. By what we have shown already, there can be no other exact nullary or exact unary decomposition. By Theorem~\ref{ThmM}, there can be no other exact join decomposition as well. If there is an exact meet decomposition, then, by Theorem~\ref{ThmP}, $F$ is binormal and both decompositions are given by the same pair of identity operators. The case when $(\kappa_\s,\kappa_\e)$ is a meet decomposition follows dually.

Finally, we prove the last claim of the theorem in the case when $(\kappa_\s,\kappa_\e)$ is a join decomposition. After this we will prove it in the case when $(\kappa_\s,\kappa_\e)$ is a meet decomposition (this does not follow by duality, although the argument will be similar).

Suppose $(\kappa_\s,\kappa_\e)$ is an exact join decomposition of $F$. Then $F_\s=F_\c$, $F_\e=F_\n$ and the operators $\kappa_\s$ and $\kappa_\e$ are (idempotent) co-closure operators, by Theorem~\ref{ThmM} and Lemma~\ref{LemE}. By Theorem~\ref{ThmAI}, $(\kappa_\s,\kappa_\e)$ is a left decomposition if and only if $\kappa_\s$ is an identity operator. On the other hand, the pair $(F_\c,F_\n)$ is semiexact by Remark~\ref{RemT}. So $(F_\c,F_\n)$ is a left exact pair if and only if the conormal operator $\alpha\colon F_\c\to F_\n$ is a left inverse of the normal operator $\beta\colon F_\n\to F_\c$. By Remark~\ref{RemO} and Lemma~\ref{LemN} with its dual, $\alpha$ and $\beta$ are given by $\alpha(A)=\kappa_\e(A)$ and $\beta(R)=\kappa_\s(R)$. If $\kappa_\s$ is an identity operator, then $\alpha(\beta(R))=\kappa_\e(R)=R$ and so, $\alpha$ is a left inverse of $\beta$ (actually, it follows anyway from Theorem~\ref{ThmAI} that if $(\kappa_\s,\kappa_\e)$ is a left decomposition then $(F_\c,F_\n)$ is a left exact pair). Conversely, if $\alpha$ is a left inverse of $\beta$, then for any $F$-cluster $C$, we have  \[\kappa_\s\kappa_\e(C)\le \kappa_\e(C)=\alpha\beta\kappa_\e(C)=\kappa_\e\kappa_\s\kappa_\e(C)\le \kappa_\s\kappa_\e(C).\]
This implies $\kappa_\s\kappa_\e(C)=\kappa_\e(C)$. Then, $\kappa_\e(C)=\kappa_\s\kappa_\e(C)\le \kappa_\s(C)$ and so
$C=\kappa_\s(C)\vee \kappa_\e(C)=\kappa_\s(C)$, proving that $\kappa_\s$ is an identity operator.

Suppose now $(\kappa_\s,\kappa_\e)$ is an exact meet decomposition of $F$. Then $F_\s=F_\c$, $F_\e=F_\n$ and the operators $\kappa_\s$ and $\kappa_\e$ are (idempotent) closure operators, by the duals of Theorem~\ref{ThmM} and Lemma~\ref{LemE}. By Theorem~\ref{ThmAI}, $(\kappa_\s,\kappa_\e)$ is a left decomposition if and only if $\kappa_\s$ is an identity operator. On the other hand, the pair $(F_\c,F_\n)$ is semiexact by Remark~\ref{RemT}. So $(F_\c,F_\n)$ is a left exact pair if and only if the conormal operator $\alpha\colon F_\c\to F_\n$ is a left inverse of the normal operator $\beta\colon F_\n\to F_\c$. By Remark~\ref{RemO} and Lemma~\ref{LemN} with its dual, $\alpha$ and $\beta$ are given by $\alpha(A)=\kappa_\e(A)$ and $\beta(R)=\kappa_\s(R)$. If $\kappa_\s$ is an identity operator, then $\alpha(\beta(R))=\kappa_\e(R)=R$ and so, $\alpha$ is a left inverse of $\beta$ (as before, it follows anyway from Theorem~\ref{ThmAI} that if $(\kappa_\s,\kappa_\e)$ is a left decomposition then $(F_\c,F_\n)$ is a left exact pair). Conversely, if $\alpha$ is a left inverse of $\beta$, then for any $F$-cluster $C$, we have  \[\kappa_\s\kappa_\e(C)\ge \kappa_\e(C)=\alpha\beta\kappa_\e(C)=\kappa_\e\kappa_\s\kappa_\e(C)\ge \kappa_\s\kappa_\e(C).\]
This implies $\kappa_\s\kappa_\e(C)=\kappa_\e(C)$. Then, $\kappa_\e(C)=\kappa_\s\kappa_\e(C)\ge \kappa_\s(C)$ and so
$C=\kappa_\s(C)\land \kappa_\e(C)=\kappa_\s(C)$, proving that $\kappa_\s$ is an identity operator.
\end{proof}

\begin{theorem}\label{ThmI}
  Let $\mathbb{C}$ be a category.
  \begin{itemize}
  \item[(i)] 
  Let $F$ be an orean form over $\mathbb{C}$ having a semiexact decomposition $(\kappa_\s,\kappa_\e)$. Then $F_\s$ is isomorphic to $F_\c$ and $F_\e$ is isomorphic to $F_\n$. So, $F$ is strongly orean. Moreover, $F$ satisfies (N2) if and only if $(F_\s ,F_\e )$ is an orean factorization.
  
  \item[(ii)] For any orean factorization $(F_\s,F_\e)$ of $\mathbb{C}$, there exists a form $F'$  satisfying (N2) with semiexact join decomposition $F'=F'_\s \vee F'_\e $ and isomorphisms $F'_\s\approx F_\s$ and $F'_\e\approx F_\e $. Namely, $F'=F_\s \times F_\e$ is such a form.
  \end{itemize}
\end{theorem}
\begin{proof} 
To prove (i), define an operator $F_\s\to F_\c$ by $\im^\s f\mapsto \im^F f$. That such operator exists and moreover, is an isomorphism, is a consequence of the following equivalence, which we will now establish: $\im^Ff\le\im^Fg\iff\im^\s f\le\im^\s g$.
Consider any two morphisms $f$ and $g$ (with the same codomain).  If $\im^Ff\le\im^Fg$, then by applying the conormal operator $\tau_\s $, we get $\im^\s f\le\im^\s g$.  If $\im^\s f\le\im^\s g$, then by applying the conormal operator $F_\s\to F_\e$, we get $\im^\e f\le\im^\e g$. By definition of decomposition, the operator $(\tau_\s,\tau_\e)\colon F\to F_\s\times F_\e$ is full. But since both $\tau_\s$ and $\tau_\e$ are conormal, we always have $(\tau_\s,\tau_\e) \im^F h=(\im^\s h,\im^\s h)$. So $\im^\s f\le\im^\s g$ and $\im^\e f\le\im^\e g$ together give $\im^Ff\le\im^Fg$. This proves that there is an isomorphism $F_\s\approx F_\c$. Dually, there is an isomorphism $F_\e\approx F_\n$. To get the last part of (i), apply Lemma~\ref{LemJ} (and Remark~\ref{RemO}) and its dual.

To prove (ii), consider an orean factorization $(F_\s,F_\e)$. By Theorem~\ref{ThmF}, $(F_\s,F_\e)$ is a semiexact pair. Let $F'=F_\s\times F_\e$ and consider the canonical join decomposition of the product, $F'=F'_\s\lor F'_\e$. By Lemma~\ref{LemS}, we have isomorphisms $F'_\s\approx F_\s$ and $F'_\e\approx F_\e$ and so $(F'_\s,F'_\e)$ is also a semiexact pair. By Theorem~\ref{ThmY}, the join decomposition $F'=F'_\s\lor F'_\e$ is semiexact. Moreover, it follows from Lemma~\ref{LemJ} that $(F'_\s,F'_\e)$ is an orean factorization, and so $F'$ satisfies (N2) by the first part of the present theorem.
\end{proof}

\subsection*{Exact decomposition of noetherian forms} Finally, we study exact decompositions of noetherian forms.

\begin{theorem}\label{ThmAH}
Let $F$ be a noetherian form. The following conditions are equivalent:
\begin{itemize}
    \item[(i)] $F$ has a left exact decomposition.
    
    \item[(ii)] $F$ is a strongly orean conormal form. 
    
    \item[(iii)] $F$ is a conormal form in which the largest clusters have normal interiors, while direct images of normal clusters have normal exteriors.
\end{itemize}
When $F$ satisfies these conditions, meets of normal $F$-clusters are normal. 
\end{theorem}

\begin{proof}
(i)$\implies$(ii) by Theorem~\ref{ThmAI}. (ii)$\implies$(iii) by the dual of Lemma~\ref{LemM}. In fact, these implications hold for any orean form. We prove (iii)$\implies$(i). Suppose (iii) holds. We first prove that meets of normal $F$-clusters are normal. Let $R$ and $S$ be normal clusters. Since $F$ is conormal, $R$ is conormal. Then $R\land S=\iota_R\cdot(S\cdot\iota_R)$. Since inverse image of a normal cluster is normal, this equality means that $R\land S$ is a direct image of a normal cluster. Then $R\land S$ must have normal exterior $\ncl{R\land S}\ge R\land S$. But at the same time, $R\ge \ncl{R\land S}$ and $S\ge \ncl{R\land S}$, since both $R$ and $S$ are normal and larger than $R\land S$. Then $\ncl{R\land S}=R\land S$, and so $R\land S$ is normal. 

Now, in light of Theorem~\ref{ThmAI}, to show (i), it is sufficient to prove the following:
\begin{itemize}
    \item[(a)] The subform $F_\n$ of normal $F$-clusters is orean.
    
    \item[(b)] There is a conormal operator $\alpha\colon F\to F_\n$.
    
    \item[(c)] $\alpha$ is a left inverse of the subform inclusion $F_\n\to F$.
\end{itemize}
We have (a) by the dual of Lemma~\ref{LemM} as well as the fact that join of normal clusters is normal (axiom (N3)), along with the fact that meet of normal clusters is normal (as we have established above). Consequently, $F$ is strongly orean (this will need a bit further below). Next, we prove (b). Since $F$ is conormal, a conormal operator $\alpha\colon F\to F_\n$, if it exists, is defined by the equality $\alpha(\im^F f)=\im^\n f$. This equality will define an operator if and only if $\im^F f\ge \im^F g$ always implies $\im^\n f\ge \im^\n g$. Suppose $\im^F f\ge \im^F g$. Decompose $f$ and $g$ as
\[f=\iota^F_{\im^F f}\circ \pi^\n_{\ker^F f},\quad g=\iota^F_{\im^F g}\circ \pi^\n_{\ker^F g}.\]
Thanks to Theorem~\ref{ThmT}, we can rewrite these decompositions as
\[f=\iota^\c_{\im^\c f}\circ \pi^\n_{\ker^\n f},\quad g=\iota^\c_{\im^\c g}\circ \pi^\n_{\ker^\n g},\]
where the embeddings and quotients come from the orean factorization $(F_\c,F_\n)=(F,F_\n)$. Since $\im^\c f=\im^F f\ge \im^F g=\im^\c g$, we get that $\iota^\c_{\im^\c g}=\iota^\c_{\im^\c f}\circ h$ for some morphism $h$. Then:
\begin{align*}
   \im^\n f &=\iota^\c_{\im^\c f}\cdot^\n (\pi^\n_{\ker^\n f}\cdot^\n \top^\n)\\
   &=\iota^\c_{\im^\c f}\cdot^\n \top^\n &\textrm{(Lemma~\ref{LemK})}\\
   &\ge \iota^\c_{\im^\c f}\cdot^\n (h\cdot^\n (\pi^\n_{\ker^\n g} \cdot^\n\top^\n))\\
   &=\im^\n g.
\end{align*}
Finally, to prove (c), we must show that for any normal cluster $R$, we have $\alpha(R)=R$. Since $R=\im^F\iota^F_R$, we have $\alpha(R)=\alpha(\im^F\iota^F_R)=\im^\n\iota^F_R$. By the dual of Lemma~\ref{LemM}, \[\im^\n\iota^F_R=\ncl{\iota^F_R\cdot^F \nccl{\top^F}}.\] Notice that in the domain of $\iota^F_R$, the largest cluster is the inverse image of $R$ along $\iota^F_R$, and so it is normal. We then get:
\[\alpha(R)=\im^\n\iota^F_R=\ncl{\im^F\iota^F_R}=\ncl{R}=R.\qedhere\]
\end{proof}

\begin{theorem}\label{ThmAO} For an orean form $F$ over a pointed category $\mathbb{C}$, the following conditions are equivalent:
\begin{itemize}
    \item[(i)] $\mathbb{C}$ is a Puppe-Mitchell exact category (or equivalently, a Grandis exact category) and $F$ is its form of normal subobjects. 
    
    \item[(ii)] $F$ is a binormal noetherian form.
    
    \item[(iii)] $F$ is a normal noetherian form admitting an exact meet decomposition.
    
    \item[(iv)] $F$ is a noetherian form admitting a right exact meet decomposition.
    
    \item[(v)] $F$ is an orean form satisfying (N2) and admitting a right exact meet decomposition.
    
    \item[(vi)] $F$ is a normal form satisfying (N2) and admitting an exact meet decomposition.
    
    \item[(vii)] $F$ is a binormal form satisfying (N2).    
 \end{itemize}
\end{theorem}

\begin{proof}
(i)$\implies$(ii): The form of normal subobjects in any Grandis exact category is a binormal noetherian form (see Remark~\ref{RemF}). 

(ii)$\implies$(iii) by Theorem~\ref{ThmP}. 

(iii)$\implies$(iv) follows from the dual of Theorem~\ref{ThmAH} and Remark~\ref{RemT}, as well as Theorem~\ref{ThmAK}.

(iv)$\implies$(v) is trivial.

(v)$\implies$(vi) follows from the definition of right decomposition.

(vi)$\implies$(vii): Suppose (vi) holds. Let $F_\s$ denote the subform of $F$ consisting of conormal clusters. Then the exact meet decomposition of $F$ is given by $F=F\land F_\s$, by Theorem~\ref{ThmM}. By the dual of Theorem~\ref{ThmI}, $(F_\s,F)$ is an orean factorization. By Remark~\ref{RemV}, the conormal operator $\alpha\colon F_s\to F$ and the normal operator $\beta\colon F\to F_\s$ are given by cokernel and kernel constructions in a pointed category. This easily implies that for any $F$-cluster $R$ we have $\alpha\beta (R)\le R$. On the other hand, since $F=F\land F_\s$ is an exact meet decomposition, we must have $\alpha\beta (R)\ge R$ for any $F$-cluster $R$, and moreover, $\beta\alpha (A)=A$ for any $F_\s$-cluster $A$. This gives that $F=F_\s$ (since $\alpha$ is an inclusion of forms), and hence $F$ is binormal, which proves (vii).

(vii)$\implies$(i) by Remark~\ref{RemF}.
\end{proof}

\begin{example}\label{Ex3}
Any category, where every morphism is a monomorphism and in addition, has finite meets and joins of subobjects, has an orean factorization $(F_\s,F_\e)$, where $F_\e$ is an isoform (see Remark~\ref{RemM}). Then $F_\s\times F_\e$ satisfies (N2) by Theorem~\ref{ThmI}. However, we have $F_\s\approx F_\s\times F_\e$ by Lemma~\ref{LemS}. So $F_\s$ satisfies (N2), by Remark~\ref{RemO}. Interestingly, $F_\s$ not only satisfies (N2), but also (N1) and (N3) and so is a noetherian form. Indeed, (N1) follows from Lemma~\ref{LemK}, while (N3) holds trivially since the form $F_\s$ is conormal and is antinormal (thanks to Lemma~\ref{LemW} and Theorem~\ref{ThmF}). This last remark together with Theorem~\ref{ThmW} and the dual of Theorem~\ref{ThmX} give that $F_\s$ has an exact join decomposition and a semiexact meet decomposition. Among such categories is the category $\mathbf{Fld}$ of fields. Any bounded lattice seen as a category is also such. Groupoids are such categories too. In fact, they are precisely those among such categories where $F_\s$ is also an isoform (as we know already from Remark~\ref{RemM}), and this will be the case if and only if $F_\s$ has an exact meet decomposition (apply Theorem~\ref{ThmX}), or equivalently, $F_\s$ is normal (apply Lemma~\ref{LemX}). The dual of this form (in the case when the category is neither pointed and nor is a groupoid, e.g., a non-trivial bounded lattice) thus gives a counterexample to the equivalence of the conditions (ii) and (iii) in Theorem~\ref{ThmAO} for a non-pointed category.
\end{example}

\begin{theorem}\label{ThmAL}
Let $F$ be a noetherian form. The following conditions are equivalent:
\begin{itemize}
    \item[(i)] $F$ has left exact join decomposition.
    
    \item[(ii)] $F$ is a strongly orean conormal form in which direct images of normal clusters are normal. 
    
    \item[(iii)] $F$ is a conormal form in which the largest clusters have normal interiors and direct images of normal clusters are normal.
\end{itemize}
When $F$ satisfies these conditions, $F$-clusters smaller than a normal cluster are normal. 
\end{theorem}

\begin{proof}
(i)$\implies$(ii) follows from Theorem~\ref{ThmM} and Theorem~\ref{ThmAH}, while (ii)$\implies$(iii) follows from Theorem~\ref{ThmAH}. Since (iii) implies \ref{ThmAH}(iii), to prove (iii)$\implies$(i) it suffices to show that (iii) implies that the left exact decomposition of $F$ guaranteed by Theorem~\ref{ThmAH} is a join decomposition. It follows from Theorem~\ref{ThmAI} that this left exact decomposition is the pair $(\kappa_\s,\kappa_\e)$, where $\kappa_\s$ is the identity operator of $F$ and $\kappa_\e$ is the composite of the unique conormal operator $\tau_\e\colon F=F_\c\to F_\n$ with the subform inclusion $F_\n\to F$ (with $F_\n$ being orean). To show that this is a join decomposition of $F$, we need to establish that the formula $S=S\lor \tau_\e S$ holds for any cluster $S$. In other words, that we always have $S\ge \tau_\e S$. Since $F$ and $\tau_\e$ are conormal, the values of $\tau_\e$ are given by $\tau_\e (\im^F f)=\im^n f$. So let $S=\im^F f$. According to Lemma~\ref{LemM}, $\im^n f=\ncl{f\cdot^F \nccl{\top^F}}$. Taking into account that direct images of normal clusters are normal by (iii), we get $\tau_\e (S)=\im^n f=f\cdot^F \nccl{\top^F}\le f\cdot^F \top^F=\im^F f=S$, as desired.

To prove the last statement of the theorem, consider clusters $N\ge C$, where $N$ is a normal cluster. Since $F$ is conormal, we have $C=\im f$, for some morphism $f$. We then have $N\cdot f=\top$ because of $N\ge C$. So the top cluster in the domain of $f$ is normal (inverse images of normal clusters are always normal). By (iii), this makes $C=\im f=f\cdot\top$ normal.
\end{proof}

\begin{theorem}
  \label{thmAA}
  Let $F$ be a noetherian form having exact join decomposition $F=F_\s\vee F_\e$. Then $F_\s=F_\c$, $F_\e=F_\n$ and the corresponding decomposition projections are given by $\tau_\s K=\cccl{K}$, $\tau_\e K=\nccl{K}$. Moreover, $(F_\s,F_\e)$ is then an orean factorization and hence a semiexact pair, and the values of the conormal operator $\alpha\colon F_\s\to F_\e$ and the normal operator $\beta\colon F_\e\to F_\s$ are given by $\alpha(A)=\nccl{A}$ and $\beta(R)=\cccl{R}$. We furthermore have: 
  \begin{itemize}
    \item[(i)]
    For any $F$-cluster $K$, we have
    \[\tau_\s K=(\tau_\s K)\ast (\tau_\e K)\quad\textrm{and}\quad\tau_\e K\ge\alpha(\tau_\s K).\]
    
    \item[(ii)] 
    If $A$ is a conormal $F$-cluster and $R$ is a normal $F$-cluster such that $A=A\ast R$ and $R\ge\alpha(A)$, then
    \[\tau_\s (A\lor^F R)=A\quad\textrm{and}\quad\tau_\e (A\lor^F R)=R.\]
    
    \item[(iii)] For $A$ and $R$ as above, we have
    \[(A\lor^F R)\cdot^F f=(A\cdot^\s f)\lor^F (R\cdot^\e  f),\]
    for any morphism $f$ whose codomain is the object in which $A$ and $R$ are clusters.

    \item[(iv)] $F_\s$ and $F_\e$ are both closed in $F$ under binary meets as well as under binary joins. Furthermore, a meet $R\land^F C$ of $F$-clusters is normal as soon as $R$ is.
      
    \item[(v)] The operators $\tau_\s$, $\tau_\e$,
    $\alpha$ and $\beta$ all preserve binary meets.
  \end{itemize}
\end{theorem}
\begin{proof} All equalities in the second sentence of the theorem follow from Theorem~\ref{ThmM} and Theorem~\ref{LemE}. By Theorem~\ref{ThmT} (or, alternatively, by Theorem~\ref{ThmI}), $(F_\s,F_\e)$ is an orean factorization. It is a semiexact pair thanks to Theorem~\ref{ThmF}, and so both $\alpha$ and $\beta$ exist. Thanks to Lemma~\ref{LemN},  Remark~\ref{RemO}, and the definition of exact join decomposition, $\alpha$ is given by the composite of the subform inclusion $F_\s\to F$ with $\tau_\e$. Hence, $\alpha(A)=\tau_\e(A)=\nccl{A}$. Similarly, $\beta(R)=\tau_\s(R)=\cccl{R}$. In the rest of the proof, we will use these facts without referring to them.

\textit{Proof of (i).} The definition of join decomposition, together with Lemma~\ref{inverseimageinteriorlem}, gives:
  \[\tau_\s K=\cccl{K}=\cccl{(\tau_\s K)\lor^F (\tau_\e K)}=(\tau_\s K)\ast (\tau_\e K).\]
  We also have $\tau_\e K=\nccl{K}\ge \nccl{\tau_\s K}= \alpha(\tau_\s K)$. 
  
  \textit{Proof of (ii).} Suppose $A$ is a conormal $F$-cluster and $R$ is a normal $F$-cluster such that $A=A\ast R$ and $R\ge\alpha(A)$. Then $\tau_\s (A\lor^F R)=A\ast R=A$ by Lemma~\ref{inverseimageinteriorlem}. It remains to show that $\tau_\e (A\lor^F R)=R$. Note that we have
    \[A\lor^F R=\tau_\s(A\lor^F R)\lor^F \tau_\e(A\lor^F R)=A\lor^F \tau_\e(A\lor^F R).\]
  The desired equality then follows from the second part of Lemma~\ref{inverseimageinteriorlem} (applied twice), as long as we can show that $A=A\ast \tau_\e(A\lor^F R)$ and $\tau_\e(A\lor^F R)\ge\alpha(A)$. By the first part of Lemma~\ref{inverseimageinteriorlem}, 
    \[A\ast \tau_\e(A\lor^F R)=\cccl{A\lor^F \tau_\e(A\lor^F R)}=\cccl{A\lor^F R}=A\ast R=A.\]
  Lastly, $\tau_\e(A\lor^F R)=\nccl{A\lor^F R}\ge \nccl{A}=\alpha(A)$.
  
  \textit{Proof of (iii).}  Applying (ii), we get:
    \[A\cdot^F f=\tau_\s (A\lor^F R)\cdot^F f\ge\tau_\s ((A\lor^F R)\cdot^F f)\ge\tau_\s (A\cdot^F f).\]
  Then
  \[\tau_\s (A\cdot^F f)=\cccl{A\cdot^F f}\ge \tau_\s((A\lor^F R)\cdot^F f).\]
  So along with Lemma~\ref{LemM} we have:
    \[\tau_\s ((A\lor^F R)\cdot^F f)=\tau_\s (A\cdot^F f)=A\cdot^\s f.\]
  A similar argument will give $\tau_\e ((A\lor^F R)\cdot^F f)=\tau_\e (R\cdot^F f)=R\cdot^\e f$. The equality in (iii) then follows by the definition of join decomposition.
  
  \textit{Proof of (iv).} That the meet of two conormal clusters is conormal and the join of two normal clusters is normal is stated in (N3). By Theorem~\ref{ThmM}, join of two conormal clusters is conormal. That the meet of a conormal cluster and a normal cluster is normal follows from (N1), the fact that any conormal cluster is the image of its embedding (which exists by (N2)), and the fact that direct and inverse images of normal clusters are normal (inverse images of normal clusters are always normal, while direct images are normal thanks to Theorem~\ref{ThmM}). We now prove that the meet of any cluster with a normal cluster is normal. Consider a meet $R\land^F C$ where $R$ is normal. Since $\tau_\s(R\land^F C)$ can be presented as a meet of a conormal cluster and a normal one, e.g.,  $\tau_\s(R\land^F C)=\tau_\s(R\land^F C)\land^F R$, we get that $\tau_\s(R\land^F C)$ is normal and hence, $\tau_\s(R\land^F C)\le \nccl{R\land^F C}=\tau_\e(R\land^F C)$. By definition of join decomposition we then get $R\land^F C=\tau_\s(R\land^F C)\vee^F \tau_\e(R\land^F C)=\tau_\e(R\land^F C)=\nccl{R\land^F C}$, and so $R\land^F C$ is normal. 
  
  \textit{Proof of (v).} Thanks to Theorem~\ref{LemE}, by (iv) and the dual of Remark~\ref{RemK}, we have that the subform inclusions $F_\s\to F$ and $F_\e\to F$, as well as the operators $\tau_\s$ and $\tau_\e$, preserve binary meets. Then so must $\alpha$ and $\beta$. 
\end{proof}

\begin{theorem}\label{ThmAB}
Consider two noetherian forms $F$ and $G$ over the same category, both having exact join decomposition. There are isomorphisms $F_\c \cong G_\c $ and $F_\n \cong G_\n $ if and only if there is an isomorphism $F\cong G$.
\end{theorem}

\begin{proof} It follows from Remark~\ref{RemO} that an isomorphism $F\cong G$ will restrict to isomorphisms $F_\c \cong G_\c $ and $F_\n \cong G_\n $, which proves the `if part' of the theorem. In what follows we prove the `only if part', where we use Remark~\ref{RemO} extensively, withought explicitly referring to it.  

Let $\gamma_\c\colon F_\c\to G_\c$ and $\gamma_\n\colon F_\n\to G_\n$ denote the two isomorphisms (they are unique, by Lemma~\ref{LemN}). Consider the operators
\begin{align*}
  \tau^F\colon F\to F_\c \times F_\n ,\quad & K\mapsto (\cccl{K},\nccl{K}),\\
  \nu^F\colon F_\c \times F_\n \to F,\quad & (A,R)\mapsto A\lor^F R,
\end{align*}
and similar ones, $\tau^G$ and $\nu^G$, for $G$.  With these operators, we construct operators 
\begin{align*}
  \nu^G\circ (\gamma_\c\times\gamma_\n)\circ \tau^F\colon & F\to G\\
  \nu^F\circ (\gamma_\c^{-1}\times\gamma_\n^{-1})\circ \tau^G\colon & G\to F.
\end{align*}
Since there is a unique isomorphism between two normal/conormal forms (see Lemma~\ref{LemN}), these operators are actually build symmetrically to each other. So, to show that they are inverses of each other, it is sufficient to prove that 
    \[\nu^F\circ (\gamma_\c^{-1}\times\gamma_\n^{-1})\circ \tau^G\circ\nu^G\circ (\gamma_\c\times\gamma_\n)\circ \tau^F\]
is the identity operator.  Consider any $F$-cluster and write $A=\cccl{K}$ and $R=\nccl{K}$.  We have:
  \[(\nu^G\circ (\gamma_\c\times\gamma_\n)\circ \tau^F)(K)=\gamma_\s A\lor^G\gamma_\e R.\]
Since $\gamma_\n$ is an isomorphism, thanks to the dual of Lemma~\ref{LemJ}, an $F$-quotient of $R$ is the same as a $G$-quotient of $\gamma_\n R$ --- let us denote it by $\pi_R$. Furthermore, by dual of Lemma~\ref{LemJ} again along with the fact that $F=F_\c\vee F_\n$ is the exact join decomposition of $F$ (Theorem~\ref{ThmM}), $\pi_R$ is also a $F_\n$-quotient of $R$ and similarly, $\pi_R$ is also a $G_\n$-quotient of $\gamma_\n R$. At this point, we employ Theorem~\ref{thmAA}, which allows us to use Wyler join. For the sake of brevity, let us not distinguish in notation between various constructions relative to the form $F$ and relative to the form $G$, including Wyler join for these two forms (more precisely, for the pair $(F_\c,F_\n)$ and for the pair $(G_\c,G_\n)$). We then have the following, where the last equality is by Theorem~\ref{thmAA}(i):
  \[(\gamma_\c A)\ast(\gamma_\n R)=(\pi_R\cdot^\c\gamma_\c A)\cdot ^\c\pi_R=\gamma_\c((\pi_R\cdot^\c A)\cdot^\c\pi_R)=\gamma_\c(A\ast R)=\gamma_\c A.\]
Thanks to Lemma~\ref{LemN} and Theorem~\ref{thmAA}(i) we further have:
  \[\gamma_\n R\ge^G \gamma_\n\alpha A=\alpha\gamma_\c A.\]
Then, by Theorem~\ref{thmAA}(ii),
  \[\cccl{\gamma_\c A\lor^G \gamma_\n R}=\gamma_\c A\quad\textrm{and}\quad
  \nccl{\gamma_\c A\lor^G \gamma_\n R}=\gamma_\n R.\]
These observations bring us to the following direct computation:
\begin{align*}
  \nu^F (\gamma_\c^{-1}\times\gamma_\n^{-1}(\tau^G(\nu^G(\gamma_\c\times\gamma_\e( \tau^F(K))))))
  &=\nu^F (\gamma_\c^{-1}\times\gamma_\n^{-1}(\tau^G(\nu^G(\gamma_\c\times\gamma_\n (A,R)))))\\
  &=\nu^F (\gamma_\c^{-1}\times\gamma_\n^{-1}(\tau^G(\nu^G( \gamma_\c A,\gamma_\n R))))\\
  &=\nu^F (\gamma_\c^{-1}\times\gamma_\n^{-1}(\tau^G( \gamma_\c A\vee^G\gamma_\n R)))\\ &=
  \nu^F (\gamma_\c^{-1}\times\gamma_\n^{-1}(\gamma_\c A,\gamma_\n R))\\
  &=\nu^F(A,R)\\
  &=A\vee^F R\\
  &=K.\qedhere
\end{align*}
\end{proof}

\begin{theorem}\label{ThmH}
  Given a pair $(F_\s ,F_\e )$ of forms over a category $\mathbb{C}$, there is a noetherian form $G$ over $\mathbb{C}$ having an exact join decomposition $G=G_\s\vee G_\e$, with $G_\s $ and $G_\e$ isomorphic to $F_\s $ and $F_\e $, respectively, if and only if the pair $(F_\s ,F_\e )$ is an orean factorization of $\mathbb{C}$ satisfying the following conditions (where $\alpha$ is the conormal operator $F_\s\to F_\e$ and $\beta$ is the normal operator $F_\e\to F_\s$):
  \begin{itemize}
    \item[(i)] $F_\e $ satisfies (N1).  Furthermore, for any morphism $f\colon X\to Y$, $F_\s$-cluster $A$ in $Y$ and $F_\e$-cluster $R$ in $Y$ such that $A\ast R=A$, $R\ge^\e \alpha(A)$, $\im^\s f\ge^\s A$ and $\im^\e f\ge^\e R$, we have $(f\cdot^\s (A\cdot^\s f))\ast R=A$.

    \item[(ii)] The conormal operator $\alpha\colon F_\s \to F_\e $ preserves binary meets. Furthermore, the formulas
    $A=A\ast\alpha(A)$ and
    $R\ge^\e \alpha\beta(R)$ hold for any $F_\s$-cluster $A$ and 
    $F_\e$-cluster $R$.
  \end{itemize} 
  Moreover, when $G$ exists, it is isomorphic to the subform of $F_\s\times F_\e$ consisting of those pairs $(A,R)$ for which $A\ast R=A$ and $\alpha(A)\le R$.
\end{theorem}

\begin{proof}
We first prove the `if' part of the theorem. Let $(F_\s ,F_\e )$ be an orean factorization of a category $\mathbb{C}$, satisfying the conditions (i,ii) of Theorem~\ref{ThmH}. 
Consider the operator $\kappa\colon F_\s\times F_\e\to F_\s\times F_\e$ defined by $(A,R)\mapsto (A\ast R,\alpha(A)\lor R)$. Monotonicity of this operator follows from monotonicity of join (Example~\ref{ExaO}), of Wyler join (Theorem~\ref{ThmR}), and of $\alpha$ (Theorem~\ref{ThmF}). By  Lemma~\ref{alphaastlemma}, (i,ii) and Remark~\ref{RemP}, we have:
  \begin{align*}
    \kappa\kappa(A,R)&=\kappa(A\ast R,\alpha(A)\lor R)\\
    &=((A\ast R)\ast (\alpha(A)\lor R),\alpha(A\ast R)\lor \alpha(A)\lor R)\\
    &=((A\ast R)\ast(\alpha(A\ast R)\lor R),\alpha(A)\lor R)\\
    &=((A\ast R)\ast R, \alpha(A)\lor R)\\
    &=(A\ast R, \alpha(A)\lor R)\\
    &=\kappa(A,R)\\
    &\ge (A,R).
  \end{align*}
This shows that $\kappa$ is an idempotent closure operator on $F_\s\times F_\e$. 
The closed clusters for $\kappa$ are those pairs $(A,R)$, where $A\ast R=A$ and $\alpha(A)\le R$. Note that thanks to Lemma~\ref{LemL}, for such a pair we further have 
  \[\beta(R)=\bot\ast R\le A\ast R=A,\]
 which we will use a few times below.

Let $G=(F_\s \times F_\e )_\kappa$.  In what follows we make use of Theorem~\ref{ThmA} several times. Firstly, it gives us that $G$ is an orean form and finite meets and inverse images of $G$-clusters are computed component-wise (i.e., as in $F_\s \times F_\e$). Furthermore, thanks to it (along with Remark~\ref{RemL}), we have:
\begin{align*}
  f\cdot^G(A,R) &=\kappa(f\cdot^\s A,f\cdot^\e R)\\
  &=((f\cdot^\s A) \ast (f\cdot^\e R), \alpha(f\cdot^\s A)\lor f\cdot^\e R)\\
  &=((f\cdot^\s A) \ast (f\cdot^\e R), f\cdot^\e\alpha(A)\lor f\cdot^\e R) & \textrm{(Theorem~\ref{LemN})}\\
  &=((f\cdot^\s A)\ast (f\cdot^\e R), f\cdot^\e (\alpha(A)\lor R)).\\
  &=((f\cdot^\s A)\ast (f\cdot^\e R), f\cdot^\e R).
\end{align*}
Using (ii) and conormality of $\alpha$, we further have:
\begin{align*}
  \im^G f & =f\cdot^G (\top^\s,\top^\e)\\
  &=((\im^\s f)\ast (\im^\e f),\im^\e f)\\
  &=((\im^\s f)\ast\alpha(\im^\s f),\alpha(\im^\s f))\\
  &=(\im^\s f, \alpha(\im^\s f)).
\end{align*}
Thus, the conormal $G$-clusters are pairs of the form $(A,\alpha(A))$. Next, thanks to Lemma~\ref{LemL} and (ii) (as well as Remark \ref{RemN}) we get that the bottom clusters for $G$ are given by the pairs
  \begin{align*} 
  \kappa(\bot^\s,\bot^\e)
  &=(\bot^\s \ast \bot^\e , \alpha(\bot^\s)\lor \bot^\e)\\
    &=(\beta(\bot^\e),\alpha\beta(\bot^\e)\lor \bot^\e)\\
    &=(\bot^\s, \bot^\e)\end{align*}
(note that we could also get $\bot^\s\ast\bot^\e=\bot^\s$ directly from Remark~\ref{RemP}). Using normality of $\beta$, we then get:
\begin{align*}
  \ker^G f &= (\bot^\s,\bot^\e )\cdot^Gf\\
  &=(\ker^\s f,\ker^\e f)\\&=(\beta(\ker^\e f),\ker^\e f).
\end{align*}
Thus, the normal $G$-clusters are pairs of the form $(\beta(R),R)$. Notice also that for a $G$-cluster $(A,R)$, we have:
\begin{align*}
  (A,\alpha(A))\lor^G(\beta(R),R)&=\kappa(A\vee^\s\beta(R),\alpha(A)\lor^\e R)\\
  &=((A\lor^\s\beta(R))\ast(\alpha(A)\lor^\e R), \alpha(A\lor^\s \beta(R))\lor^\e\alpha(A)\lor^\e R)\\
  &=(A\ast R,\alpha(A)\lor^\e R)\\
  &=(A,R).
\end{align*}  
By the shape of conormal and normal $G$-clusters established above, it is clear that for any $G$-cluster $(A,R)$, the cluster $(A,\alpha(A))$ is the largest conormal cluster below it and $(\beta(R), R)$ is the largest normal cluster below it. This fact together with the equality established above leaves us to prove validity of conditions (ii) and (iii) in Theorem~\ref{ThmM}, in order to be able to conclude that $G$ has an exact join decomposition. To prove \ref{ThmM}(ii) for $G$, consider a $G$-cluster $(B,S)$ in an object $Y$ and a morphism $f\colon X\to Y$. The largest normal $G$-cluster below $(B,S)$ is $(\beta(S),S)$. Using Lemma~\ref{LemN}, we get \[(\beta(S),S)\cdot^G f=(\beta(S)\cdot^\s f,S\cdot^\e f)=(\beta(S\cdot^\e f),S\cdot^\e f),\] which is the largest normal $G$-cluster below $(B\cdot^\s f,S\cdot^\e f)=(B,S)\cdot^G f$. Next, we prove \ref{ThmM}(iii). The largest normal $G$-cluster in the domain of $f$ is $(\beta(\top^\e),\top^\e)$.
We have:
  \begin{align*}
  \beta(f\cdot^\e \top^\e ) &=\beta(f\cdot^\e \top^\e )\ast (f\cdot^\e \top^\e) & \textrm{(Remark~\ref{RemP})}\\ &\ge (f\cdot^\s \beta(\top^\e ))\ast (f\cdot^\e \top^\e) \\ &\ge \bot^\s \ast (f\cdot^\e \top^\e)\\  &=\beta(f\cdot^\e \top^\e ),\end{align*}
This gives that
  \begin{align*}
  f\cdot^G(\beta(\top^\e ),\top^\e ) &=((f\cdot^\s \beta(\top^\e ))\ast (f\cdot^\e \top^\e),f\cdot^\e \top^\e )\\
  &=(\beta(f\cdot^\e \top^\e ),f\cdot^\e \top^\e )\\
  &=(\beta(\im^\e f),\im^\e f)\\
  &=(\beta(\im^\e f),\alpha(\im^\s f))
  \end{align*}
is the largest normal $G$-cluster below $\im^G f$. 

By Theorem~\ref{ThmM}, $G$ has an exact join decomposition $G=G_\s\lor G_\e$, where $G_\s$ is the subform of $G$ consisting of conormal $G$-clusters and $G_\e$ is the subform of $G$ consisting of normal $G$-clusters. By the shape of normal and conormal clusters for $G$, it is clear that we have isomorphisms $F_\s\approx G_\s$ and $F_\e\approx G_\e$ given by the mappings $A\mapsto (A,\alpha(A))$ and $R\mapsto (\beta(R),R)$.  
    
Finally, we will conclude the proof of the `if' part of the theorem by showing that $G$ is a noetherian form. $G$ satisfies (N2) thanks to Theorems~\ref{ThmI} and \ref{ThmS}. By (ii), $\alpha$ preserves meets, which easily implies that the meet of conormal $G$-clusters is conormal, and also by Theorem~\ref{ThmM}, the join of normal $G$-clusters is normal.  Thus $G$ satisfies (N3) as well.
We now want to prove that $G$ satisfies (N1). For this we will rely on Remark~\ref{RemB}. Consider a $G$-cluster $(A,R)$ below $\im^G f=(\im^\s f,\alpha(\im^\s f))=(\im^\s f,\im^\e f)$.  Applying (i), we get:
\begin{align*}
  f\cdot^G((A,R)\cdot^G f)
  &=\kappa(f\cdot^\s (A\cdot^\s f),f\cdot^\e (R\cdot^\e f))\\
  &=((f\cdot^\s (A\cdot^\s f))\ast (f\cdot^\e (R\cdot^\e f)), f\cdot^\e (R\cdot^\e f))\\
  &=((f\cdot^\s (A\cdot^\s f))\ast R,R)\\
  &=(A,R).
\end{align*}
So the meet formula of (N1) holds.  To show the join formula of (N1), consider any morphism $f$ and $G$-cluster $(A,R)$ above $\ker^G f=(\beta(\ker^\e f),\ker^\e f)=(\ker^\s f,\ker^\e f)$.  
Thanks to (i) and Lemma~\ref{alphaastlemma}, we have:
\begin{align*}
  (f\cdot^G(A,R))\cdot^G f &=(\kappa(f\cdot^\s A,f\cdot^\e R))\cdot^G f\\ &=(((f\cdot^\s A)\ast (f\cdot^\e  R))\cdot^\s f, (f\cdot^\e  R)\cdot^\e f)\\
    &=(A\ast R, R).\\
  &=(A, R).
\end{align*}  
We have thus proved that $G$ is a noetherian form. 

Next, we prove the `only if part' of the theorem. For this, it suffices to show that for any noetherian form $F$ having an exact joint decomposition $F=F_\s\vee F_\e$, the pair $(F_\s,F_\e)$ is an orean factorization for which the conditions (i-ii) of Theorem~\ref{ThmH} hold. This is because of Theorem~\ref{ThmS} as well as invariance of (i-ii) under isomorphisms of forms. So consider a noetherian form $F$ having an exact joint decomposition $F=F_\s\vee F_\e$. We will make use of Theorem~\ref{thmAA}, without explicitly referring to it. In particular, we readily get that $(F_\s,F_\e)$ is an orean factorization. We now want to prove that the conditions (i-ii) hold for the pair $(F_\s,F_\e)$. In what follows, in addition to Theorem~\ref{thmAA}, we will regularly make use of the fact that direct images of normal $F$-clusters are normal (which is attested by Theorem~\ref{ThmM}), and hence direct images in $F_\e$ are computed as in $F$, thanks to the dual of Lemma~\ref{LemM}. Note that the inverse images too are computed in $F_\e$ as in $F$. We then readily get that the join formula of (N1) holds for $F_\e$:
\begin{align*}
    (f\cdot^\e R)\cdot^\e f &= (f\cdot^F R)\cdot^\e f\\
    &=(f\cdot^F R)\cdot^F f\\
    &=R\vee^F \ker^F f\\
    &=R\vee^\e \ker^\e f.
\end{align*}
To prove the meet formula of (N1) for $F_\e$, we make use Remark~\ref{RemB}. Suppose $\im^\e f\ge^\e R$. Note that we have \[\im^\e f=f\cdot^\e \top^\e=f\cdot^\e \nccl{\top^F}=f\cdot^F \nccl{\top^F}=\nccl{\im^F f},\] where the last equality is due to Theorem~\ref{ThmM}(iii). So $\im^F f\ge^\e R$ and we have:
\begin{align*}
    f\cdot^\e (R\cdot^\e f) &=f\cdot^\e (R\cdot^F f) \\
    &=f\cdot^F (R\cdot^F f) \\
    &=R.
\end{align*}
To prove the second part of  (i), assume $R$ is a normal $F$-cluster and $A$ is a conormal $F$-cluster, such that $A\ast R=A$, $\alpha(A)\le R$. Assume furthermore that $f$ is a morphism such that $A\le\im^\s f(=\im^F f)$ and $R\le\im^\e f(\le \im^F f)$.  Then:
\begin{align*}
  (f\cdot^\s (A\cdot^\s f))\ast R&=\cccl{(f\cdot^\s (A\cdot^\s f))\lor^F R} &\textrm{(Lemma~\ref{inverseimageinteriorlem})}\\
  &=\cccl{(f\cdot^F(A\cdot^\s  f))\lor^F (f\cdot^F (R\cdot^F f))}\\
  &=\cccl{f\cdot^F((A\cdot^\s f)\lor^F (R\cdot^\e f))}\\
  &=\cccl{f\cdot^F((A\lor^F R)\cdot^F f)} &\textrm{(Theorem~\ref{thmAA})}\\
  &=\cccl{(A\lor^F R)\land^F \im^F f}\\
  &=\cccl{A\lor^F R}\\
  &=A & \textrm{(Lemma~\ref{inverseimageinteriorlem})}
\end{align*}
As for (ii), we just need to apply (i) and (v) from Theorem~\ref{thmAA}. This completes the proof of the `only if part'.

The last statement of the theorem follows from what we showed in the proof of the `if part' and from Theorem~\ref{ThmAB}. This completes the proof.
\end{proof}

\begin{remark}\label{RemI}
  In view of Remark~\ref{RemB}, the condition (i) in Theorem~\ref{ThmH} is a consequence of the following condition:
  \begin{itemize}
    \item[(i$'$)] $F_\e $ satisfies (N1) and $F_\s$ satisfies the meet formula in (N1).
  \end{itemize}
\end{remark}

\begin{remark}
  Note that Theorem~\ref{ThmH} provides a characterization of all noetherian forms admitting exact join decomposition, as starting with such a form $F$ we can consider the pair $(F_\s,F_\e)$ of its subforms of conormal/normal clusters, and then apply the theorem (along with Theorem~\ref{ThmM}). 
\end{remark}

\begin{remark}
While working on this paper,
Theorem~\ref{ThmH} went through a number of simplifications. The earlier versions had more requirements than those included in the current (i) and (ii). In the following theorem we show that these requirements can be simplified further.
\end{remark}

\begin{theorem}\label{ThmAC}
  An orean factorization $(F_\e,F_\s)$ of a category $\mathbb{C}$ satisfies the conditions (i) and (ii) of Theorem~\ref{ThmH}, if and only if it satisfied the following conditions: 
  \begin{itemize}
    \item[(i)] $F_\e$ satisfies (N1).  Furthermore, for any $F_\e$-projection $e$ and $F_\s$-cluster $A$ in the codomain of $e$, we have $(e\cdot^\s (A\cdot^\s e))\ast \alpha(A)=A\ast \alpha(A)=A$.

    \item[(ii)] The conormal operator $\alpha\colon F_\s \to F_\e $ preserves binary meets and bottom clusters.
  \end{itemize} 
\end{theorem}

\begin{proof}
To prove this theorem, it suffices to observe the following, in all of which $(F_\e,F_\s)$ is assumed to be an orean factorization:
\begin{itemize}
    \item[(a)] The second part of \ref{ThmAC}(i) is equivalent to the conjunction of the second part of \ref{ThmH}(i) and the law $A\ast\alpha(A)=A$ in \ref{ThmH}(ii).
    
    \item[(b)] Preservation of the bottom clusters by $\alpha$ is equivalent to the law $R\ge^\e \alpha\beta(R)$ required in the case when $R$ is a bottom cluster, which was the only instance of this law used in the proof of the `if part' of Theorem~\ref{ThmH}.
\end{itemize}
The first part of (b) is a rather obvious consequence of the fact that $\beta$ preserves bottom clusters (Remark~\ref{RemN}). The second part of (b) can be established by an inspection of the proof of the `if part' of Theorem~\ref{ThmH}. To prove (a), we first show that the second part of \ref{ThmAC}(i) follows from the second part of \ref{ThmH}(i) and the law $A\ast\alpha(A)=A$ in \ref{ThmH}(ii). 

Suppose the second part of \ref{ThmH}(i) and the law $A\ast\alpha(A)=A$ from \ref{ThmH}(ii) hold. Consider an $F_\e$-projection $e$ and $F_\s$-cluster $A$ in the codomain of $e$. We want to show that $(e\cdot^\s (A\cdot^\s e))\ast\alpha(A)=A$. Apply the second part of \ref{ThmH}(i) in the case when $f=e$ and $R=\alpha(A)$. We then get: $\im^\s e\ge^\s \alpha(A)$ and $\im^\e e\ge^\e \alpha(A)$ imply $(e\cdot^\s (A\cdot^\s e))\ast \alpha(A)=A$. So we just need to establish the inequalities $\im^\s e\ge^\s \alpha(A)$ and $\im^\e e\ge^\e \alpha(A)$. But we do have these inequalities, thanks to Lemma~\ref{LemK}.

Suppose now the second part of \ref{ThmAC}(i) holds. Then we get, at once, the law $A\ast\alpha(A)$ in 149(ii). Consider a morphism $f\colon X\to Y$, $F_\s$-cluster $A$ in $Y$ and $F_\e$-cluster $R$ in $Y$ such that $A\ast R=A$, $R\ge^\e \alpha(A)$, $\im^\s f\ge^\s A$ and $\im^\e f\ge^\e R$. The factorization $f=\iota^\s_{\im^\s f}\circ \pi^\e_{\ker^\e f}$ and the fact that $\im^\s f\ge^\s A$ along with diagram chasing and applying Lemma~\ref{LemK} several times can be used to check that we have the following commutative diagram, where $B=f\cdot^\s(A\cdot^\s f)$:
\[\xymatrix@=70pt{\bullet\ar[r]^-{\pi^\e_{(\ker^\e f)\cdot^\e \iota^\s_{A\cdot^\s f}}}\ar[d]_-{\iota^\s_{A\cdot^\s f}} & \bullet\ar[d]_-{\iota^\s_{\pi^\e_{\ker^\e f}\cdot^\s (A\cdot^\s f)}} \ar[r]^-{\iota^\s_{B\cdot^\s\iota^\s_A}} & \bullet\ar[d]^-{\iota^\s_A}\ar[dl]^-{\iota^\s_{A\cdot^\s\iota^\s_{\im^\s f}}}\\ \bullet\ar[r]_-{\pi^\e_{\ker^\e f}} & \bullet\ar[r]_-{\iota^\s_{\im^\s f}} & \bullet }\]
By Lemma~\ref{LemAC}, to show that $B\ast R=A$, it suffices to show that the composite $\pi^\e_{\top^\e}\circ \iota^\s_{B\cdot^\s\iota^\s_A}$ is an $F_\e$-quotient. Set $A'=A\cdot^\s\iota^\s_{\im^\s f}$ and $B'=e\cdot^\s(A'\cdot^\s e)$, where $e=\pi^\e_{\ker^\e f}$. Then $B'\cdot^\s\iota^\s_{A'}=B\cdot^\s\iota^\s_{A}$.
Since, by \ref{ThmAC}(i), \[(e\cdot^\s (A'\cdot^\s e))\ast \alpha(A')=A'\ast \alpha(A')=A',\]
we get (by Lemma~\ref{LemAC}) that the composite $\pi^\e_{\top^\e}\circ \iota^\s_{B\cdot^\s\iota^\s_A}=\pi^\e_{\top^\e}\circ \iota^\s_{B'\cdot^\s\iota^\s_{A'}}$ is indeed an $F_\e$-quotient. This proves that the second part of 149(i) holds as well. The proof of (a), and hence of the entire theorem, is now complete.
\end{proof}

\begin{example}
  Theorem~\ref{ThmH} invites the following question: can a category have at least two non-isomorphic noetherian forms admitting exact join decomposition? In other words, can a category have more than one factorization system such that the corresponding orean factorization satisfies the conditions (i-ii) in the theorem? The answer is yes, and here is a simple example of such category (which, by the way, has all limits and colimits). Consider the $2$-chain regarded as a category:
    \[\xymatrix{0\ar[r]^-{f} & 1\ar[r]^-{g} & 2 }\]
  It has a noetherian form admitting conormal exact join decomposition thanks to Example~\ref{Ex3}. This form can be pictured via the direct image mappings as follows: 
    \[\xymatrix{ &  & \top_2\\ & \top_1\ar@{|->}[r] & \bullet\ar@{-}[u]  \\ \bot_0=\top_0 \ar@{|->}[r] &  \bot_1\ar@{|->}[r]\ar@{-}[u] & \bot_2\ar@{-}[u] \\ 0\ar[r]_-{f} & 1\ar[r]_-{g} & 2 }\]
  The corresponding factorization system $(\mathcal{E},\mathcal{M})$ has no non-identity morphisms in $\mathcal{E}$  and $f,g\in \mathcal{M}$. 
  The following display describes clusters and the direct images maps for the form over the same category, which is also noetherian and admits non-conormal non-normal exact join decomposition: 
    
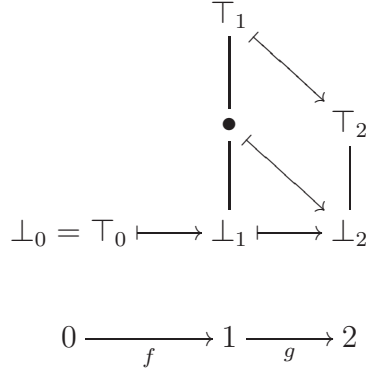
\[\xymatrix{ & \top_1\ar@{|->}[rd] & \\ & \bullet\ar@{|->}[rd]\ar@{-}[u] & \top_2  \\ \bot_0=\top_0 \ar@{|->}[r] &  \bot_1\ar@{|->}[r]\ar@{-}[u] & \bot_2\ar@{-}[u] \\ 0\ar[r]_-{f} & 1\ar[r]_-{g} & 2 }\]
  The corresponding factorization system $(\mathcal{E},\mathcal{M})$ has $g\in \mathcal{E}$  and $f,gf\in \mathcal{M}$.
\end{example}  

\begin{example}\label{ExaQ}
  Recall that a category is `balanced' when in the category, every morphism that is both an epimorphism and a monomorphism is automatically an isomorphism. A balanced category has at most one proper factorization system: the one given by the pair $(\mathcal{E},\mathcal{M})$, where $\mathcal{E}$ is the class of all epimorphisms and $\mathcal{M}$ is the class of all monomorphisms. In view of Theorems~\ref{ThmH} and \ref{ThmV}, a balanced category can have, up to isomorphism, at most one noetherian form admitting exact join decomposition. The category of sets is balanced. The conditions of Theorem~\ref{ThmH} for the form $F_\s$ of subsets and the form $F_\e$ of equivalence relations can be easily verified after Remark~\ref{RemI} and Examples~\ref{ExaR} and \ref{ExaAA}. By application of Theorem~\ref{ThmH} it then follows that the category of sets has (up to an isomorphism, unique) noetherian form having exact join decomposition, for which the clusters in a set $X$ are given by pairs $(A,R)$ where $R$ is an equivalence relation on $X$ and $A$ is either the empty set or one of the equivalence classes of $R$. In this form, $(B,S)\ge_f (A,R)$ if and only if $f(A)\subseteq B$ and $f(x)Sf(y)$ for all $xRy$.
\end{example}

\begin{theorem}\label{ThmAF}
Consider an orean factorization $(F_\s,F_\e)$ of a category $\mathbb{C}$, such that (N1) holds for the form $F_\e$. The following conditions are then equivalent (where $\alpha$ denotes the conormal operator $\alpha\colon F_\s\to F_\e$):
\begin{itemize}
    \item[(a)] $\alpha$ preserves binary meets.
    
    \item[(b)] Given a pullback 
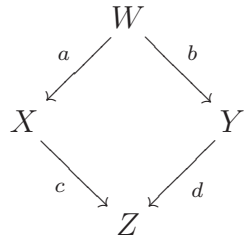
\[\xymatrix@!=15pt{ & W\ar[dl]_-{a}\ar[dr]^-{b} & \\ X\ar[dr]_-{c} & & Y\ar[dl]^-{d} \\ & Z & }\]
    of $F_\s$-embeddings, we have $b\cdot^\e(R\cdot^\e a)=(c\cdot^\e R)\cdot^\e d$ for any $F_\e$-cluster $R$ in $X$. 
    
    \item[(c)] Given a pullback 
    of $F_\s$-embeddings as shown above, we have $b\cdot^\e(\top^\e_X\cdot^\e a)=(c\cdot^\e \top^\e_X)\cdot^\e d$.     
    
    \item[(d)] $\alpha$ preserves inverse images along each $F_\s$-embedding.
\end{itemize}
\end{theorem}

\begin{proof} $\alpha$ exists by Theorem~\ref{ThmF}.

(a)$\implies$(b): Consider a pullback as in the statement of (b) above, where $c$ and $d$ are $F_\s$-embeddings. Then $a$ and $b$ are also embeddings by Lemma~\ref{LemP}. In particular, if $A,B,C,D$ denotes images of $a,b,c,d$, respectively, then $a$ is an embedding of $A=D\cdot^\s c$ and $b$ is an embedding of $B=C\cdot^\s d$. 
Thanks to Lemma~\ref{LemK}, the composite $ca=db$ is an embedding of \[c\cdot^\s A=c\cdot^\s (D\cdot^\s c)=C\land^\s D=d\cdot^\s (C\cdot^\s d)=d\cdot^\s B.\]
Suppose (a) holds. Then \[\im^\e ac=\alpha(C\land^\s D)=\alpha(C)\land^\e \alpha(D)=\im^\e c\land^\e \im^\e d.\]
Thanks to Lemma~\ref{LemK} and (N1) for $F_\e$, given an $F_\e$-cluster $R$ in $X$, we have:
\begin{align*}
    b\cdot^\e(R\cdot^\e a) &= b\cdot^\e(((c\cdot^\e R)\cdot^\e c)\cdot^\e a)\\
    &=b\cdot^\e((c\cdot^\e R)\cdot^\e ca)\\
    &=(d\cdot^\e (b\cdot^\e((c\cdot^\e R)\cdot^\e ca)))\cdot^\e d\\
    &=(db\cdot^\e((c\cdot^\e R)\cdot^\e ca))\cdot^\e d\\
    &=(ca\cdot^\e((c\cdot^\e R)\cdot^\e ca))\cdot^\e d\\
    &=((c\cdot^\e R)\wedge^\e \im^\e ca)\land^\e d\\
    &=((c\cdot^\e R)\land^\e \im^\e c\land^\e \im^\e d)\cdot^\e d\\
    &=((c\cdot^\e R)\land^\e \im^\e d)\cdot^\e d\\ 
    &=(c\cdot^\e R)\cdot^\e d\land^\e \top^\e_Y\\ 
    &=(c\cdot^\e R)\cdot^\e d
\end{align*}

(b)$\implies$(c) is trivial.

(c)$\implies$(d): Suppose (c) holds. Consider an $F_\s$-embedding $d\colon Y\to Z$ and an $F_\s$-cluster $C$ in $Z$. We want to prove that $\alpha( C\cdot^\s d)=\alpha(C)\cdot^\e d$. Let $c$ denote an $F_\s$-embedding of $C$. By Lemma~\ref{LemP}, there is a pullback as displayed in (b), where $b$ is an $F_\s$-embedding of $B=C\cdot^\s d$. Applying Lemma~\ref{LemK}, we can conclude that (a) is also an $F_\s$-embedding. Then, by (c) and Theorem~\ref{ThmF}, we have:
\[\alpha( C\cdot^\s d)=b\cdot^\e\top^\e_W=b\cdot^\e(\top^\e_X\cdot^\e a)=(c\cdot^\e \top^\e_X)\cdot^\e d=\alpha(C)\cdot^\e d.\]

(d)$\implies$(a): Suppose (d) holds. To prove (a), consider two conormal $F_\s$-clusters $C$ and $D$ in an object $Z$. Let $c$ denote an $F_\s$-embedding of $C$ and let $d$ denote an $F_\s$-embedding of $D$. Using Lemma~\ref{LemP}, we can build a pullback as shown in (b), where $b$ is an embedding of $B=C\cdot^\e d$. By Lemma~\ref{LemK}, the composite $db=ac$ is an embedding of $d\cdot^\s B=C\land^\s D$.  
By (d) we then have:
\begin{align*}
\alpha(C\land^\s D) &= db\cdot^\e \top^\e_W\\
&=d\cdot^\e (b\cdot^\e \top^\e_W)\\
&=d\cdot^\e \alpha(B)\\
&=d\cdot^\e \alpha(C\cdot^\s d)\\
&=d\cdot^\e (\alpha(C)\cdot^\e d)\\
&=\alpha(C)\land^\e \im^\e d\\
&=\alpha(C)\land^\e \alpha(D) & \textrm{(Theorem~\ref{ThmF})} & \qedhere
\end{align*}
\end{proof}

\section{Bicategories}\label{SecK}

\subsection*{Bicategories underlying emd-noetherian forms} By an \emph{emd-noetherian} form we mean a noetherian form admitting an exact meet decomposition. In this section we describe these forms in a language that originates from \cite{M50}. 

Inspired by the terminology from \cite{M50,I58}, let us call a triple $(\mathbb{C},\mathcal{E},\mathcal{M})$, where $\mathbb{C}$ is a category and $\mathcal{E}$, $\mathcal{M}$ are classes of morphisms in $\mathbb{C}$, a \emph{bicategory}. Note that there is an entirely different notion of bicategory in higher dimensional category theory, introduced first in \cite{Ben67}. A bicategory in the sense of Isbell \cite{I58} is our bicategory, such that $(\mathcal{E},\mathcal{M})$ is a proper factorization system \cite{FK72}. In this section we are interested in bicategories where $\mathcal{E}$ and $\mathcal{M}$ are classes of quotients and embeddings, respectively, for some noetherian form $F$ over $\mathbb{C}$. By the results obtained in \cite{GJ19} it follows that all such bicategories are Isbell bicategories. 

In a general bicategory $(\mathbb{C},\mathcal{E},\mathcal{M})$, a \emph{right morphism} refers to a morphism from the class $\mathcal{E}$. It is displayed as an arrow with a double arrow head:
\[\xymatrix{\bullet\ar@{->>}[r] & \bullet}\]
A \emph{left morphism} refers to a morphism from the class $\mathcal{M}$, denoted as:
\[\xymatrix{\bullet\ar@{>->}[r] & \bullet}\]
By a \emph{terminal right morphism} we mean a right morphism $e\colon X\to Y$ such that for any other quotient $e'\colon X\to Y'$, there is a unique morphism $u\colon Y\to Y'$ such that $ue'=e$. On a diagram, we mark a terminal right morphism with an asterisk:
\[\xymatrix{\bullet\ar@{->>}[r]|-{*} & \bullet}\] The universal property of the terminal right morphism can thus be visualized as:
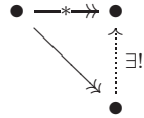
\[\xymatrix{\bullet\ar@{->>}[r]|-{*}\ar@{->>}[dr]_-{} & \bullet\\ & \bullet\ar@{..>}[u]_-{\exists!}}\]
The dual property, where the class $\mathcal{E}$ is replaced with the class $\mathcal{M}$, defines an \emph{initial left morphism}:
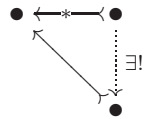
\[\xymatrix{\bullet\ar@{<-<}[r]|-{*}\ar@{<-<}[dr]_-{} & \bullet\\ & \bullet\ar@{<..}[u]_-{\exists!}}\]

Given a bicategory $(\mathbb{C},\mathcal{E},\mathcal{M})$, the \emph{dual bicategory} is the bicategory $(\mathbb{C}^\mathsf{op},\mathcal{M},\mathcal{E})$. Dual of a statement inside a bicategory is, accordingly, the usual categorical dual where the roles of morphism classes $\mathcal{E}$ and $\mathcal{M}$ are swapped around.

When we talk of a prenoetherian form of a bicategory $(\mathbb{C},\mathcal{E},\mathcal{M})$, we necessarily refer to a noetherian form $F$ over $\mathbb{C}$, such that $\mathcal{E}$ is the class of $F$-quotients and $\mathcal{M}$ is the class of $F$-embeddings. Note that if $F$ is a prenoetherian form of a bicategory, then $F^\mathsf{op}$ will be a prenoetherian form of the dual bicategory, and vice versa.

By an \emph{orean bicategory} we mean a bicategory $(\mathbb{C},\mathcal{E},\mathcal{M})$ such that $(\mathcal{E},\mathcal{M})$ is a proper factorization system corresponding to an orean factorization by Theorem~\ref{ThmV}. This notion comes very close to Mac~Lane's `lattice-ordered category' \cite{M50}. In light of Theorem~\ref{ThmV}, a bicategory is orean if and only if we have:
\begin{itemize}
    \item[(B0)] The classes of right and left morphisms form a proper factorization system $(\mathcal{E},\mathcal{M})$ such that for any object $X$, the posets of $\mathcal{E}$-quotients and $\mathcal{M}$-subobjects are lattices and in addition, along arbitrary morphisms, pushouts of right morphisms as well as pullbacks of left morphisms exist (it follows from the factorization system axioms that $\mathcal{M}$ is stable under pullbacks along arbitrary morphisms and $\mathcal{E}$ is stable under pushouts along arbitrary morphisms). 
\end{itemize}
Since in any factorization system, the intersection $\mathcal{E}\cap\mathcal{M}$ is the class of isomorphisms, an orean bicategory will have its underlying category balanced if and only if all epimorphisms are right morphisms and all monomorphisms are left morphisms. We refer to such orean bicategories as \emph{balanced orean bicategories}.

Note that the dual of an orean bicategory is an orean bicategory. Similarly, the notion of a balanced orean bicategory is self-dual. 

In a bicategory $(\mathbb{C},\mathcal{E},\mathcal{M})$, when we speak of \emph{subobjects}, we mean $\mathcal{M}$-subobjects, and when we speak of \emph{quotients}, we mean $\mathcal{E}$-quotients.

We will make use of, without reference, the following properties of orean bicategories, which are in fact well known properties of proper factorization systems:
\begin{itemize}
    \item Every left morphism is a monomorphism and dually, every right morphism is an epimorphism.
    
    \item Every split monomorphism is a left morphism. Dually, every split epimorphism is a right morphism.
    
    \item A morphism is an isomorphism if and only if it is both a left and a right morphism.
    
    \item Composites of left morphisms are left morphisms and dually, composites of right morphisms are right morphisms.
    
    \item If a composite $fg$ is a left morphism, then so is $g$. Dually, if a composite $fg$ is a right morphism, then so is $f$.
    
    \item If $fe=mg$, where $e$ is a right morphism and $m$ is a left morphism, then $g=de$ and $f=md$, for a unique morphism $d$. If furthermore $f$ is a left morphism and $g$ is a right morphism, then $d$ is an isomorphism.
    
    \item Any morphism $f$ factors as $f=me$ where $e$ is a right morphism and $m$ is a left morphism.
    
    \item Left morphisms are stable under pullbacks and right morphisms are stable under pushouts (as remarked already in the formulation of (B0)).
\end{itemize}
The following property of initial left morphisms and terminal right morphisms in an orean bicategory will be freely used as well, without reference. This property is close to obvious, once the properties above are taken into account.
\begin{itemize}
    \item Any morphism whose codomain is the domain of an initial left morphism is necessarily a right morphism. Dually, any morphism whose domain is the codomain of a terminal right morphism is necessarily a left morphism.
\end{itemize}
This suggests the following notions. An object is said to be \emph{left trivial} if it is the domain of an initial left morphism. The dual notion is that of a \emph{right trivial} object: it is an object that appears as a codomain of a terminal right morphism. The property above says that any morphism into a left trivial object is a right morphism and dually, any morphism out of a right trivial object is a left morphism. It is not difficult to see that, under (B0), these properties characterize left and right trivial objects:
\begin{itemize}
    \item An object $X$ is left trivial if and only if any morphism with codomain $X$ is a right morphism. Dually, $X$ is right trivial if and only if any morphism with domain $X$ is a left morphism.
\end{itemize}
We also have another characterization:
\begin{itemize}
    \item An object $X$ is left trivial if and only if any left morphism with codomain $X$ is an isomorphism. Dually, $X$ is right trivial if and only if any right morphism with domain $X$ is an isomorphism.
\end{itemize}

Our goal now is to translate the duals of conditions (i) and (ii) in Theorem~\ref{ThmAC} in purely bicategorical terms. In this process, we rely on how the orean structure of the forms of quotients and subobject (e.g., direct and inverse images) are obtained using factorization of a morphism $f=me$ into a right morphism $e$ followed by a left morphism $m$, as well as using pushouts and pullback of left and right morphisms (see Examples~\ref{ExaH} and \ref{ExaF}). In particular:
\begin{itemize}
    \item Relative to the form of subobjects, the inverse image of an subobject, represented by a left morphism $m$, along a morphism $f$, is given by the subobject represented by a pullback $m'$ of $m$ along $f$.
    
    \item Relative to the form of subobjects, the direct image of an subobject, represented by a left morphism $m$, along a morphism $f$, is given by the subobject represented by a morphism $m'$ in the factorization $fm=m'e$ of the composite $fm$ into a right morphism $e$ followed by a left morphism $m'$.
    
    \item Relative to the form of subobjects, the meet of two subobjects represented by left morphisms $m$ and $m'$ is given by the composite $mm''$, where $m''$ is a pullback of $m'$ along $m$. 
    
    \item The largest subobject of an object $X$ is the subobject of $X$ represented by the identity morphism $1_X$, which the smallest subobject of $X$ is given by the subobject represented by the initial left morphism.  
    
    \item Relative to the form of quotients, the inverse image of an quotient, represented by a right morphism $e$, along a morphism $f$, is given by the quotient represented by a morphism $e'$ in the factorization $ef=me'$ of the composite $ef$ into a right morphism $e'$ followed by a left morphism $m$.  
    
    \item Relative to the form of quotients, the direct image of an quotient, represented by a right morphism $e$, along a morphism $f$, is given by the quotient represented by a pushout $e'$ of $e$ along $f$.

    \item Relative to the form of quotients, the join of two quotients represented by right morphisms $e$ and $e'$ is given by the composite $ee''$, where $e''$ is a pushout of $e'$ along $e$.
    
    \item The smallest quotient of an object $X$ is the quotient of $X$ represented by the identity morphism $1_X$, while the largest quotient of $X$ is given by the quotient represented by the terminal right morphism.
\end{itemize}

In light of these remarks along with Theorem~\ref{ThmAE} and Remark~\ref{RemB}, we get that for an orean bicategory, the form  $F_\s$ of subobjects satisfies (N1) if and only if the following holds:
\begin{itemize}
    \item[(B1)] Pullback of a right morphism along a left morphism is always a right morphism and given a commutative diagram
    
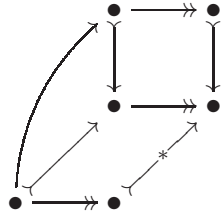
\[\xymatrix{ & \bullet\ar@{>->}[d]\ar@{->>}[r] & \bullet\ar@{>->}[d] \\ & \bullet\ar@{->>}[r] & \bullet \\ \bullet\ar@{>->}[ur]\ar@{->>}[r]\ar@/^10pt/[uur] & \bullet\ar@{>->}[ur]|-{*} & }\]
    where the parallelogram is a pullback, the square must also be a pullback.
\end{itemize}
Under (B0), this property has the following easy reformulation:
\begin{itemize}
    \item[(B1$'$)] Pullback of a right morphism along a left morphism is always a right morphism and given a commutative diagram
    
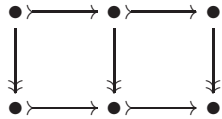
\[\xymatrix{\bullet\ar@{->>}[d]\ar@{>->}[r] & \bullet\ar@{->>}[d]\ar@{>->}[r] & \bullet\ar@{->>}[d] \\ \bullet\ar@{>->}[r] & \bullet\ar@{>->}[r] & \bullet}\]
    if the outer rectangle is a pullback then the square on the right is a pullback as well (equivalently, the outer rectangle is a pullback if and only if the two squares are). 
\end{itemize}
When (B0) and the first part (B1) holds, the law $R= \beta(R)\astop R$ from the dual of (i) in Theorem~\ref{ThmAC}, for the pair $(F_\s,F_\e)$, where $F_\s$ is as before and $F_\e$ is the form of quotients, is equivalent to the following property:
\begin{itemize}
    \item[(B2)] Any pullback
    
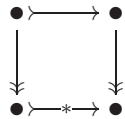
\[\xymatrix{ \bullet\ar@{->>}[d]\ar@{>->}[r] & \bullet\ar@{->>}[d] \\ \bullet\ar@{>->}[r]|-{*} & \bullet}\]
    is a pushout. 
\end{itemize}
This is actually easily equivalent to the following stronger property (the equivalence does not require (B1)):
\begin{itemize}
    \item[(B2$'$)] Any pullback
    
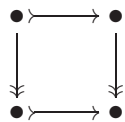
\[\xymatrix{ \bullet\ar@{->>}[d]\ar@{>->}[r] & \bullet\ar@{->>}[d] \\ \bullet\ar@{>->}[r] & \bullet}\]
    is a pushout. 
\end{itemize}
Thanks to the dual of Lemma~\ref{LemAC}, the dual of the remaining part of (i) in Theorem~\ref{ThmAC} for the same pair $(F_\s,F_\e)$ as above, under (B0) and (B2), can be equivalently reformulated as follows:
\begin{itemize}
    \item[(B3)] In a commutative diagram
    
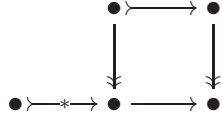
\[\xymatrix{  & \bullet\ar@{->>}[d]\ar@{>->}[r] & \bullet\ar@{->>}[d] \\ \bullet\ar@{>->}[r]|-{*} & \bullet\ar[r] & \bullet }\]
    where the right square is a pushout, the composite of bottom two morphisms is a left morphism. 
\end{itemize}
Thanks to the dual of Lemma~\ref{LemAD}, under (B0), preservation of top clusters by the normal operator $\beta\colon F_\e\to F_\s$ is equivalent to:
\begin{itemize}
\item[(B4)] Every right trivial object is left trivial.
\end{itemize}
Finally, thanks to Theorem~\ref{ThmAF}, under (B0) and (B1), preservation of binary joins by $\beta$ is equivalent to:
\begin{itemize}
\item[(B5)] For a commutative diagram
    
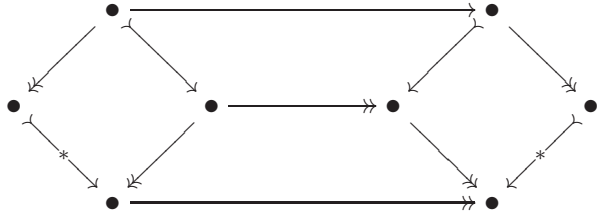
\[\xymatrix{ & \bullet\ar@{->>}[dl]\ar@{>->}[dr]\ar[rrrr] & & & & \bullet\ar@{>->}[dl]\ar@{->>}[dr] & \\ \bullet\ar@{>->}[dr]|-{*} & & \bullet\ar@{->>}[dl]\ar@{->>}[rr] & & \bullet\ar@{->>}[dr] & & \bullet\ar@{>->}[dl]|-{*} \\  & \bullet\ar@{->>}[rrrr] & & & & \bullet & }\] the top horizontal arrow is a right morphism, provided the diamonds are pullbacks and the bottom trapezium is a pushout. 
\end{itemize}
The condition above has an equivalent reformulation (under just (B0)):
\begin{itemize}
\item[(B5$'$)] For a commutative diagram
    
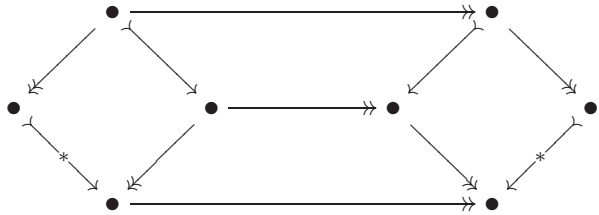
\[\xymatrix{ & \bullet\ar@{->>}[dl]\ar@{>->}[dr]\ar@{->>}[rrrr] & & & & \bullet\ar@{>->}[dl]\ar@{->>}[dr] & \\ \bullet\ar@{>->}[dr]|-{*} & & \bullet\ar@{->>}[dl]\ar@{->>}[rr] & & \bullet\ar@{->>}[dr] & & \bullet\ar@{>->}[dl]|-{*} \\  & \bullet\ar@{->>}[rrrr] & & & & \bullet & }\] the right diamond is a pullback provided the left diamond is a pullback and the bottom trapezium is a pushout. 
\end{itemize}
Theorem~\ref{ThmH} actually gives an explicit computation of the noetherian form over a bicategory having an exact meet decomposition. Thanks to the dual of Lemma~\ref{LemAC}, it can be presented as a subform of the form of subquotients (i.e., the form of $(\mathcal{E},\mathcal{M})$-subquotients from Example~\ref{ExaAC}, where $\mathcal{E}$ is the class of right morphisms and $\mathcal{M}$ is the class of left morphisms), given by those subquotients $[e,m]$ where the pushout of $m$ along $e$ is an initial left morphism:
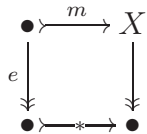
\[\xymatrix{ \bullet\ar@{->>}[d]_-{e}\ar@{>->}[r]^-{m} & X\ar@{->>}[d] \\ \bullet\ar@{>->}[r]|-{*} & \bullet}\]
Altogether, along with Theorem~\ref{ThmH} and \ref{ThmAC}, we obtain:

\begin{theorem}\label{ThmAD}
  There is a bijection between isomorphism classes of emd-noetherian forms over a category (i.e., noetherian forms admitting exact meet decomposition), and bicategory structures satisfying (B0-5). Under this bijection: in the bicategory structure corresponding to an emd-noetherian form, right morphisms are quotients and left morphisms are embeddings for the noetherian form; for a given bicategory, the subform of the form of subquotients given by those subquotients $[e,m]$ where the pushout of $m$ along $e$ is an initial left morphism is a noetherian form corresponding to the bicategory structure.
\end{theorem}

The following technical theorem explores behaviour of initial and terminal objects in the context of an orean bicategory satisfying (B4). It will be used in the theorem that follows, which enables one to establish that the dual of any topos has an emd-noetherian form (see Example~\ref{ExaAE}).

\begin{theorem}\label{ThmAN}
If an orean bicategory has a terminal object $T$, then $T$ is a right trivial object. In this case, the following conditions are equivalent:
\begin{itemize}
    \item[(i)] (B4) holds.
    
    \item[(ii)] All terminal objects are left trivial.
    
    \item[(iii)] Every morphism whose codomain is a terminal object is a right morphism.
    
    \item[(iv)] Right trivial objects are the same as terminal objects.
\end{itemize}
When these conditions hold, terminal right morphisms with domain $X$ are precisely the morphisms whose codomain is a terminal object, and every morphism whose domain is a terminal object is an initial left morphism. 
If the bicategory also has an initial object $I$, then (B4) is further equivalent for the unique morphism $I\to T$ to be a right morphism. When there is an initial object (irrespective of whether there is a terminal object or not and whether (B4) holds), an initial left morphism $m\colon X\to Y$ can be obtained by factorizing the morphism $i\colon I\to Y$ as $i=me$, where $e$ is a right morphism, $m$ is a left morphism, and $I$ is an initial object.
\end{theorem}

\begin{proof}
Every morphism $f$ going out from a terminal object is a split monomorphism. So, as soon as $f$ is a right morphism (and hence an epimorphism by (B0)), it will be an isomorphism. We get that every terminal object is right trivial. The implication (i)$\implies$(ii) is a trivial consequence of this fact. Consider a morphism $f\colon T$ whose codomain is a terminal object. Decompose it as $f=me$ where $e$ is a right morphism and $m$ is a left morphism. If (ii) holds, then $m$ is an isomorphism and hence $f$ is a right morphism. This shows (ii)$\implies$(iii). Next, we show (iii)$\Rightarrow$(iv). Suppose (iii) holds. Consider a right trivial object $X$. The unique morphism $f\colon X\to T$ to the terminal object must be a right morphism by (iii). Since $X$ is right trivial, $f$ is an isomorphism and so $X$ is terminal. Conversely, a terminal object is always right trivial. Next, we prove (iv)$\implies$(i). Let $X$ be a right trivial object. Consider a left morphism $m\colon W\to X$ and any right morphism $e\colon W\to Y$. Then $m=fe$ for some $f$, when (iii) holds. This implies that $e$ is a left morphism and hence an isomorphism. So $W$ is right trivial. By (iii), $W$ must also be a terminal object and hence $m$ must be an isomorphism. So $X$ is left trivial.

Codomains of terminal right morphisms are the right trivial objects, by definition. It then follows easily that when (ii) and (iv) hold, terminal right morphisms with domain $X$ are precisely the morphisms whose codomain is a terminal object. To show that under (B4), any morphism whose domain is a terminal object is an initial left morphism, consider a morphism $m\colon T\to X$. Firstly, note that since $m$ is a split mono, it is a left morphism. The initial left morphism with codomain $X$ must than factor through $m$, via some left morphism. This left morphism must be an isomorphism by (ii). So, $m$ itself is an initial left morphism.

Suppose now $I$ is an initial object. The condition that the unique morphism $I\to T$ is a right morphism is easily equivalent to (iii). 

The last statement of the lemma can be easily established by properties of left and right morphisms in an orean bicategory. 
\end{proof}

\begin{theorem}\label{ThmAG}
Consider balanced orean bicategories $\mathbb{C}$ and $\mathbb{C}'$, the first of which has finite colimits, terminal object, and pullbacks. If there is a faithful functor $T\colon \mathbb{C}\to\mathbb{C}'$ which preserves finite colimits, terminal object, monomorphisms, as well as pullbacks of epimorphisms along arbitrary morphisms and $\mathbb{C}'$ has an emd-noetherian form then $\mathbb{C}$ has an emd-noetherian form.
\end{theorem}

\begin{proof}
Preservation of finite colimits implies preservation of epimorphisms, thanks to a well known and easily provable fact that a morphism $e\colon X\to Y$ is an epimorphisms if and only if the square
\[\xymatrix{X\ar[r]^-{e}\ar[d]_-{e} & Y\ar[d]^-{1_Y}\\ Y\ar[r]_-{1_Y} & Y}\]
is a pushout. Moreover, the same fact can be used to show that $T$ reflects epimorphisms, due to the fact that $T$ is also faithful. Next, we prove that $T$ reflects isomorphisms. Consider a morphism $f$ such that $T(f)$ is an isomorphism. Decompose $f$ as $f=me$, where $e$ is an epimorphism and $m$ is a monomorphism. Then $T(f)$ is an isomorphism, and since $T$ preserves both monomorphisms and epimorphisms, both $T(m)$ and $T(e)$ are isomorphisms. In particular, $T(m)$ is an epimorphism. Hence, $m$ is an epimorphism, and consequently, an isomorphism. Since $T$ must preserve a pullback of $e$ along itself ($T$ preserves pullbacks of epimorphisms along arbitrary morphisms), by faithfulness of $T$ and the fact that $T(e)$ is a monomorphism, we get that $e$ is a monomorphism (using an argument dual to the one that shows that $T$ reflects epimorphisms). Since $e$ is both an epimorphism and a monomorphism, it must be an isomorphism. Therefore, $f=me$ is an isomorphism. This completes the proof that $T$ reflects isomorphisms. It then follows that $T$ reflects monomorphisms, epimorphisms, terminal objects as well as all finite colimits, and pullbacks of epimorphisms. Thanks to Theorem~\ref{ThmAN}, this furthermore gives that $T$ preserves and reflects initial left morphisms and terminal right morphisms. Thanks to Theorem~\ref{ThmAN} again, if (B4) holds for $\mathbb{C}'$, then it holds for $\mathbb{C}$ as well. The preservation/reflection properties of $T$ ensure that the same is true for each of (B0-5). So the desired result follows by an application of Theorem~\ref{ThmAD}.    
\end{proof}

\begin{example}\label{ExaAE}
We know from Example~\ref{ExaQ} that the category of sets is a balanced orean bicategory having a noetherian form admitting exact join decomposition. It is well known that just like sets, an elementary topos is a balanced orean bicategory. It turns out that, furthermore, any topos has a noetherian form admitting exact join decomposition. For a small topos, this is immediate from the dual of Theorem~\ref{ThmAG} above and Theorem~3.24 in \cite{Fre72}, according to which, every small topos has a faithful functor $T$ to a power of the category of sets, which preserves finite limits, coproducts, epimorphisms and pushouts of monomorphisms. An intrinsic proof that for a (not necessarily small) topos the duals of (B0-5) hold is left as an exercise. Note that (B0), (B2$'$), and the first part of (B1$'$) are well known properties of the dual of a topos. Also, in light of Theorem~\ref{ThmAN}, the dual of (B4) follows from a well known property of a topos that any morphism into an initial object is an isomorphism.  
\end{example}

\begin{remark}\label{RemX}
In view of the intimate link between noetherian forms and semi-abelian categories, and the fact that a semi-abelian category is the same as a pointed Barr exact protomodular category having binary sums, one may wonder whether the dual of a topos has a noetherian form admitting exact meet decomposition because so does, in general, the bicategory of a Barr exact protomodular category (having binary sums) where left morphisms are monomorphisms and right morphisms are regular epimorphisms (recall from \cite{Bou04} that the dual of a topos is such bicategory). By Theorem~\ref{ThmAD}, this is not the case, since the bicategory of rings with identity does not satisfy (B3). For instance, in the following instance of the diagram from (B3),
\[\xymatrix{  & \mathbb{Z}\ar@{->>}[d]\ar@{>->}[r]^-{\subseteq} & \mathbb{Q}\ar@{->>}[d] \\ \mathbb{Z}_2\ar@{>->}[r]|-{*} & \mathbb{Z}_2\ar[r] & 0 }\]
the square is a pushout, while the composite of bottom two morphisms is not a monomorphism. 
\end{remark}

\subsection*{Left exact bicategories} A \emph{left/right exact bicategory} is an orean bicategory such that the corresponding orean factorization is a left/right exact pair. The dual of a left exact bicategory is a right exact bicategory and vice versa.

\begin{theorem}\label{ThmAP} A bicategory is left exact if and only if it is an orean bicategory satisfying (B2) (or equivalently, (B2$'$)) and the dual of (B4). Furthermore, in a left exact bicategory having an initial object, a morphism is a right morphism if and only if it is a regular epimorphism, and it is a left morphism if and only if it is a monomorphism.
\end{theorem}

\begin{proof} Let $(F_\s,F_\e)$ be an orean factorization corresponding to an orean bicategory.
That $(F_\s,F_\e)$ is a semiexact pair we have by Theorem~\ref{ThmF}. Let $\alpha$ denote the conormal operator $\alpha\colon F_\s\to F_\e$ and let $\beta$ denote the normal operator $\beta\colon F_\e\to F_\s$. That $\alpha$ is a left inverse of $\beta$ translates to the following condition on the bicategory:
\begin{itemize}
    \item[(B2$^*$)] For each pullback     \[\xymatrix{ \bullet\ar@{->>}[d]_-{s}\ar@{>->}[r]^-{m} & \bullet\ar@{->>}[d]^-{e} \\ \bullet\ar@{>->}[r]|-{*}_{i} & \bullet}\] 
    there is a pushout
 \[\xymatrix{ \bullet\ar@{->>}[d]|-{*}_-{t}\ar@{>->}[r]^-{m} & \bullet\ar@{->>}[d]^-{e} \\ \bullet\ar[r] & \bullet}\]    
\end{itemize}
In the case when $e$ is an identity morphism, (B2$^*$) states that for a given initial left morphism $i$ with domain $X$, the terminal right morphism $t$ (an hence any right morphism) with domain $X$ is an isomorphism. So, it states every left trivial object is right trivial, i.e., the dual of (B4). On the other hand, it is not difficult to see that under the dual of (B4), any right morphism whose codomain is a left trivial object, must be a terminal right morphism. So, under the dual of (B4), the right morphism $s$ in (B2$^*$) is also terminal. With this, it is not hard to see that (B2$^*$) becomes equivalent to (B2).

It is not difficult to see that regular epimorphisms are right morphisms in any orean bicategory.
Suppose now both (B2) and the dual of (B4) hold, and an initial object exists. Thanks to the dual of Theorem~\ref{ThmAN}, the domain of every initial left morphism is an initial object. Consider a right morphism $e\colon X\to Y$. By (B2), the pullback     
\[\xymatrix{ \bullet\ar[d]_-{s}\ar@{>->}[r]^-{m} & X\ar@{->>}[d]^-{e} \\ I\ar@{>->}[r]|-{*}_{i} & Y}\]
is a pushout. Since $I$ is an initial object, there is a morphism $j\colon I\to X$ such that $ej=i$. It is then easy to see that $e$ is a coequalizer of $m$ and $js$.

The coincidence of the class of right morphisms with the class of regular epimorphisms implies coincidence of the class of left morphisms with the class of monomorphisms.
\end{proof}

\begin{theorem}\label{ThmAQ}
  There is a bijection between isomorphism classes of orean forms $F$ over a category $\mathbb{C}$, having left exact decomposition and satisfying (N2), and left exact bicategory structures on $\mathbb{C}$. Under this bijection, in the bicategory structure corresponding to $F$, right morphisms are $F$-quotients and left morphisms are $F$-embeddings. For a given left exact bicategory, the corresponding form $F$ is the form of subobjects of the bicategory. Finally, a left exact bicategory satisfies (B1) and (B5) if and only if the corresponding form $F$ is a noetherian form.
\end{theorem}

\begin{proof} 
To establish the stated bijection, it suffices to prove the following:
\begin{itemize}
    \item[(a)] For any orean form $F$ having left exact decomposition and satisfying (N2), the bicategory where left morphisms are $F$-embeddings and right morphisms are $F$-quotients, is a left exact bicategory.
    
    \item[(b)] In (a), $F$ is isomorphic to the form of subobjects of the bicategory.
    
    \item[(c)] For a left exact bicategory, the form $F$ of subobjects satisfies (N2) and has left exact decomposition.
    
    \item[(d)] In (c), left morphisms are the same as $F$-embeddings and right morphisms are the same as $F$-quotients.
    
    \item[(e)] In (c), the bicategory satisfies (B1) if and only if $F$ satisfies (N1).
    
    \item[(f)] In (c), if the bicategory satisfies (B1), then it satisfies (B5) if and only if $F$ satisfies (N3).
\end{itemize}

According to Theorem~\ref{ThmAI}, orean forms having left exact decomposition are the same as conormal strongly orean forms $F$ such that the subform inclusion $F_\n\to F$ has a conormal left inverse. Consider such form $F$ over a category $\mathbb{C}$, which furthermore satisfies (N2). By the dual of Lemma~\ref{LemM}, the form $F_\n$ is normal and the subform inclusion $F_\n\to F$ is normal. So the pair $(F,F_\n)$ is right exact. At the same time, by Theorem~\ref{ThmT}, the pair $(F,F_\n)$ is an orean factorization. Moreover, the class $\mathcal{E}$ of $F$-quotients is the same as the class of $F_\n$-quotients. Via Theorem~\ref{ThmV} we then get that the bicategory $(\mathbb{C},\mathcal{E},\mathcal{M})$, where $\mathcal{M}$ denotes the class of $F$-embeddings, satisfies (B0). We thus get a left exact bicategory, which proves (a). Moreover, since $F=F_\c$, the form $F$ is isomorphic to the form $F_\s$ of subobjects of this bicategory by Remark~\ref{RemJ}, thus proving (b). In fact, furthermore, $F_\n$ is isomorphic to the form $F_\e$ of right quotients of the bicategory, by the dual of Remark~\ref{RemJ}.

Let us now start with a left exact bicategory.
Let $F$ denote the form of subobjects and let $F_\e$ denote the form of quotients of the bicategory. By Theorem~\ref{ThmV}, $(F,F_\e)$ is an orean factorization, which is a left exact pair (by the definition of a left exact bicategory), and furthermore, left morphisms are precisely the $F$-embeddings, while right morphisms are precisely the $F_\e$-quotients. It then follows from Lemma~\ref{LemAH} that $F$ has left exact decomposition, $F_\e$ is isomorphic to $F_\n$, and by applying Theorem~\ref{ThmI}, that $F$ satisfies (N2). This proves (c), as well as (d). 

We have (e) by the discussion at the start of this section. 

Consider a left exact bicategory satisfying (B1) and let $F$ be as in (c). By (c) and (e), $F$ satisfies (N1) and (N2). Then, by Remark~\ref{RemD}, $F$ satisfies (N3) if and only if direct images of normal $F$-clusters along right morphisms are normal. By what we have established when proving (c) and (d) above, and the discussion at the start of this section and the dual of Theorem~\ref{ThmAF}, we know that (B5) is equivalent to the subform inclusion $F_\n\to F$ preserving direct images along right morphisms. By the dual of Lemma~\ref{LemM}, this becomes equivalent to requiring that for any right morphism $e\colon X\to Y$ and a normal cluster $N$ in $X$, we have $e\cdot^\s N=\ncl{e\cdot^\s N}$. But this is exactly the same as to say that direct images of normal $F$-clusters along right morphisms are normal. This proves (f).
\end{proof}

\begin{theorem}\label{ThmAT}
The bijection of Theorem~\ref{ThmAQ} restricts to a bijection between isomorphism classes of noetherian forms $F$ over a category $\mathbb{C}$, having left exact join decomposition, and left exact bicategory structures on $\mathbb{C}$ satisfying the duals of (B2), (B3), and the dual of the following condition: 
\begin{itemize}
\item[(B1a)] pullback of a right morphism along a left morphism is always a right morphism. 
\end{itemize}
Such bicategory structures are the same as bicategory structures satisfying (B0), the dual of (B1a), (B2$'$) and its dual, and the duals of (B3) and (B4). Furthermore, such bicategory structures satisfy duals of (B1) and (B5), and they also satisfy (B3). Finally, a category having an initial object has at most one such bicategory structure and when it exists, left morphisms are monomorphisms and right morphisms are regular epimorphisms, while existence of such bicategory structure is equivalent to every morphism with initial object as the domain being a monomorphism and (B0), (B2$'$) and its dual, and duals of (B1a) and (B3) being true for monomorphisms as left morphisms and regular epimorphisms as right morphisms.
\end{theorem}

\begin{proof} To establish the first statement of the theorem, we must prove that a left exact bicategory satisfies the duals of (B1a), (B2) and (B3), if and only if the form $F$ of its subobjects is a noetherian form having left exact join decomposition. Applying the dual of Theorem~\ref{ThmAD}, we only need to prove that a left exact bicategory satisfying duals of (B1a), (B2) and (B3), also satisfies duals of (B1), (B4) and (B5). By Theorem~\ref{ThmAP}, left exactness is equivalent to (B2) and the dual of (B4). It then suffices to show that (B2$'$) and its dual, along with (B1a), imply duals of (B1) and (B5$'$), which can be established by easy arguments using basic properties of pullbacks and pushouts (and noting that the first part of (B1) is just (B1a)). This completes the proof of the first statement of the theorem.

The second statement of the theorem is by Theorem~\ref{ThmAP}. The first part of the third statement has already been established. That (B3) holds is an easy consequence of the dual of (B4). For the last statement, apply the second statement of the current theorem, the second part of Theorem~\ref{ThmAP} and the dual of Theorem~\ref{ThmAN}.
\end{proof}

\begin{theorem}\label{ThmAU}
The bijection of Theorem~\ref{ThmAQ} restricts to a bijection between isomorphism classes of noetherian forms $F$ over a category $\mathbb{C}$, having left exact meet decomposition, and left exact bicategory structures on $\mathbb{C}$ satisfying (B1), (B2), (B4) and its dual, and (B5). A category having initial or terminal object has at most one such bicategory structure given by monomorphisms as left morphisms and regular epimorphisms as right morphisms, and such structure exists if and only if the category is pointed, and (B0), (B1), (B2) and (B5) hold with monomorphisms as left morphisms and regular epimorphisms as right morphisms, in which case the conditions (B1), (B2) and (B5) can be equivalently reformulated as follows (where arrows with tail represent monomorphisms and arrows with a double head represent regular epimorphisms):
\begin{itemize}
    \item[(B1*)] Pullback of a regular epimorphism along a monomorphism is always a regular epimorphism and a commutative diagram
    \[\xymatrix{ \bullet\ar@{>->}[d]_-{m}\ar@{->>}[r] & \bullet\ar@{>->}[d] \\ \bullet\ar@{->>}[r]_{e} & \bullet}\] is a pullback as soon as the kernel of $e$ factors through $m$.
    
    \item[(B2*)] Every regular epimorphism is a cokernel of its kernel.
    
    \item[(B5*)] Given a commutative diagram 
     \[\xymatrix{ \bullet\ar@{>->}[d]_-{m}\ar@{->>}[r] & \bullet\ar@{>->}[d]^-{m'} \\ \bullet\ar@{->>}[r] & \bullet}\] the morphism $m'$ is a kernel provided $m$ is.
\end{itemize}
\end{theorem}

\begin{proof}
To get the first statement of the theorem, apply Theorems~\ref{ThmAD}, \ref{ThmAP} and \ref{ThmAQ} along with a simple observation that the dual of (B4) implies (B3) (as noted already in the proof of Theorem~\ref{ThmAT}). If either a terminal or initial object exists, then an easy argument shows that (B4) with its dual imply that the category is pointed, so the first part of the second statement of the theorem follows from Theorem~\ref{ThmAP}. The `only if' part of the second part of the second statement is immediate, while the `if' part follows from the fact that in a pointed category, objects which do not have proper subobjects, as well as objects which do not have proper regular quotients, are precisely the zero objects. That in a pointed orean bicategory where left morphisms are monomorphisms and right morphisms are regular epimorphisms, (B1), (B2) and (B5) reformulate to (B1*), (B2*) and (B5*), respectively, can be easily verified. 
\end{proof}

\begin{remark}
The last part of Theorem~\ref{ThmAU} reveals an intimate link between the axioms (B0-5) and the ``old-style'' axioms for a semi-abelian category from \cite{JMT02}. As we will see in Theorem~\ref{ThmAM}, the concept analyzed in Theorem~\ref{ThmAU} becomes that of a semi-abelian category once the existences of finite products and sums is required.
\end{remark}

\subsection*{Semi-abelian and abelian categories} Here we point out various characterizations of semi-abelian and abelian categories based on properties of forms investigated in this paper.

By Theorem~\ref{ThmAK} and Remark~\ref{RemT}, every noetherian form $F$ having exact decomposition is strongly orean and has unique exact decomposition, where $F_\s=F_\c$ and $F_\e=F_\n$. This decomposition is:
\begin{itemize}
    \item a left decomposition, if and only if $F$ is conormal (according to Theorem~\ref{ThmAH}) --- conormal noetherian forms having exact decomposition are precisely the strongly orean conormal noetherian forms according to Theorem~\ref{ThmAH};
    
    \item a right decomposition, if and only if $F$ is normal (by the dual of the above) --- normal noetherian forms having exact decomposition are precisely the strongly orean normal noetherian forms.
\end{itemize}
As a result of Theorems~\ref{ThmAL}, \ref{ThmQ} and \ref{ThmP}, we then have:
\begin{itemize}
    \item a form is a conormal noetherian form having exact join decomposition if and only if it is a strongly orean conormal noetherian form in which normal clusters are stable under direct images; dually, a form is a normal emd-noetherian form if and only if it is a strongly orean normal noetherian form in which conormal clusters are stable under inverse images;
    
    \item a form is a conormal emd-noetherian form if and only if it is a conormal noetherian form having normal exteriors; dually, a form is a normal noetherian form having exact join decomposition if and only if it is a normal noetherian form having conormal interiors; 
    
    \item a noetherian form has dual properties of having exact meet decomposition and exact join decomposition if and only if it is binormal.
\end{itemize}
We will use some of these observations in what follows.

\begin{theorem}\label{ThmAM}
A category $\mathbb{C}$ having finite products and sums has, up to isomorphism, at most one conormal emd-noetherian form, and the following conditions are equivalent:
\begin{itemize}
    \item[(i)] $\mathbb{C}$ has a conormal emd-noetherian form.
    
    \item[(ii)] $\mathbb{C}$ is a pointed category having an emd-noetherian form.
    
    \item[(iii)] $\mathbb{C}$ is a pointed category having a conormal noetherian form having exact decomposition.

    \item[(iv)] $\mathbb{C}$ is a pointed category having a conormal noetherian form.
    
    \item[(v)] $\mathbb{C}$ is a pointed category having a noetherian form such that inverse images of conormal clusters are conormal and bottom clusters are conormal.
    
    \item[(vi)] $\mathbb{C}$ is a semi-abelian category in the sense of \cite{JMT02}. 
\end{itemize}
\end{theorem}

\begin{proof} Uniqueness of a conormal emd-noetherian form follows from Theorems~\ref{ThmAP} and \ref{ThmAQ}.
(i)$\implies$(ii) and (i)$\implies$(iii): Suppose $\mathbb{C}$ has a conormal emd-noetherian form. To prove (ii) as well as (iii), it suffices to show that $\mathbb{C}$ is a pointed category. Let $I$ and $T$ denote, respectively, an initial and a terminal object of $\mathbb{C}$. Consider the unique morphism $f\colon I\to T$. We will show that $f$ is an isomorphism. Decompose $f$ as $f=\iota_{\ker f}\circ \pi_{\im f}$. In view of Lemma~\ref{LemV}, to show that $f$ is an isomorphism, it is sufficient (actually, equivalent) to show that $\ker f$ is a bottom cluster in $I$ and $\im f$ is a top cluster in $T$. Since $I$ is an initial object, the embedding $\iota_{\bot_I}$ (which exists, since the form is conormal) is an isomorphism. Then by Lemma~\ref{LemV} again, $\bot_I=\top_I$. Consequently, $\ker f=\bot_I$. Dual argument works in $T$, where $\top_T$ must be normal in view of the fact that in a conormal emd-noetherian form clusters have normal exteriors. 

(iii)$\implies$(iv)$\implies$(v) are trivial.

(ii)$\implies$(v): Suppose $\mathbb{C}$ is a pointed category having an emd-noetherian form $F$. Then $F$ is a noetherian form in which inverse images of conormal clusters are conormal, by the dual of Theorem~\ref{ThmM}. Furthermore, since normal exteriors exist, the top clusters are normal. In particular, the top cluster of the zero object $Z$ is normal. This implies that $\top_Z=\bot_Z$, which can be established by a similar argument to the one needed in the proof of (i)$\implies$(ii) and (i)$\implies$(iii) above. It is then easy to see why the bottom cluster in every object will be conormal.

(v)$\iff$(vi) is already known: it is Theorem~2.17 in \cite{vanniekerk19}.

For (vi)$\implies$(i) we will show that if $\mathbb{C}$ is a semi-abelian category, then its form $F$ of subobjects is a conormal emd-noetherian form. That the form $F$ of subobjects of a semi-abelian category is a conormal noetherian form is already known from \cite{JW16b,J14}. This leaves us to show that $F$ has normal exteriors. It is not hard to show that the normal exterior of a subobject represented by a monomorphism $m$ will be the subobject represented by the usual categorical kernel of the cokernel of $m$. 
\end{proof}

\begin{theorem}
For a category $\mathbb{C}$ having finite products and sums, the following conditions are equivalent:
\begin{itemize}
    \item[(i)] $\mathbb{C}$ has a binormal noetherian form.
    
    \item[(ii)] $\mathbb{C}$ has a binormal noetherian form having exact decomposition.
    
    \item[(iii)] $\mathbb{C}$ has an emd-noetherian form having exact join decomposition
    
    \item[(iv)] $\mathbb{C}$ is a pointed category having normal emd-noetherian form.
    
    \item[(v)] $\mathbb{C}$ is an abelian category.
\end{itemize}
\end{theorem}

\begin{proof}
(i)$\implies$(ii) and (ii)$\implies$(iii) are by Theorem~\ref{ThmP}. (iii)$\implies$(iv) is by Theorems~\ref{ThmP} and \ref{ThmAM}. (iv)$\implies$(v) is by Theorem~\ref{ThmAO} and the well known fact that a Puppe-Mitchell exact category having finite products and sums is the same as an abelian category (see, e.g., \cite{P62,M65,G12}). (v)$\implies$(i) is by the same fact along with Remark~\ref{RemF}. 
\end{proof}

\subsection*{Optimal noetherian forms} The use of a noetherian form over a bicategory $\mathbb{C}$ is that it allows one to prove homomorphism theorems in $\mathbb{C}$, whose formulation only requires the bicategory structure of $\mathbb{C}$. Thus, two noetherian forms over the same bicategory will yield same homomorphism theorems. This motivates the question of finding the most `optimal' noetherian form over a bicategory.   

\begin{definition}
A noetherian form over a bicategory is said to be \emph{optimal} when it is isomorphic to a subform of any other noetherian form over the same bicategory.
\end{definition} 

\begin{theorem}\label{ThmAS}
Every noetherian form $G$ over a bicategory having an exact join decomposition is optimal. Dually, every emd-noetherian form is optimal.
\end{theorem}

\begin{proof}[Proof (sketch).] 
Let $F$ and $G$ be noetherian forms over a bicategory, and so left morphisms are the same as $F$-embeddings as well as $G$-embeddings, and right morphisms are the same as $F$-quotients as well as $G$-quotients. $G_\c$ and $F_\c$ are isomorphic to the form of $\mathcal{M}$-subobjects and $G_\n$ and $F_\n$ are isomorphic to the form of $\mathcal{E}$-quotients, by Remark~\ref{RemJ} and its dual. Suppose furthermore that $G$ has an exact join decomposition. Then by Theorem~\ref{ThmH}, $(F_\c,F_\n)$ is an orean factorization and $G$ is isomorphic to a subform of $F_\c\times F_\n$ consisting of those clusters $(A,R)$ where $A\ast R=A$ and $\alpha(A)\le R$. Let $\delta$ denote the corresponding full operator $G\to F_\c\times F_\n$. Consider the composite
\[\xymatrix{ G\ar[r]^-{\delta} & F_\c\times F_\n\ar[r] & F\times F\ar[r]^-{\lor} & F}\]
where the middle operator is a subform inclusion and the one on the right comes from Example~\ref{ExaO}. By Lemma~\ref{inverseimageinteriorlem}, we obtain a full operator $G\to F$. By Lemma~\ref{LemY}, we get that $G$ is isomorphic to a subform of $F$.
\end{proof}

\begin{example}
It is not difficult to construct an example of a bicategory $\mathbb{C}$ and an infinite sequence
\[\xymatrix{ \dots\ar[r]^-{\tau_{i-1}} & F_i\ar[r]^-{\tau_{i}} & F_{i+1}\ar[r]^-{\tau_{i+1}} & \dots }\]
of subform inclusions, where each $F_i$ is an optimal noetherian form over $\mathbb{C}$ and no $\tau_i$ is an identity operator (so all subform inclusions are proper). In fact, $\mathbb{C}$ will have all finite limits and each $F_i$ will have a left exact join decomposition. Set $\mathbb{C}$ to be the linearly ordered set $[0,1]$ seen as a category, with the bicategory structure given as follows: all morphisms are left morphisms and only identity morphisms are the right morphisms. Consider the following form $F$ over $\mathbb{C}$. For each $r\in [0,1]$, clusters in $r$ are all non-negative real numbers, while for a morphism $f\colon r\to s$ and clusters $a$ in $r$ and $b$ in $s$, 
\[b\geqslant_f a\quad\Leftrightarrow\quad a\leqslant b,\]
where the `$\leqslant$' on the right hand side represents the usual inequality of real numbers. Now, for each integer $i\in\mathbb{Z}$, define $F_i$ to be the subform of $F$ where for each object $r$ the clusters are only those non-negative real numbers $a$ that satisfy $a\leqslant 2^ir$. We then obtain an infinite sequence of subform inclusions as shown above. Now, for each $i$, the mapping $a\mapsto 2^{-i}a$, clearly defines an isomorphism of forms $F_i\to F_0$. So to show that each $F_i$ is a noetherian having a left exact join decomposition (and hence is optimal, by Theorem~\ref{ThmAS}), it suffices to show that such is $F_0$. However, it is not difficult to see that $F_0$ is in fact isomorphic to the form of subobjects of our bicategory $\mathbb{C}$. So by Theorem~\ref{ThmAT}, it suffices to show that the bicategory $\mathbb{C}$ satisfies (B0), the dual of (B1a), (B2$'$) and its dual, and the dual of (B3). That each of these conditions hold can be easily confirmed.
\end{example}

\section{Conclusion}

Ordinary subobjects in a semi-abelian category serve as clusters for a noetherian form over it, while for Grandis exact categories, it is the `normal subobjects' defined relative to an ideal of null morphisms that serve as clusters for a noetherian form over its underlying category. As it is well known, the dual of the category of pointed sets is a semi-abelian category, while the category of sets and partial bijections is Grandis exact. So both of these categories admit a noetherian form. In the case of pointed sets, the noetherian form is the form of quotients (i.e., clusters are given by quotients, or equivalently, equivalence relations). The form of quotients is no longer noetherian for the category of sets (this deficiency of the category of sets was the backbone of research carried out in \cite{vZ17}). Neither is the form of subsets of sets noetherian. The work presented in this paper originated with an attempt to prove that the category of sets does not have a noetherian form at all. In the end, we proved the opposite: \emph{the category of sets does have a noetherian form}. In particular, we found that the form for which clusters in a set $X$ are given by pairs $(S,E)$, where $E$ is an equivalence relation on $X$ and $S$ is a set of equivalence classes (in other words, the form of subquotients in the dual of the category of sets), is noetherian. In fact, it turned out that a smaller form where $S$ is either empty or a singleton, is still a noetherian form. We then wanted to identify intrinsic properties of this particular noetherian form, in order to understand what distinguishes it form from the duals of noetherian forms of subobjects over semi-abelian categories, and specifically, the noetherian form of quotients of the category of pointed sets, which seems to be the obvious pointed version of our noetherian form over the category of sets (the base point forces $S$ to reduce to the singleton consisting of the equivalence class of the base point). The present paper reports on the results obtained in this investigation, which can be summarized as follows:
\begin{itemize}
  \item not only is it possible to identify a reasonable list of intrinsic properties of the noetherian form above, whose duals are shared with the noetherian form of subobjects of any semi-abelian category (and hence these properties are shared with the noetherian form of quotients of the category of pointed sets),
    
  \item but moreover, up to an isomorphism, these intrinsic properties determine a unique noetherian form over the category of sets (the one described above), as well as the category of pointed sets, in which case we recover the noetherian form of quotients;

  \item and finally, semi-abelian categories are precisely those pointed categories having finite products and coproducts that admit a noetherian form having properties dual to those that we identified. 
\end{itemize}
Note that the last two points imply that we have a way of thinking of any semi-abelian category as well as the dual of the category of sets as the same type of a category, but this is certainly not the first such unification: for instance, both of these categories are Barr exact \cite{Bar71} and protomodular \cite{Bou91}.

The intrinsic properties of the special noetherian form of the category of sets that we found are gathered together under the concept of an \emph{exact join decomposition} of a form introduced and studied in Section~\ref{SecG} of this paper. The sections preceding that section give a reasonably self-contained account of a preliminary material on forms, most of which is needed for formulating and establishing the new results of this paper. Although this material can be considered largely known from \cite{J14,JW14,JW16,JW16b,DAGJ17,GJ19,vanniekerk19}, there are also some new ideas and new auxiliary concepts introduced there, such as the discussion of analogy between forms and monoidal category structures in Section~\ref{secA} and the notion of an \emph{orean factorization} in Section~\ref{secH}, which is conceptually the same as a proper factorization system $(\mathcal{E},\mathcal{M})$ \cite{I58,FK72} admitting pullbacks of morphisms in $\mathcal{M}$, pushouts of morphisms in $\mathcal{E}$, as well as finite meets and joins of $\mathcal{M}$-subobjects and $\mathcal{E}$-quotients. In an orean factorization, such a structure $(\mathcal{E},\mathcal{M})$ is seen in terms of a pair $(F_\e,F_\s)$ of forms, where $F_\e$ is the form of $\mathcal{E}$-quotients and $F_\s$ is the form of $\mathcal{M}$-subobjects. The relevance of factorization systems in the study of noetherian forms is that they give the data of epimorphisms and monomorphisms that would be involved in the statement of homomorphism theorems that can be established with the use of a noetherian form. The corresponding forms $F_\e$ and $F_\s$ are then the `subforms' of the original form consisting of clusters that are, respectively, images and kernels (as defined at the start of the Introduction) relative to the form.

Let us interlude with a note that we tried our best to keep the paper as self-contained as possible, and in light of that, knowledge of factorization systems is not essential in how orean factorizations are used later on in the paper. However, when clarifying the link between orean factorizations and factorization systems (Theorem~\ref{ThmV}) in Section~\ref{secH} and in related remarks in the same section, we rely on a fairly good insight to factorization systems (that can be obtained from \cite{CJKP97}, \cite{FK72}) as well as on the results linking them with forms from \cite{JW14}. In Section~\ref{SecK}, where we end up working in a category equipped with a (proper) factorization system, we do provide a recollection of some known properties of a (proper) factorization system.

The culminating point in the paper is Theorem~\ref{ThmH} in Section~\ref{SecG} (see also Theorem~\ref{ThmAC}), which gives a characterization of noetherian forms admitting exact join decomposition in terms of the underlying orean factorization $(F_\s,F_\e)$. The same theorem states that a given orean factorization arises from at most one (up to an isomorphism) noetherian form having exact join decomposition. The second bullet-point above is a simple consequence of this theorem (see Example~\ref{ExaQ}). The third bullet-point, on the other hand, is a corollary of Theorem~2.17 in \cite{vanniekerk19}, as explained in Remark~\ref{RemG} of the present paper (see also Theorem~\ref{ThmAM}). The proof of Theorem~\ref{ThmH}, which is not short, would have been longer had we not extracted some of its pieces that have been embedded in the preparatory material built up gradually throughout the paper. Work on this theorem was a productive back and forth exchange between the authors. The final improvement to the theorem was made when the paper was nearly finished; for a couple of reasons, we decided to formulate this last improvement as a separate theorem (Theorem~\ref{ThmAC}). One of these reasons is that the methodology for arriving to this second theorem, Theorem~\ref{ThmAC}, was slightly different to the methodology of arriving to Theorem~\ref{ThmH}: in the first case, we tried to stick to the form-theoretic thinking as much as possible; the simplification obtained in Theorem~\ref{ThmH}, which is still form-theoretic in its formulation, was reached after trying to illustrate the form-theoretic conditions from Theorem~\ref{ThmAC} in the language of factorization systems.

In Section~\ref{SecK}, we give a complete illustration of the form-theoretic characterization of noetherian forms admitting exact join decomposition, established in Theorem~\ref{ThmH} (and refined in Theorem~\ref{ThmAC}), in the language of a category equipped with a proper factorization system. Our immediate intension in this section is to arrive to the conclusion that not only the category of sets, but any elementary topos has a noetherian form admitting exact join decomposition (see Theorem~\ref{ThmAG} and Example~\ref{ExaAE}). The same section should be helpful to compare the properties of a noetherian form admitting exact join decomposition with (duals of) non-pointed generalizations of exactness properties of semi-abelian categories, which we leave for a future work. We do note, however, that while the usual epi-regular mono factorization systems of toposes and duals of semi-abelian categories both admit noetherian forms having exact join decomposition, the same is not true for the dual of a Barr exact protomodular category, in general. This, despite the fact that the dual of any topos is a Barr exact protomodular category (see \cite{Bou04}) and pointed Barr exact protomodular categories having binary coproducts are the same as semi-abelian categories (see \cite{JMT02}). A counterexample is given by the category of rings with identity (see Remark~\ref{RemX}). This gives a partial negative answer to a question posed to us by James Gray (private communication): the noetherian form that we have identified for a topos cannot be reconstructed in every (bicomplete) Barr exact protomodular category using its regular epi-mono factorization system. His question asks whether in general a Barr exact protomodular category always has a noetherian form. We have not touched on this question in the present paper (the question remains open).

In Section~\ref{SecK}, we also include characterizations of semi-abelian and abelian categories in terms of various properties of noetherian forms that were introduced and studied in this paper as a result of analysing the concept of an exact join decomposition of a noetherian form. Section~\ref{SecK} concludes with an unexpected finding. As an easy application of Theorem~\ref{ThmH}, we show that noetherian forms admitting an exact join decomposition are optimal, in the following sense: a noetherian form $G$ having join exact decomposition is isomorphic to a subform of every other noetherian form $F$ giving rise to the same factorization system $(\mathcal{E},\mathcal{M})$ as $G$ (see Definition~\ref{DefF} for the definition of a subform of a form). 

The use of noetherian forms in deriving homomorphism theorems in a given category is similar to the use of fibrations in topology, where one extends a space with a morphism into it, in order to establish some properties of that space. A noetherian form extends a category with a functor into it, in order to establish homomorphism theorems in the category. This means that the concept of an optimal noetherian form is essential to the use of noetherian forms in practice: producing an optimal noetherian form for a given proper factorization system $(\mathcal{E},\mathcal{M})$ provides as efficient as possible way of extending the category to derive homomorphism theorems stated in terms of the given class $\mathcal{E}$ of epimorphisms and the given class $\mathcal{M}$ of monomorphisms, from homomorphism theorems valid in a general noetherian form. 

In retrospect and in conclusion, we see the present paper as a first step in exploration of optimal noetherian forms. The first obvious question that we have not touched on in this paper is whether there are interesting/useful noetherian forms which are optimal, but which do not admit an exact join decomposition or its dual counterpart, an exact meet decomposition. There is also an interesting question for future exploration pertaining to noetherian forms that admit  exact join decomposition. In the recent work \cite{vn23} it is shown (see Theorem~5.5 there) that the given a category $\mathbb{X}$ having a noetherian form $F$ and a functor $R\colon\mathbb{A}\to\mathbb{X}$ having a left adjoint $L$, the pullback of $F$ along $R$ is a noetherian form over $\mathbb{A}$ if and only if $R$ is monadic and the monad $RL$ preserves pushouts of $\mathcal{E}$-morphisms, where $\mathcal{E}$ is from the factorization system $(\mathcal{E},\mathcal{M})$ associated with the form $F$. Now, the noetherian form over sets that we have identified is the pullback of the noetherian form of quotients of the category of pointed sets along the functor from the category of sets to the category of pointed sets which attaches to every set a base point. This functor is in fact comonadic, and the corresponding comonad preserves pullbacks of injective maps of pointed sets. Since injective maps of pointed sets form the class $\mathcal{M}$ from the factorization system $(\mathcal{E},\mathcal{M})$ 
associated with the noetherian form of quotients, we see that our noetheriam form over the category of sets fits the dual of the result above. So an interesting question for future investigation is: what is the relation between noetherian forms admitting exact meet decomposition and pullbacks of forms of subobjects in semi-abelian categories along monadic functors having the previously mentioned property? Clarifying links with recently introduced ideally exact categories \cite{J24} should certainly be part of this investigation, as ideally exact categories are also defined by means of monadic functors to semi-abelian categories.

\section*{Acknowledgements}

The first author is grateful to Davide Trotta for their discussions about doctrines.

\end{document}